\documentclass{article}
\usepackage{amssymb}
\usepackage{amsthm}
\usepackage{ifthen,amsfonts,graphicx,latexsym,algorithm,algorithmic,url}
\usepackage{amsmath,color}
\usepackage{verbatim}
\usepackage[latin1]{inputenc}
\usepackage[english]{babel}
\usepackage{esint}
\usepackage{bbm}

\usepackage{tcolorbox}
\usepackage{dsfont}
\usepackage{tikz}
\usepackage{enumitem}

\numberwithin{equation}{section}

\newtheorem{theorem}{Theorem}[section]
\newtheorem{lemma}[theorem]{Lemma}
\newtheorem{proposition}[theorem]{Proposition}
\newtheorem{corollary}[theorem]{Corollary}

\newtheorem{defi}[theorem]{Definition}
\theoremstyle{plain}

\theoremstyle{remark}

\def\RRn{\mathbb{R}^n}
\def\RR{\mathbb{R}}
\def\RRN{\mathbb{R}^N}
\def\ZZ{\mathbb{Z}}
\def\ZZN{\mathbb{Z}^N}
\def\RRnm{\mathbb{R}^{n \times m}}

\def\RRdn{\mathbb{R}^{d \times n}}

\def\LLp{\LL^p}
\def\LLinfty{\LL^{\infty}}

\def\dd{\, \mathrm{d}}

\def\grad{\nabla}
\def\weakConv{\rightharpoonup}
\def\W1p{\mathrm{W}^{1,p}}
\def\Wm1p{\mathrm{W}^{-1,p}}
\def\LL{\mathrm{L}}
\def\WW{\mathrm{W}}
\def\naturals{\mathbb{N}}

\def\LpS0{\LL^p_{S,0}}
\def\LpS{\LL^p_S}

\def\calM{\mathcal{M}}
\def\calE{\mathcal{E}_p}

\def\indyk{\mathds{1}}

\def\intOmega{\int_{\Omega}}
\def\intRRnm{\int_{\RRnm}}

\def\scale{\odot}

\def\Probability{\mathcal{P}}
\def\projection{\mathrm{P}_{\ba}}

\def\weaksConv{\overset{\ast}{\rightharpoonup}}

\def\LebesgueN{\mathcal{L}^N}

\def\hp{\mathbb{H}^p}
\def\hpn{\mathbb{H}^p_0}

\newcommand{\ggrad}{\ensuremath{\nabla_{\ba}}}
\newcommand{\fullgrad}{\widetilde{\ggrad}}
\newcommand{\lowgrad}{\widehat{\ggrad}}

\def\WWap{\mathrm{W}^{S,p}}
\def\WWapn{\mathrm{W}^{S,p}_0}
\def\Ccinfty{C_c^{\infty}}

\def\ba{\mathbf{a}}

\def\intRRN{\int_{\RRN}}
\def\QQ{\mathcal{Q}}

\def\overlineF{\overline{F}}
\def\translation{\delta}

\def\gm{\mathcal{G}_M}

\def\II{\mathrm{I}}
\def\ApgradConv{\to_{\ba\text{-}p}}

\DeclareMathOperator{\im}{Im}
\DeclareMathOperator{\id}{Id}
\renewcommand{\phi}{\varphi}
\renewcommand{\epsilon}{\varepsilon}
\renewcommand{\geq}{\geqslant}
\renewcommand{\leq}{\leqslant}

\def\WWS2{\mathrm{W}^{S,2}}
\def\WWS2n{\mathrm{W}^{S,2}_0}

\newcommand{\measurerestr}{%
  \,\raisebox{-.127ex}{\reflectbox{\rotatebox[origin=br]{-90}{$\lnot$}}}\,%
}

\def\Xint#1{\mathchoice
{\XXint\displaystyle\textstyle{#1}}%
{\XXint\textstyle\scriptstyle{#1}}%
{\XXint\scriptstyle\scriptscriptstyle{#1}}%
{\XXint\scriptscriptstyle\scriptscriptstyle{#1}}%
\!\int}
\def\XXint#1#2#3{{\setbox0=\hbox{$#1{#2#3}{\int}$ }
\vcenter{\hbox{$#2#3$ }}\kern-.6\wd0}}

\def\dashint{\Xint-}

\renewcommand{\phi}{\varphi}
\renewcommand{\epsilon}{\varepsilon}
\renewcommand{\geq}{\geqslant}
\renewcommand{\leq}{\leqslant}

\def\RRn{\mathbb{R}^n}
\def\RR{\mathbb{R}}
\def\RRN{\mathbb{R}^N}
\def\ZZ{\mathbb{Z}}
\def\ZZN{\mathbb{Z}^N}
\def\multi-indices{\mathbb{Z}^N_+}
\def\RRnm{\mathbb{R}^{n \times m}}

\def\naturals{\mathbb{N}}

\def\LL{\mathrm{L}}
\def\LLp{\LL^p}
\def\LLq{\LL^q}
\def\LLinfty{\LL^{\infty}}

\def\Ccinfty{C_c^{\infty}}

\def\ba{\mathbf{a}}
\def\bamo{\mathbf{a}^{-1}}

\def\WW{\mathrm{W}}
\def\WWap{\WW^{\ba,p}}
\def\WWapg{\WW^{\ba,p}_g}
\def\WWapn{\WW^{\ba,p}_0}
\def\WWaq{\WW^{\ba,q}}
\def\WWaqg{\WW^{\ba,q}_g}

\def\WWat{\WW^{\ba,2}}

\def\dd{\, \mathrm{d}}
\def\grad{\nabla}
\def\weakConv{\rightharpoonup}
\def\scale{\odot}

\def\intOmega{\int_{\Omega}}
\def\intRRnm{\int_{\RRnm}}
\def\intRRN{\int_{\RRN}}

\def\Qtaui{Q_{\tau,i}}

\def\Ptaui{P_{\tau,i}}
\def\etataui{\eta_{\tau,i}}

\def\Fun{\II}
\def\thetaqc{\theta^{\text{qc}}}

\def\og{\leavevmode\raise.3ex\hbox{$\scriptscriptstyle\langle\!\langle$~}}
\def\fg{\leavevmode\raise.3ex\hbox{~$\!\scriptscriptstyle\,\rangle\!\rangle$}}

\vspace{9mm}
\title{Existence of minimisers of variational problems posed in spaces of mixed smoothness}
\vspace{9mm}
\author{Adam Prosinski \thanks{Carnegie Mellon University {\tt aprosins@andrew.cmu.edu} }}
\vspace{9mm}

\begin{document}
\maketitle

\abstract{
The present work constitutes a first step towards establishing a systematic framework for treating variational problems that depend on a given input function through a mixture of its derivatives of different orders in different directions. For a fixed vector $\ba := (a_1, \ldots, a_N) \in \naturals^N$ and  $u \colon \RRN \supset \Omega \to \RRn$ we denote by $\ggrad u := (\partial^{\alpha} u)_{\langle \alpha, \bamo \rangle = 1}$ the matrix whose $i$-th row is composed of derivatives $\partial^\alpha u^i$ of the $i$-th component of the map $u$, and where the multi-indices $\alpha$ satisfy $\langle \alpha, \bamo \rangle = \sum_{j=1}^N \frac{\alpha_j}{a_j} = 1$. We study functionals of the form
$$ \WWap(\Omega;\RRn) \ni u \mapsto \int_\Omega F(\ggrad u(x)) \dd x,$$
where $\WWap(\Omega; \RRn)$ is an appropriate Sobolev space of mixed smoothness and $F$ is the integrand. We study existence of minimisers of such functionals under prescribed Dirichlet boundary conditions. We characterise coercivity, lower semicontiuity, and envelopes of relaxation of such functionals, in terms of an appropriate generalisation of Morrey's quasiconvexity. 
}


\section{Introduction}
A classical problem in the calculus of variations is asserting existence of minimisers of integral functionals of the form $\WW^{1,p}(\Omega, \RRn) \ni u \mapsto \int_\Omega F(\grad u) \dd x$. Usually, the key difficulty is ensuring that the functional is sequentially lower semicontinuous in the appropriate topology. Thanks to a long series of important contributions (see \cite{AcerbiFusco84}, \cite{BallCurrieOlver81}, \cite{Dacorogna82}, \cite{Kristensen99one}, \cite{Marcellini85}, \cite{Meyers65} among others) we know that this depends on quasiconvexity of the integrand $F$. In the vector-valued case ($n > 1$) quasiconvexity (introduced by Morrey in \cite{Morrey52} and studied in \cite{AlibertDacorogna92}, \cite{Ball78}, \cite{BallKircheimKristensen00}, \cite{BallMurat84}, \cite{Cagnetti11}, \cite{DacorognaMarcellini88}, \cite{DalMasoFonsecaLeoniMorini04}, \cite{Evans86}, \cite{KirchheimKristensen16}, \cite{Kristensen99two}, \cite{Muller99}, \cite{Sverak92}) 
is strictly weaker than ordinary convexity which is sufficient, but far from necessary, for lower semicontinuity when working with gradients of Sobolev functions. This disparity is due to the special structure of gradient vector fields, encompassed in the relation $\mathop{curl} \grad u = 0$. This phenomenon is more general than just gradients and, in the framework of Murat and Tartar's compensated compactness (\cite{Murat78}, \cite{Murat81} \cite{Tartar79}, \cite{Tartar83}, \cite{Tartar92}) has led to another family of variational results (see, for instance, \cite{ArroyoPhilippisRindler17}, \cite{BraidesFonsecaLeoni00}, \cite{FonsecaLeoniMuller04}, \cite{FonsecaMuller99}, \cite{Raita19}) for functionals acting on vector fields $v$ satisfying $\mathcal{A} v = 0$ for a first-order constant-rank differential operator $\mathcal{A}$.

Without attempting to give a comprehensive account of the field (see instead the books \cite{Dacorogna07}, \cite{FonsecaLeoni07}, \cite{RindlerBook}) let us underline that a common point of all these results, be it gradients, higher order gradients, or $\mathcal{A}$-free fields, is that the partial differential operators involved are homogeneous of fixed order. This, however, need not always be the case in applications. Pantographic sheets (see \cite{DellIsolaLekszyckiPawlikowskiGrygorukGreco15}, \cite{EremeyevDellIsolaBoutinSteigmann18}, \cite{TurcoGiorgioMisraDellIsola17}) furnish an example of a recently introduced metamaterial, characterised by energies of the form $\int |\partial_{x} u|^2 + |\partial^2_{yy} u|^2 \dd x$, i.e., with maximal derivatives of different orders in different directions. The aim of the present paper is to develop a systematic framework for studying such variational problems.

We are interested in existence of minimisers of variational problems posed in Sobolev spaces of mixed smoothness. For a fixed vector $\ba := (a_1, \ldots, a_N)$ with positive integer coordinates and a given function $u \colon \RRN \supset \Omega \to \RRn$ we denote by 
$$\ggrad u := (\partial^{\alpha} u)_{\langle \alpha, \bamo \rangle = 1}$$ 
the matrix whose $i$-th row is composed of derivatives $\partial^\alpha u^i$ of the $i$-th component of the map $u$, and where the multi-indices $\alpha$ satisfy $\langle \alpha, \bamo \rangle = \sum_{j=1}^N \frac{\alpha_j}{a_j} = 1$. We study functionals of the form
\begin{equation}\label{eqIntroFunctional} 
\WWap(\Omega;\RRn) \ni u \mapsto \int_\Omega F(\ggrad u(x)) \dd x,
\end{equation}
where $F$ is the integrand and $\WWap(\Omega; \RRn)$ is an appropriate Sobolev space of mixed smoothness, the elements of which satisfy $\partial^{\alpha} u \in \LLp$ for all $\alpha$ with $\langle \alpha, \bamo \rangle = 1$.

The central notion of this paper is that of $\ba$-quasiconvexity. We say that a function $F \colon \RRnm \to [-\infty, \infty)$ is $\ba$-quasiconvex if for every $V \in \RRnm$ one has
$$ F(V) \leq \inf_{u \in \Ccinfty(Q;\RRn)} \dashint_Q F(V + \ggrad u(x)) \dd x \ldotp$$
Our main results are on coercivity, lower semicontinuity, and relaxation of functionals of the form \eqref{eqIntroFunctional}. For continuous integrands $F$ with $|F(V)| \leq C(|V|^p +1)$ we prove in Theorem \ref{thmCoercivity} that, having fixed a Dirichlet class $\WWapg(\Omega)$, all minimising sequences of \eqref{eqIntroFunctional} are bounded if, and only if, there exists a constant $c > 0$ and a point $V_0 \in \RRnm$ such that $V \mapsto F(V) - c|V|^p$ is $\ba$-quasiconvex at $V_0$. Then, in Theorem \ref{thmLSCWithpGrowth}, we show that for continuous integrands $F \colon \RRnm \to [0,\infty)$ satisfying $|F(V)| \leq C(|V|^p + 1)$ the functional \eqref{eqIntroFunctional} is sequentially weakly lower semicontinuous on $\WWap(\Omega)$ if, and only if, $F$ is $\ba$-quasiconvex. Finally, in the last section, we study relaxations of \eqref{eqIntroFunctional} and in Theorem \ref{thmRelaxationPGrowth} show that, if $F$ is as before, then the sequentially weakly lower semicontinuous envelope of \eqref{eqIntroFunctional} is again an integral functional and
\begin{equation}\label{eqIntroRelax} 
\inf_{u_j \weakConv u} \left\{ \liminf_{j \to \infty} \int_\Omega F (\ggrad u_j) \dd x \right\}= \int_{\Omega} \QQ F(\ggrad u(x)) \dd x,
\end{equation}
where $\QQ F$ is the $\ba$-quasiconvex envelope of $F$ given by
$$ \QQ F(V) := \inf_{\phi \in \Ccinfty(\Omega)} \frac{1}{|Q|} \int_Q F(V + \ggrad \phi(x)) \dd x.$$
The infimum in \eqref{eqIntroRelax} is taken over all sequences $u_j$ converging to $u$ weakly in $\WWap(\Omega)$, and $\QQ F$ is the largest $\ba$-quasiconvex function that is no greater than $F$. This result is proven under the additional assumption that $F$ is locally Lipschitz, or that it satisfies $F(V) \geq c|V|^p - C$ for some constants $C, c > 0 $ and all $V \in \RRnm$. In the latter case we are able, in Theorem \ref{thmContinuousImpliesRelaxation}, to remove the $p$-growth bound from above and obtain a relaxation formula also for integrands $F$ that may take the value $+ \infty$. 

The starting point of this project is the theory of Sobolev spaces of mixed smoothness. The main question here is what other regularity and integrability properties follow when a function $u$ is assumed to be $\LLp$ integrable, together with all its derivatives $\partial^{a_i}_{x_i} u$ for some $\ba = (a_1, \ldots, a_N)$. The theory of embeddings of spaces of mixed smoothness was largely developed by Nikolskii, who started the study in \cite{Nikolskii51}. Since then a number of authors have made important contributions to the theory of spaces of mixed smoothness, among which we recall \cite{Besov67}, \cite{Besov74}, \cite{Boman72}, \cite{Burenkov66}, \cite{Ilin68}, \cite{Kolyada07},  \cite{KolyadaPerez04}, \cite{PelczynskiSenator86}, \cite{PelczynskiSenator862}, \cite{Slobodeckii58one}, \cite{Slobodeckii58two}, \cite{Solonnikov75}. Let us note here that the list is far from being fully comprehensive. Instead we refer the reader to the two-volume book by Besov, Il'in, and Nikolskii (see \cite{BesovIlinNikolskii78} and \cite{BesovIlinNikolskii78p2}), which we will follow in most of the technical preliminaries. Let us note that, in their book, the authors call the spaces we work with anisotropic Sobolev spaces. For the purpose of the present article we have opted for the name Sobolev spaces of mixed smoothness to avoid confusion as to the nature of the anisotropy present in the problems we consider. Indeed, variational problems with different growth properties in different directions are often called anisotropic in the existing literature.

The arguments in the main part of the paper, i.e., the lower semicontinuity and relaxation, are based on a Young measures approach.  These measures were first introduced by Young in \cite{Young37}, \cite{Young42one}, \cite{Young42two} and have since been used and studied by a number of authors, see \cite{Balder84}, \cite{Ball89}, \cite{BerliocchiLasry73}, \cite{Kristensen99one},  \cite{KinderlehrerPedregal91}, \cite{KinderlehrerPedregal94}, and \cite{McShane40} among others. A classical reference that contains a much more complete bibliography is \cite{Pedregal97}. The prevalence of Young measures in modern calculus of variations is due to the fact that they conveniently describe oscillation effects (see \cite{DiPernaMajda87} for a generalisation capable of describing concentration as well) that may occur in weakly, but not strongly, convergent sequences of vector fields. Thus, Young measures facilitate limit passages in nonlinear quantities, a crucial issue in the field. Here, we are particularly interested in Young measures generated by sequences of $\ba$-gradients which, as shown in Theorems \ref{thmCharacterizationOfYM} and \ref{thmCharacterisationYMnonhom}, can be characterised by duality with $\ba$-quasiconvex functions in the spirit of the Kinderlehrer-Pedregal (\cite{KinderlehrerPedregal91}, \cite{KinderlehrerPedregal94}) result for classical gradients. 

Throughout the paper we make ample use of what has been done in the classical gradient and $\mathcal{A}$-free frameworks. Our main focus is on the new difficulties induced by the mixed smoothness setting. To separate those from other technical issues, we work with the model case of autonomous integrands $F$, that is ones that only depend on $\ggrad u$ rather than lower order derivatives of $u$ or on the spatial variable $x$. We are also very liberal when it comes to assumptions on the domain $\Omega$ and we do not attempt to optimise our results in that regard. We also limit ourselves to the reflexive regime $\WWap(\Omega)$ with $p \in (1,\infty)$, although it would certainly be interesting to consider the case $p = 1$. 
Finally, whilst the present work establishes existence of minimisers, we do not say anything about their regularity. This will be treated in an ongoing collaboration with Kristensen \cite{KristensenProsinski21}.

\subsection{Acknowledgements}
The present work is based on the author's doctoral thesis \cite{ProsinskiThesis} prepared under the supervision of Prof. Jan Kristensen, whose support and guidance have been of immense help. The author would also like to thank Prof. Sir John Ball FRS, Prof. Gregory Seregin, Prof. Luc Nguyen, Prof. Gui-Qiang Chen, and Prof. Kewei Zhang, who have all acted as referees for the thesis at various stages of its completion. The generous financial support of Oxford EPSRC CDT in Partial Differential Equations, the Clarendon Fund, and St John's College Oxford is gratefully acknowledged.

\section{Spaces of mixed smoothness}\label{chapter:chapter1}
We begin by introducing the function spaces in which our variational problems are set. We collect basic facts about Sobolev spaces of mixed smoothness $\WWap$ which, while available in the literature, might not be as classical and well-known as the corresponding theory of the usual Sobolev spaces $\WW^{k,p}$. Let us note that, what we call Sobolev spaces of mixed smoothness, is often referred to as `anisotropic Sobolev spaces' in the literature that we cite. We have opted for the name `mixed smoothness' as the term `anisotropic variational problems' is already widespread and used to describe problems where the integrand and the input functions exhibit different growth properties in derivatives in different directions. That is, `anisotropic variational problems' usually refer to problems posed in the space $\WW^{k,\mathbf{p}}$ with vector parameter $\mathbf{p}$, rather than $\WWap$ with vector parameter $\ba$, which we are interested in here. Nevertheless, our setting is certainly anisotropic and we shall use this term occasionally, particularly when talking about scaling.

\subsection{Preliminaries}\label{sectionPreliminaries}
Any partial differentiation operator $\partial^{\alpha} := \partial_1^{\alpha_1} \partial_2^{\alpha_2} \ldots \partial_N^{\alpha_N}$ acting on functions mapping a subdomain of $\RRN$ to $\RRn$ may be identified with the multi-index $\alpha = (\alpha_1, \ldots, \alpha_N) \in \ZZN_+$, where $\ZZN_+$ denotes the set of points in $\RRN$ with non-negative integer coordinates. 

Specifying a set of derivatives $A \subset \ZZN_+$ and their desired integrability defines a Sobolev-like space --- for example the classical $\WW^{k,p}$ Sobolev spaces correspond to $A := \{ \alpha \in \ZZN_+ \colon |\alpha| \leq k\}$ with $\LLp$ integrability on all the derivatives. In this work our principal assumption is that the functions in our spaces admit maximal pure derivatives in each direction, but the orders of these maximal derivatives remain arbitrary. 

Here and in all that follows, $\Omega$ is a bounded open Lipschitz subset of $\RRN$ with $|\partial \Omega| = 0$, where $|\partial \Omega|$ denotes the $N$-dimensional Lebesgue measure of the the boundary of $\Omega$. We denote by $\Ccinfty(\Omega, \RRn)$ the space of smooth and compactly supported functions $\phi \colon \Omega \to \RRn$. We often omit the target space when it is clear from the context and simply write $\Ccinfty(\Omega)$.

We fix a vector $\ba = (a_1, \ldots, a_N) \in \naturals^N$, where the respective entries $a_i$ denote the desired maximal order of differentiability with respect to the $x_i$ coordinate. We fix an exponent $p \in (1, \infty)$, let $\bamo := (a_1^{-1}, \ldots, a_N^{-1})$,  and introduce the following

\begin{defi}
For a bounded open set $\Omega \subset \RRN$ the Sobolev space $\WWap(\Omega; \RRn)$ is defined as the completion of $C^{\infty}(\Omega; \RRn) \cap \{ \phi \colon  \sum_{\langle \alpha, \bamo \rangle \leq 1} \| \partial^{\alpha} \phi \|_{\LLp(\Omega; \RRn)} < \infty \}$ with respect to the norm
$$ \| u \|_{\WWap} := \sum_{\langle \alpha, \bamo \rangle \leq 1} \| \partial^{\alpha} u \|_{\LLp(\Omega; \RRn)} \ldotp$$
We often omit the target space $\RRn$ and write simply $\WWap(\Omega)$ or even $\WWap$ if the domain $\Omega$ is clear from the context.
We also denote by $\WWapn(\Omega; \RRn)$ the completion of $C_c^{\infty}(\Omega; \RRn)$ in the same norm. 
\end{defi}

The set $\{\alpha \in \ZZ_{\geq 0}^N \colon \langle \alpha, \bamo \rangle \leq 1\}$ has the important property that if $\alpha, \beta \in\ZZ_{\geq 0}^N$ are two multi-indices with $\beta \leq \alpha$ (coordinate-wise) and $\alpha$ is in our set then so is $\beta$, which makes the collection a smoothness in the language of Pe\l{}czy\'{n}ski-Senator (see \cite{PelczynskiSenator86}). Moreover, all the maximal elements of the smoothness lie on a common hyperplane $\{\alpha \colon \langle \alpha, \bamo \rangle = 1\}$, thus allowing for convenient scaling, which we will discuss later. This hyperplane is called a pattern of homogeneity by Kazaniecki, Stolyarov, and Wojciechowski in \cite{KazanieckiStolyarovWojciechowski17}. It is worth mentioning that their paper is the first one to introduce a simple version of anisotropic quasiconvexity (to be discussed later) thus inspiring the present work. Let us observe that, in our case, the hyperplane of homogeneity intersects all coordinate axes at integer points (i.e. we have maximal pure derivatives in all directions), which is important for the structure of the relevant Sobolev spaces. Our exposition of the theory of these spaces is based on the book \cite{BesovIlinNikolskii78} by Besov, Il'in, and Nikolskii.
Finally, let us remark that the functions considered are, in general, vector-valued and we impose the same differentiability on all components of the functions considered. In general, it would be interesting to allow for different smoothnesses in different components. This is, however, outside of the scope of the present paper. 

\begin{proposition}[see \cite{BesovIlinNikolskii78}]
For $p \in [1, \infty)$ the space $\WWap(\Omega; \RRn)$ coincides with the space of functions $u \in \LLp(\Omega; \RRn)$ with distributional derivatives $\partial^{\alpha} u \in \LLp(\Omega; \RRn)$ for all $\langle \alpha, \bamo \rangle \leq 1$.
The spaces $\WWap(\Omega)$ and $\WWapn(\Omega)$ are both separable Banach spaces. For $p \in (1, \infty)$ the two spaces are also reflexive.
\end{proposition}
This is shown in \cite{BesovIlinNikolskii78}, bar the reflexivity part, which may be immediately deduced by considering $\WWap(\Omega; \RRn)$ as a closed subspace of $\bigoplus_{\langle \alpha, \bamo \rangle \leq 1} \LL^p(\RRN; \RRn)$ through the embedding $u \mapsto (\partial^\alpha u)_{\langle \alpha, \bamo \rangle \leq 1}$.

\subsection{Embeddings of Sobolev spaces of mixed smoothness}\label{sectionEmbeddings}
An important structural property of spaces of mixed smoothness is the existence of continuous and compact embeddings, similar to the classical Sobolev embeddings. To begin with, we note that, as in the classical case of, say $\WW^{1,p}(\Omega)$, there are certain regularity assumptions that one must impose on the domain $\Omega$.

\begin{defi}
Let $b \in \RRN$ be a vector with non-zero coordinates. Fix $h \in (0,\infty)$ and $\epsilon \in (0,\infty)$. The set 
$$V(b, h, \epsilon) := \bigcup_{0 < v < h} \left\{ x \in \RRN \colon \frac{x_i}{b_i} > 0, v < (\frac{x_i}{b_i})^{a_i} < (1+\epsilon)v \text{ for all } i \in \{1,2 \ldots, N\} \right\}$$
is called an $\ba$-horn of radius $h$ and opening $\epsilon$.  
\end{defi}

\begin{defi}
Let $\Omega \subset \RRN$ be open and let $K \in \naturals$. Suppose that for $k \in \{1, 2, \ldots, K\}$ there exist open sets $\Omega_k$ and $\ba$-horns $V_k$ (with coefficients $b_k,h_k,\epsilon_k$ depending on $k$) such that 
$$ \Omega = \bigcup_{k=1}^K \Omega_k = \bigcup_{k=1}^K (\Omega_k + V_k) \ldotp$$
Then we say that $\Omega$ satisfies the weak $\ba$-horn condition.
\end{defi}

\begin{theorem}[see Theorem 9.5 in \cite{BesovIlinNikolskii78}]\label{thmBesovIlinNikolskiiIneq}
Suppose that an open set $\Omega \subset \RRN$ satisfies the weak $\ba$-horn condition and let $p \in (1,\infty)$. Then there exists a real number $h_0 \in (0,\infty)$ depending on $\Omega$ and a constant $C$ such that, for all $h \in (0,h_0)$ and all $u \in \WWap(\Omega)$, one has, for all multi-indices $\beta \in \ZZ^N_+$ with $\langle \beta, \bamo \rangle \leq 1$, that 
$$ \| \partial^{\beta} u \|_p \leq C \left( h^{1 - \langle \beta, \bamo\rangle} \sum_{i=1}^N \| \partial_i^{a_i} u \|_p +  h^{- \langle \beta, \bamo \rangle} \|u\|_p \right) \ldotp$$
\end{theorem}

Thus, on domains satisfying the weak $\ba$-horn condition (see \cite{BesovIlin68} where the condition was first studied) the intermediate derivatives are controlled by the maximal pure derivatives and the function itself, so that the relevant Sobolev space is well behaved. 

In the isotropic case (i.e., $a_i = a_j$ for all $i,j$) the $\ba$-horn is in fact a cone and horn conditions are equivalent to the, more familiar, cone conditions. In the genuinely anisotropic scenario the $\ba$-horn condition is more surprising. For instance, (see \cite{BesovIlinNikolskii78}) the two-dimensional disc only satisfies the weak $\ba$-horn condition if $\frac{1}{2} a_1 \leq a_2 \leq 2 a_1$. Fortunately, there are no issues when working with rectangular domains, which we note in the following:

\begin{lemma}[see \cite{BesovIlinNikolskii78}]\label{lemmaBoxesAHorn}
Any set of the form $(l_1, r_1) \times \ldots \times (l_N, r_N) \subset \RRN$ for some $l_i, r_i \in \RR$ satisfies the weak $\ba$-horn condition.
\end{lemma}

Before we proceed, let us note several important consequences of the embeddings:
\begin{proposition}\label{propCompactIntoLp}
Let $\Omega \subset \RRN$ be a bounded open set with a Lipschitz boundary. Then the embedding of $\WWap(\Omega)$ into $\LL^p(\Omega)$ is compact.
\end{proposition}
\begin{proof}
For any $\ba$ one has the inclusion $\WWap(\Omega) \subset \WW^{1,p}(\Omega)$ into the standard Sobolev space. Since the inclusion $\WW^{1,p}(\Omega) \subset \LLp(\Omega)$ is compact the proof is finished.
\end{proof}

\begin{lemma}\label{lemmaCompactEmbedding}
Suppose that a bounded open set $\Omega \subset \RRN$ with a Lipschitz boundary satisfies the weak $\ba$-horn condition and let $p \in (1,\infty)$.
Then for any $\beta$ with $\langle \beta, \bamo \rangle < 1$ the mapping $\WWap(\Omega) \hookrightarrow \LLp(\Omega)$ given by $u \mapsto \partial^\beta u$ is completely continuous, i.e., if $u_j \weakConv u$ in $\WWap(\Omega)$ then $\partial^\beta u_j \to \partial^\beta u$ in $\LLp(\Omega)$.
\end{lemma}
\begin{proof}
Considering $u_j - u$ instead of $u_j$ we may assume that $u = 0$. Fix a $\beta$ with $\langle \beta, \bamo \rangle < 1$ and note that Theorem \ref{thmBesovIlinNikolskiiIneq} shows that there exists an $h_0 > 0$ such that, for all $h \in (0, h_0)$, one has 
$$ \| \partial^{\beta} u_j \|_p \leq C \left( h^{1 - \langle \beta, \bamo\rangle} \left(\sum_{\langle \alpha, \bamo \rangle = 1} \| \partial^\alpha u \|_p \right)  + h^{- \langle \beta, \bamo\rangle} \|u_j\|_p \right) \ldotp$$
Proposition \ref{propCompactIntoLp} implies that $u_j$ converges strongly to $0$ in $\LLp$. Hence, there exists a sequence $h_j \in (0, h_0)$ with $h_j \to 0$ and $C h_j^{-\langle \beta, \bamo\rangle} \|u_j\|_p \to 0 $ as $j \to \infty$. Finally, because $u_j$ is bounded in $\WWap$, we know that  $\left(\sum_{\langle \alpha, \bamo \rangle = 1} \| \partial^\alpha u \|_p \right)$ is bounded, and since $C h_j^{1 - \langle \beta, \bamo\rangle} \to 0$ we conclude that $\| \partial^{\beta} u_j \|_p \to 0 $, which ends the proof.
\end{proof}

Another consequence of Theorem \ref{thmBesovIlinNikolskiiIneq} is 
\begin{corollary}\label{corAllNormsEquivalent}
Let $\Omega \subset \RRN$ satisfy the weak $\ba$-horn condition. Then all of the following
\begin{equation}
\begin{aligned}
\| u \| &:= \| u \|_p + \sum_{i=1}^N \| \partial^{a_i}_i u \|_p, \\
\|u \| &:= \| u \|_p + \sum_{\langle \alpha, \bamo \rangle = 1} \| \partial^{\alpha} u \|_p, \\
\|u \| &:= \sum_{\langle \beta, \bamo \rangle \leq 1} \| \partial^{\beta} u \|_p,
\end{aligned}
\end{equation}
yield equivalent norms on $\WWap(\Omega)$.
\end{corollary}

In closing this subsection let us note that the study of embeddings for spaces of mixed smoothness has been started by Besov and Il'in in \cite{BesovIlin68}. Together with Nikolskii, these authors have expanded, compiled, and clarified the theory in \cite{BesovIlinNikolskii78}. It is, however, not the only relevant source. Similar questions have been studied, among others, by Boman (see  \cite{Boman72}) who also treats the case $p = \infty$, as well as Demidenko and Upsenskii (see \cite{DemidenkoUpsenskii03}) who worked on quasielliptic operators, a class important in regularity of solutions to mixed smoothness variational problems, which we will return to in a forthcoming paper. Finally, in the particular case of rectangular sets (like in Lemma \ref{lemmaBoxesAHorn}) a bounded extension may be constructed using the Hestenes' method (as observed by Burenkov and Fain in \cite{BurenkovFain76}) which then, together with embeddings on the full space, yields existence of embeddings on rectangular domains as well.

\subsection{Canonical Projection}\label{sectionProjection}
For $u \in \WWap$ we write $\ggrad u$ for the $\ba$-gradient of $u$ given by $\ggrad u := (\partial^{\alpha}u)_{\langle \alpha, \bamo \rangle = 1}$. For future use let us denote the cardinality of the set $\{ \alpha \in \mathbb{Z}_+ \colon \langle \alpha, \bamo \rangle = 1\}$ by $m$, so that for $u \in \WWap(\Omega, \RRn)$ the $\ba$-gradient is a map defined on $\Omega$ with values in $\RRnm$.
In what follows we will often need to carry out certain operations, for example truncations, on $\ba$-gradients of various functions. These are easy to do on mappings $\Omega \to \RRnm$, but they need not preserve the $\ba$-gradient structure, and to remedy that we turn to Canonical Sobolev Projections following Pe\l{}czy\'{n}ski's work in \cite{Pelczynski89}. The aim is to obtain an analogue of the Helmholtz decomposition for the mixed smoothness setting. We will then use it for regularizing generating sequences of certain Young measures, similarly to what has been done by Fonseca and M\"{u}ller in \cite{FonsecaMuller99} in the case of $\mathcal{A}$-free vector fields and $\mathcal{A}$-quasiconvexity.

There is a canonical embedding 
$$ \WWap(\RRN; \RRn) \rightarrow \bigoplus_{\langle \alpha, \bamo \rangle \leq 1} \LL^p(\RRN; \RRn)$$
given by $ u \mapsto (\partial^{\alpha}u)_{\langle \alpha, \bamo \rangle \leq 1},$
but it is not surjective. With $p=2$ one may define the canonical projection of the target space onto the image of this embedding, i.e.,
\begin{equation}\label{eqProjectionDef} 
\projection \colon \bigoplus_{\langle \alpha, \bamo \rangle \leq 1} \LL^2(\RRN; \RRn) \to \im \left(\WWat(\RRN; \RRn) \to \bigoplus_{\langle \alpha, \bamo \rangle \leq 1} \LL^2(\RRN; \RRn) \right).
\end{equation}
It has been shown in \cite{Pelczynski89} (see Corollary 5.1 therein) that the projection from \eqref{eqProjectionDef} is of strong type $(p,p)$ for $1<p<\infty$, thus one can extend it by continuity from 
$$\bigoplus_{\langle \alpha, \bamo \rangle \leq 1} \LL^2 (\RRN; \RRn) \to \bigoplus_{\langle \alpha, \bamo \rangle \leq 1} \LL^2 (\RRN; \RRn)$$ 
to 
$$\bigoplus_{\langle \alpha, \bamo \rangle \leq 1} \LL^p (\RRN; \RRn) \to \bigoplus_{\langle \alpha, \bamo \rangle \leq 1} \LL^p (\RRN; \RRn).$$

\begin{lemma}\label{lemmaProjection}
Fix any $p \in (1, \infty)$ and denote by $\projection$ the extension of the canonical projection discussed above. Then:

i) the map $\projection$ is a bounded linear operator on $\bigoplus_{\langle \alpha, \bamo \rangle \leq 1} \LL^p(\RRN; \RRn)$;

ii) for any $V \in \bigoplus_{\langle \alpha, \bamo \rangle \leq 1} \LL^p$ we have $\projection(\projection V) = \projection V$; 

iii) if the family $\{ V_j \} \subset \bigoplus_{\langle \alpha, \bamo \rangle \leq 1} \LL^p$ is $p$-equiintegrable then so is $\{ \projection V_j \}$.
\end{lemma}

\begin{proof}
The first assertion is the content of Corollary 5.1 in \cite{Pelczynski89}. 
Point ii) is, by definition, true in $\bigoplus_{\langle \alpha, \bamo \rangle \leq 1} \LL^2 (\RRN; \RRn)$. For a general exponent $p$ let us fix $V \in \bigoplus_{\langle \alpha, \bamo \rangle \leq 1} \LL^p (\RRN; \RRn)$ and a family $V_j \subset \bigoplus_{\langle \alpha, \bamo \rangle \leq 1} \LL^2 (\RRN; \RRn)$ with $\| V - V_j \|_p \to 0$. By continuity of $\projection$ (and thus of $\projection \circ \projection$) one has 
$$ \projection V_j \to \projection V \, \text{ and } \, \projection(\projection V_j) \to \projection(\projection V) \, \text{ in } \, \LL^p (\RRN; \RRn),$$
and since $\projection V_j = \projection (\projection V_j)$ for each $j$ the claim is proven.

For the last part consider the standard truncation $\tau_k$ given by 
\begin{equation}\label{eqDefTau}
\tau_{k}(X) := 
\begin{cases}
X \quad \text{if } |X| \leq k,\\
k \frac{X}{|X|} \quad \text{if } |X| > k \ldotp
\end{cases}
\end{equation}
Then fix any sequence $V_j$ satisfying the assumptions of point iii). Since $\{\tau_{k}(V_j)\}$ is bounded in $\LL^{\infty}$ and in $\LL^p$ we know, by continuity of $\projection$ as a map from $\LL^q$ to $\LL^q$, that $\{ \projection \tau_{k}(V_j) \}$ is bounded in any $\LL^q$ with $p \leq q < \infty$, so that this family is equiintegrable in $\LL^p$. Then again, $p$-equiintegrability of $\{V_j\}$ itself yields
$$ \lim_{k \rightarrow \infty} \sup_j \left|\left|V_j - \tau_{k}(V_j)\right|\right|_p = 0,$$
so again continuity of $\projection: \LL^p \rightarrow \LL^p$ gives
$$ \lim_{k \rightarrow \infty} \sup_j \left|\left|\projection( V_j - \tau_{k}(V_j) )\right|\right|_p = 0,$$
hence $\{ \projection(V_j) \}$ is $p$-equiintegrable as claimed.

\end{proof}

\subsection{Anisotropic scaling}\label{sectionScaling}
In what follows we will often need to use a specific anisotropic scaling. For a real number $R > 0$ and a vector $v = (v_1, v_2, \ldots, v_N) \in \RRN$ we let 
$$R \scale v := (R^{1/{a_1}}v_1, R^{1/a_2} v_2, \ldots, R^{1/a_N} v_N).$$ 
Here, and in all that follows, $Q \subset \RRN$ will, unless otherwise specified, denote $[-1,1]^N$. We let $Q_R(x_0) \subset \RRN$ be the open box centred at $x_0 \in \RRN$ and scaled according to the rule just described, so that 
$$ Q_R(x_0) := \{ x \in \RRN \colon |x^i - x^i_0|^{a_i} < R \text{ for all } 1 \leq i \leq N \}.$$
Equivalently, we could write $Q_R(x_0) = x_0 + R \scale Q$. We call $R$ the anisotropic radius of the box $Q_R$ and $x_0$ its centre. 
From now on we will always understand a `box' to mean a set of the form above, i.e., an anisotropically scaled and translated unit cube. 

Observe that, unless the scaling is in fact isotropic (that is, $a_i = a_j$ for all $i,j$),  our family of boxes is not of bounded eccentricity, i.e., there does not exist a constant $c >0$ such that each box $Q_R(x_0)$ in our collection is contained in some (Euclidean) ball $B$ with $|Q_{R}(x_0)| \geq c|B|$. Therefore, it is not obvious if standard results such as the Vitali covering lemma, or the Lebesgue differentiation theorem, hold for balls replaced by anisotropically scaled boxes. Note that even sharper statements of the Vitali covering lemma, such as the one in \cite{Saks37} by Saks (which only requires the eccentricity to be bounded along fixed sequences converging to a given point), or the one in \cite{MejlbroTopsoe77} by Mejlbro and Tops{\o}e (where the condition on the eccentricity is in integral form), are not directly applicable. 

Nevertheless, these results may be proven in a straightforward way --- it suffices to follow the proof of Vitali's covering lemma given in \cite{BenedettoCzaja}, and the Lebesgue differentiation theorem (which is what we are after) follows immediately. In fact, the result could also be deduced as a special case of the more general work by Calder\'{o}n and Torchinsky (see \cite{CalderonTorchinsky}), thus we omit the proof of the following:

\begin{theorem}[Anisotropic Lebesgue differentiation theorem]\label{thmAnisotropicLebesgueDiff}
Let $f \in \LL^1_{\text{loc}}(\RRN)$. For Lebesgue almost every $x_0 \in \RRN$ one has
$$ \limsup_{R \to 0} \frac{1}{|Q_R(x_0)|} \int_{Q_R(x_0)} |f(x) - f(x_0)| \dd x = 0 \ldotp$$
\end{theorem}

\subsection{Polynomial approximation}\label{sectionPolynomial}
For a function $f \in \LL^1(Q_R(x_0))$ we denote by $(f)_{Q_R(x_0)}$ its average over $Q_R(x_0)$, i.e., 
$$(f)_{Q_R(x_0)} := \dashint_{Q_R(x_0)} f(x) \dd x.$$
We often write simply $Q_R$ if the center is not important or clear from the context.
\begin{lemma}\label{lemmaCutOff}
There exists a constant $C$ such that, for any $r > 0$ and any $\sigma \in (0,\frac{1}{2})$, there exists a cut-off function $\eta \in \Ccinfty (Q_r; [0,1])$ which is identically equal to $1$ on $Q_{(1-\sigma)r}$ and satisfies
$$ \| \partial^{\beta} \eta \|_{\LLinfty} \leq C r^{-\langle \beta, \bamo \rangle} \sigma^{-|\beta|},$$
for all multi-indices $\beta$ with $\langle \beta, \bamo \rangle \leq 1$. 
\end{lemma}
\begin{proof}
It is enough to consider the case $r = 1$, as the general result will then follow by scaling. With $r = 1$ observe that the distance between the faces of $Q$ and $Q_{1-\sigma}$ along the $x_i$ axis is equal to $1 - (1-\sigma)^{1/a_i}$. For sufficiently small $C$ the function $C \sigma + (1-\sigma)^{1/a_i}$ is decreasing in $\sigma$, so there exists a constant $C > 0$ such that, for all $\sigma \in (0, \frac{1}{2})$, we have
$$ 1 - (1-\sigma)^{1/a_i} \geq C \sigma,$$
for all $i$.  Now it is enough to construct one dimensional cut-off functions $\eta^i(x_i)$ that realise the desired cut-off along the particular axes and satisfy $\| \partial^k \eta^i \|_{\LLinfty} \leq 2 (C\sigma)^{-k}$. We then conclude by setting $\eta(x) := \prod_{i=1}^N \eta^i(x_i)$.
\end{proof}

We recall the following version of the Poincar\'{e} inequality in $\WWap$ proven by Dupont and Scott in \cite{DupontScott80}, where, instead of requiring zero boundary values, we allow for correction in terms of the kernel of the operator $\ggrad$. 
\begin{proposition}[see Theorem 4.2 in \cite{DupontScott80}]\label{propPolynomialApproximation}
There exists a constant $C$ such that for any function $u \in \WWap(Q)$ with $p \in [1,\infty)$ there exists a polynomial $P_u \in C^{\infty}(Q)$ with $\ggrad P_u \equiv 0$ satisfying
$$ \| u - P_u \|_{\WWap(Q)} \leq C \|\ggrad u \|_{\LLp(Q)}.$$
\end{proposition}
Observe that, in the above, $Q$ is fixed to be the unit cube, and we do not assert anything about approximations on other domains. However, we will only ever use this result on anisotropic boxes, and it is easy to see how to adjust the constant to scaling, as shown in the following:
\begin{corollary}\label{corPolynomialApproximation}
There exists a constant $C$ such that for any function $u \in \WWap(Q_r)$ with $p \in [1,\infty)$ there exists a polynomial $P_u \in C^{\infty}(Q_r)$ with $\ggrad P_u \equiv \left( \ggrad u \right)_{Q_r}$ such that for any $\beta$ with $\langle \beta, \bamo \rangle \leq 1$ we have
$$ r^{-1 + \langle \beta, \bamo \rangle} \| \partial^\beta (u - P_u) \|_{\LLp(Q_r)} \leq C \|\ggrad (u - P_u) \|_{\LLp(Q_r)}.$$
The constant $C$ does not depend on the function $u$ nor the radius $r$.
\end{corollary}
\begin{proof}
First of all, note that using our anisotropic rescaling we may reduce to the case $r = 1$. This also determines the scaling, i.e., the $ r^{-1 + \langle \beta, \bamo \rangle}$ factor. Secondly, it is clearly enough to prove the result for $u \in C^\infty(Q)$, as the general case then follows from density of smooth functions in $\WWap(Q)$. Observe that by considering 
$$ \widetilde{u}(x) := u(x) - \sum_{\langle \alpha, \bamo \rangle = 1} x^\alpha \left( \partial^\alpha u \right)_Q,$$
we reduce our task to finding a polynomial $\widetilde{P}_{\widetilde{u}}$ with $\ggrad \widetilde{P}_{\widetilde{u}} \equiv 0$ and such that 
$$ \| \partial^\beta (\widetilde{u} - \widetilde{P}_{\widetilde{u}}) \|_{\LLp(Q)} \leq C \|\ggrad \widetilde{u} \|_{\LLp(Q)}$$
for all $\beta$ with $\langle \beta, \bamo \rangle < 1$, as the case $\langle \beta, \bamo \rangle = 1$ is trivial. The existence of such a $\widetilde{P}_{\widetilde{u}}$ is the content of Proposition \ref{propPolynomialApproximation}, which completes the proof.
\end{proof}

\begin{proposition}[see Theorem 10.16 in \cite{DacorognaMarcellini99}]\label{propPiecewiseApprox}
For any bounded open Lipschitz domain $\Omega$, any $u \in \WWap(\Omega)$ and any $\epsilon > 0$ there exist a function $u_\epsilon \in \WWap_u(\Omega)$ and a finite family of disjoint boxes $\{Q_{\epsilon,i}\}_i$ such that $\ggrad u_\epsilon$ is constant on each $Q_{\epsilon,i} \subset \Omega$, $\left| \Omega \setminus \bigcup_i Q_{\epsilon, i} \right| < \epsilon$, and $\|u - u_\epsilon\|_{\WWap(\Omega)} < \epsilon$.
\end{proposition}
\begin{proof}
Fix an arbitrary $u \in \WWap(\Omega)$ and a parameter $\tau \in (0,1)$ to be determined later. Decompose $\Omega$, up to a set of measure zero, into a countable family of disjoint, open boxes $\{Q_{\tau,i}\}$ of radii equal, or smaller than, $\tau$. Select a finite subfamily of $I$ boxes covering $\Omega$ up to a set of measure less than $\epsilon /2$ and relabel the elements so that $\left| \Omega \setminus \bigcup_{i=1}^I Q_{\tau,i} \right| < \epsilon / 2$. From now on we focus our attention only on the boxes $Q_{\tau,i}$ with $1 \leq i \leq I$. 
Let $P_{\tau,i}$ denote the polynomial approximating $u$ on $\Qtaui$ given by Corollary \ref{corPolynomialApproximation}. Let $\sigma \in (0,1/2)$ be a parameter to be determined later. For every $\Qtaui$ take a cut-off function $\etataui \in \Ccinfty(\Qtaui)$ identically equal to $1$ on $(1-\sigma) \scale Q_{\tau,i}$, as in Lemma \ref{lemmaCutOff}. Define
$$ v(x) := u(x) + \sum_{i = 1}^I \etataui(x) \left(\Ptaui(x) - u(x)\right),$$
so that $\ggrad v$ is constant on each $(1-\sigma) \scale Q_{\tau,i}$ and $v \in \WWap_u(\Omega)$. We may now calculate
\begin{equation*}
\begin{aligned} 
\| u - v\|_{\WWap(\Omega)}^p = & \sum_{\langle \beta, \bamo \rangle \leq 1} \int_\Omega | \partial^\beta (u-v) |^p \dd x \\
= & \sum_{\langle \beta, \bamo \rangle \leq 1}  \sum_{i=1}^I \int_{\Qtaui} \left| \partial^\beta \left(\etataui(x) \left(u(x) - \Ptaui(x)\right) \right)  \right|^p \dd x \\
 \leq  & \sum_{\langle \beta, \bamo \rangle \leq 1}  \sum_{i=1}^I \int_{\Qtaui} \sum_{0 \leq \gamma \leq \beta}  \left| \partial^\gamma \etataui(x) \right|^p \left| \partial^{\beta - \gamma} \left(u(x) - \Ptaui(x)\right)   \right|^p \dd x.
\end{aligned}
\end{equation*}
Using the bounds on the derivatives of $\eta$ we may write
$$ \| u - v\|_{\WWap(\Omega)}^p  \leq C \sum_{\langle \beta, \bamo \rangle \leq 1} \sum_{0 \leq \gamma \leq \beta} \sum_{i=1}^I \sigma^{-p|\beta|} \tau^{-p\langle \gamma, \bamo \rangle} \int_{\Qtaui}  \left| \partial^{\beta - \gamma} \left(u(x) - \Ptaui(x)\right)   \right|^p \dd x.$$
Using the bound from Corollary \ref{corPolynomialApproximation} on each $\Qtaui$ now yields
$$ \| u - v\|_{\WWap(\Omega)}^p  \leq C \sum_{\langle \beta, \bamo \rangle \leq 1} \sum_{0 \leq \gamma \leq \beta} \sum_{i=1}^I \sigma^{-p|\beta|} \tau^{-p \langle \gamma, \bamo \rangle} \tau^{p - p\langle \beta - \gamma, \bamo \rangle} \|\ggrad (u - \Ptaui)\|_{\LLp(\Qtaui)}^p.$$
Thus,
$$ \| u - v\|_{\WWap(\Omega)}^p  \leq C \sum_{\langle \beta, \bamo \rangle \leq 1} \sum_{0 \leq \gamma \leq \beta} \sum_{i=1}^I \sigma^{-p|\beta|} \tau^{p - p\langle \beta, \bamo \rangle} \|\ggrad (u - \Ptaui)\|_{\LLp(\Qtaui)}^p,$$
and finally, since $\langle \beta, \bamo \rangle \leq 1$, $\sigma < 1/2$, and $\tau < 1$, we may write
$$ \| u - v\|_{\WWap(\Omega)}^p  \leq C \sigma^{-p\max_j a_j} \sum_{i=1}^I  \|\ggrad (u - \Ptaui)\|_{\LLp(\Qtaui)}^p.$$
Now it is time to choose the parameters $\sigma$ and $\tau$. Observe that $\frac{ |(1-\sigma) \scale Q_{\tau,i}|}{|Q_{\tau,i}|} = (1 - \sigma)^{|\bamo|}$ for any $\tau$ and any $i$. Thus, choosing $\sigma$ small enough ensures that, with any $\tau$, we will have $\left| \Omega \setminus \left( \bigcup_i (1-\sigma) \scale Q_{\tau,i}\right) \right| < \epsilon$. With $\sigma$ fixed it is enough to choose $\tau$ small enough so that 
$$ \sum_{i=1}^I  \|\ggrad (u - \Ptaui)\|_{\LLp(\Qtaui)}^p \leq C^{-1} \sigma^{p \max_j a_j} \epsilon,$$
which is possible as, thanks to Theorem \ref{thmAnisotropicLebesgueDiff}, one can approximate $\ggrad u$ in the $\LLp$ norm by its averages over a grid of boxes of sufficiently small radii. Using the corresponding $v$ as $u_\epsilon$ ends the proof.
\end{proof}

\section{Young measures}\label{chapter:chapter2}
The main technical tool that we will use in studying lower semicontinuity of integral functionals defined on Sobolev spaces of mixed smoothness is the theory of Young measures (see  \cite{Young37}, \cite{Young42one}, \cite{Young42two}, and \cite{Young00}). These measures describe the behaviour of weakly converging sequences more accurately than just their weak limits and facilitate limit passages in nonlinear quantities. 

In this section we introduce, and study, the oscillation Young measures generated by weakly convergent sequences of $\ba$-gradients. 
In doing so we follow the strategies of Kristensen in \cite{KristensenNotes} (who studies classical gradients) and Fonseca and M\"{u}ller in \cite{FonsecaMuller99} (who work with $\mathcal{A}$-free vector fields). In what follows we focus on the adjustments required by the mixed smoothness setting whilst skipping the technical parts that translate without major adaptations. We refer the reader to the author's thesis \cite{ProsinskiThesis}, where a full step-by-step exposition may be found.

\subsection{Oscillation Young measures}\label{sectionOscillationYMs}
We denote by $\calM(\RRnm)$ the set of all Radon measures on $\RRnm$, and by $\Probability(\RRnm) \subset \calM(\RRnm)$ the set of all probability measures on $\RRnm$. We say that a function $F \colon \Omega \times \RRnm \rightarrow (-\infty, \infty]$ is a normal integrand if $F$ is Borel measurable and, for every fixed $x \in \Omega$, the function $W \mapsto F(x, W)$ is lower semicontinuous. Similarly, we say that a function $F \colon \Omega \times \RRnm \rightarrow \RR$ is Carath\'{e}odory if both $F$ and $-F$ are normal integrands. Finally, a map $\nu \colon \Omega \rightarrow \calM (\RRnm)$ is said to be weak*-measurable if $x \mapsto \langle \nu_x, \varphi \rangle$ is (Lebesgue) measurable for any continuous and compactly supported function $\varphi \colon \RRnm \to \RR$.

We begin with the following version of the Fundamental Theorem of Young Measures which may be found, for example, in Pedregal's book, (see \cite{Pedregal97}, followed by a simple result on translations.
\begin{theorem}[see \cite{Pedregal97}]\label{thmFToYM}
Let $\Omega \subset \RRN$ be a measurable set of finite measure and $V_j \colon \Omega \rightarrow \RRnm$ be a bounded sequence of $\LL^p$ functions for some $p \in [1, \infty]$. Then there exists a subsequence $V_{j_k}$ and a weak$^*$-measurable map $\nu \colon \Omega \rightarrow \calM (\RRnm)$ such that the following hold:

i) every $\nu_x$ is a probability measure;

ii) if $F \colon \Omega \times \RRnm \rightarrow \RR \cup \{\infty\}$ is a normal integrand bounded from below, then
$$ \liminf_{j \rightarrow \infty} \int_{\Omega} F(x, V_{j_k}(x)) \dd x  \geqslant \int_{\Omega} \overline{F}(x) \dd x,$$
where
$$ \overline{F}(x) := \langle \nu_x, F(x,\cdot) \rangle = \int_{\RRnm}F(x,y) \dd \nu_x(y);$$

iii) if $F \colon \Omega \times \RRnm \rightarrow \RR \cup \{\infty\}$ is Carath\'{e}odory and bounded from below, then
$$ \lim_{j \rightarrow \infty} \int_{\Omega} F(x, V_{j_k}(x)) \dd x  = \int_{\Omega} \overline{F}(x)\dd x < \infty$$
if and only if $\{F(\cdot, V_{j_k}(\cdot))\}$ is equiintegrable (in the usual $\LL^1$ sense). In this case
$$ F(\cdot, V_{j_k}(\cdot)) \weakConv \overline{F} \text{ in } \LL^1(\Omega) \cdot$$
The family $\{ \nu_x \}_{x \in \Omega}$ is called the (oscillation) Young measure generated by $V_{j_k}$. If there exists some $x_0 \in \Omega$ such that $\nu_x = \nu_{x_0}$ for almost every $x \in \Omega$ then we say that $\nu$ is a homogeneous Young measure and often identify the family $\{\nu_x\}$ with the single measure $\nu_{x_0}$ if there is no risk of confusion.
\end{theorem}

\begin{proposition}[see \cite{Pedregal97}]\label{propTranslatedYM}
If $\{V_j\}$ generates an oscillation Young measure $\nu$ and if $W_j \rightarrow W$ in measure, then $\{V_j + W_j\}$ generates the translated Young measure
$$\widetilde{\nu}_x := \delta_{W(x)} \ast \nu_x,$$
where
$$\langle \delta_{U} \ast \mu, \varphi \rangle = \langle \mu, \varphi(\cdot + U) \rangle$$
for $U \in \RRnm$ and $\varphi \in C_0(\RRnm)$. In particular, if $W_j \rightarrow 0$ in measure, then $\{V_j + W_j\}$ still generates $\nu$. Similarly, if $\| V_j - W_j \|_p \rightarrow 0$ for some $p \in [1,\infty]$ then both $V_j$ and $W_j$ generate the same Young measure. 
\end{proposition}

\subsection{Decomposition and localisation}\label{sectionDecomposition}
First, following the approach from Kristensen's lecture notes (see \cite{KristensenNotes}), we show how to decompose a given weakly convergent sequence of $\ba$-gradients into a $p$-equiintegrable oscillation part supported away from the boundary, and a concentration part that converges to $0$ in measure.

In this section we denote by $\fullgrad u$ the full gradient of $u$ given by $\fullgrad u := (\partial^\alpha u)_{\langle \alpha, \bamo \rangle \leq 1}$ and by $\lowgrad u$ its lower gradient $\lowgrad u := (\partial^\alpha u)_{\langle \alpha, \bamo \rangle < 1}$, so that $\fullgrad u = \lowgrad u \oplus \ggrad u$. We also set $d := \left| \left\{ \alpha \colon \langle \alpha, \bamo \rangle \leq 1 \right\} \right|$ so that for $u \colon \Omega \to \RRn$ we have $\fullgrad u \colon \Omega \to \RRdn$ and $\lowgrad u \colon \Omega \to \RR^{(d-m) \times n}$.

\begin{defi}
For $p \in (1, \infty)$ we say that a family of functions $\{ V_k \} \subset \LL^p(\Omega, \RR^n)$ is $p$-equiintegrable if the family $\{ |V_k|^p \} \subset \LL^1(\Omega; \RR)$ is equiintegrable in the usual sense.
\end{defi}

\begin{lemma}\label{lemmaZeroBoundary}
Let $\Omega \subset \RRN$ be a bounded open Lipschitz set satisfying the weak $\ba$-horn condition. Then, for any sequence $u_j \weakConv 0$ in $\WWap(\Omega)$, there exists a sequence $v_j \in \Ccinfty(\Omega)$ such that the sequence $(\fullgrad u_j - \fullgrad v_j)$ converges to $0$ in measure. In particular, if $\ggrad u_j$ (or, equivalently, $\fullgrad u_j$) generates some oscillation Young measure $\nu$ then $\ggrad v_j$ (respectively $\fullgrad v_j$) also generates $\nu$. Furthermore, if $\{ \fullgrad u_j \}$ is $p$-equiintegrable then so is $\{\fullgrad v_j\}$.
\end{lemma}
\begin{proof}
Take an increasing family of smooth, open sub-domains $\Omega_k \Subset \Omega_{k+1} \Subset \Omega$ (here $\Subset$ denotes compact inclusion) with $\bigcup_{k=1}^{\infty} \Omega_k = \Omega$. Fix a family of cut-off functions $\phi_k \in \Ccinfty(\Omega; [0,1])$ with $\phi_k \equiv 1$ on $\Omega_k$ and denote $M_k := ||\fullgrad \phi_k||_{\LLinfty(\Omega)}  < \infty$.

For any $k, j$ we may write $u_j = \phi_k u_j + (1-\phi_k) u_j$, and the goal is to show that the second term is small. We have
$$ \intOmega |\fullgrad ((1-\phi_k)u_j)| \dd x = \int_{\Omega \setminus \Omega_k} |\fullgrad ((1-\phi_k)u_j)| \dd x,$$
because on $\Omega_k$ the integrand is identically equal to $0$. Differentiating the product we distinguish between the case where all the derivatives fall on $u_j$ and the one where we also differentiate $(1-\phi_k)$, which yields
$$ \intOmega |\fullgrad ((1-\phi_k)u_j)| \dd x \leq  \int_{\Omega \setminus \Omega_k} (1-\phi_k) |\fullgrad u_j| \dd x +  \int_{\Omega \setminus \Omega_k} M_k |\lowgrad u_j| \dd x \ldotp$$

Since the family $\{|\fullgrad u_j|\}$ is bounded in $\LL^p$, it is uniformly integrable in $\LL^1$. Adding the fact that $|1-\phi_k| \leq 1$ and $|\Omega \setminus \Omega_k| \to 0$ we deduce that the first term in our inequality converges to $0$ with $k \to \infty$, uniformly in $j$. On the other hand, Lemma \ref{lemmaCompactEmbedding} shows that the lower gradients $\lowgrad u_j$ converge to $0$ strongly in $\LL^p$, thus in particular in $\LL^1$. Therefore, there exists a sequence $k_j \to \infty$ such that $\int_{\Omega \setminus \Omega_{k_j}} M_{k_j} |\lowgrad u_j| \dd x  \to 0$ with $j \to \infty$. Combining the two we see that $\fullgrad((1-\phi_{k_j})u_j \to 0$ strongly in $\LL^1$. 

Repeating the above reasoning with $\LL^1$ norm replaced by $\LL^p$ we get
\begin{equation}\label{eqCutOffArgument} \intOmega |\fullgrad ((1-\phi_k)u_j)|^p \dd x \leq  C\int_{\Omega \setminus \Omega_k} (1-\phi_k)^p |\fullgrad u_j|^p \dd x +  C\int_{\Omega \setminus \Omega_k} M_k^p |\lowgrad u_j|^p \dd x ,
\end{equation}
with some absolute constant $C$. The first term is now bounded uniformly in $k$ and $j$, whereas for the second one we may select a sequence ${k'}_j$ such that it converges to $0$. Hence, adjusting the first sequence $k_j$ (i.e., slowing it down if necessary) we obtain that $(1-\phi_{k_j})u_j$ is bounded in $\WWap$.

Combining the two we deduce that $(1 - \phi_{k_j})u_j \weakConv 0$ in $\WWap$, since the sequence is bounded and the only possible limit is $0$, as the full gradient converges to $0$ in $\LL^1$. By construction, the sequence also converges to $0$ in measure, thus setting $v_j := \phi_{k_j} u_j$ ends the proof of the first part of the statement. 

For equiintegrability it is enough to notice that if $\{ \fullgrad u_j\}$ is $p$-equiintegrable then the first term in \eqref{eqCutOffArgument} converges to $0$ with $k \to \infty$, uniformly in $j$ (as $(1 - \phi_k)$ is bounded in $\LL^{\infty}$ and $|\Omega \setminus \Omega_k| \to 0$). Thus, in this case, $\fullgrad u_j - \fullgrad v_j$ converges to $0$ strongly in $\LL^p$, which proves $p$-equiintegrability of $\fullgrad v_j$.
\end{proof}

The following result is the key point of this section. It shows that (up to a subsequence) one may decompose a weakly convergent sequence into a $p$-equiintegrable part that carries the oscillation and a part converging to $0$ in measure, which carries the concentration.
\begin{proposition}\label{propDecomposition}
Let $\Omega \subset \RRN$ be a bounded open Lipschitz domain satisfying the weak $\ba$-horn condition and let $p \in (1, \infty)$. Suppose that $u_j \weakConv u$ in $\WWap(\Omega; \RRn)$.
Then, there exists a subsequence $u_{j_k}$ and sequences $\{ g_k \} \subset \Ccinfty(\Omega; \RRn)$ and $\{b_k\} \subset \WWap(\Omega; \RRn)$, both weakly convergent to $0$ in $\WWap(\Omega; \RRn)$, and such that the family $\fullgrad g_k$ is $p$-equiintegrable, $\fullgrad b_k \to 0$ in measure, and $u_{j_k} = u + g_k + b_k$. 
\end{proposition}
\begin{proof}
By considering $\{u_j - u\}$ instead of $\{u_j\}$ we reduce to the case $u \equiv 0$. Furthermore, taking a subsequence if necessary, we may assume that $\fullgrad u_j$ generates some oscillation Young measure $\nu$. Lemma \ref{lemmaZeroBoundary} shows that we may also take $u_j = v_j + b^1_j$ with $b^1_j \to 0$ in measure and $v_j \in \WWapn(\Omega)$. Thus, in what follows we focus on decomposing $v_j$, remembering that $\fullgrad v_j$ generates the same Young measure as $\fullgrad u_j$.

For $l \in \naturals$ recall that the standard truncation $\tau_l \colon \RRdn \to \RRdn$ is given by
\begin{equation*}
\tau_{l}(W) := 
\begin{cases}
W \quad \text{if } |W| \leq l,\\
l \frac{W}{|W|} \quad \text{if } |W| > l \ldotp
\end{cases}
\end{equation*}
Since the truncation is bounded and continuous we get
$$ \lim_{l \to \infty} \lim_{j \to \infty} \int_{\Omega} |\tau_l(\fullgrad v_j)|^p \dd x = \lim_{j \to \infty} \int_{\Omega} \int_{\RRdn} |\tau_l(W)|^p \dd \nu_x \dd x = \int_{\Omega} \int_{\RRdn} |\cdot|^p \dd \nu_x \dd x,$$
where the first equality is due to Theorem \ref{thmFToYM} and the second one is an application of the Monotone Convergence Theorem. We infer that one can extract a sequence $j_l \to \infty$ such that
$$\lim_{l \to \infty} \int_{\Omega} |\tau_l(\fullgrad v_{j_l})|^p \dd x = \int_{\Omega} \int_{\RRdn} |\cdot|^p \dd \nu_x \dd x.$$
Note that $\LL^p$ boundedness of $\fullgrad v_{j_l}$ implies equiintegrability in $\LL^q$ for $q \in [1,p)$, hence $\tau_l(\fullgrad v_{j_l}) - \fullgrad v_{j_l} \to 0$ strongly in $\LL^q$, and thus the two generate the same oscillation Young measure $\nu$. This, paired with the last equality and Theorem \ref{thmFToYM}, implies that the family $\{\tau_l(\fullgrad v_{j_l}\}$ is $p$-equiintegrable.

In this way we have constructed a sequence $w_l := \tau_l(\fullgrad v_{j_l})$ that is $p$-equiintegrable and generates $\nu$. The last thing we need to take care of is the fact that $w_l$ is not necessarily a full gradient of a $\WWapn$ function. To remedy that, extend $w_l$ by $0$ to the whole of $\RRN$ and do the same with $v_{j_l}$, keeping the same notation for the extensions. Since we had $v_{j_l} \in \WWapn(\Omega)$, its extension by $0$ is in the space $\WWap(\RRN)$.
Apply the canonical projection $P_{\ba}$ to $w_l$ to get the decomposition
$$w_l = \fullgrad g_l + r_l \ldotp$$
We have 
\begin{equation*}
\begin{aligned}
\|r_l\|_{\LL^q} = & \| w_l - \fullgrad g_l \|_{\LL^q} \leq 
\| w_l - \fullgrad v_{j_l} \|_{\LL^q} + \|\fullgrad g_l - \fullgrad v_{j_l} \|_{\LL^q} \\
= &  \| w_l - \fullgrad v_{j_l} \|_{\LL^q} + \| P_{\ba}(w_l - \fullgrad v_{j_l}) \|_{\LL^q} \leq C\| w_l - \fullgrad v_{j_l} \|_{\LL^q} \to 0,
\end{aligned}
\end{equation*}
where we have used $P_{\ba}(\fullgrad v_{j_l}) = \fullgrad v_{j_l}$ and the $\LL^q$ continuity of $P_{\ba}$. 
In particular, this gives $r_l \to 0$ in measure.
Furthermore, Lemma \ref{lemmaProjection} shows that $p$-equiintegrability of $\{ w_l \}$ yields the same for $\{ \fullgrad g_l \}$. Lastly, we restrict $g_l$ to $\Omega$ to get $g_l \in \WWap(\Omega)$ and apply the cut-off argument from Lemma \ref{lemmaZeroBoundary} to end the proof.
\end{proof}
The following is a simple and useful corollary of the above.
\begin{corollary}\label{corOscillationEquiintegrable}
Let $\Omega$ be a bounded open Lipschitz domain satisfying the weak $\ba$-horn condition. Let $\nu$ be a $\WWap$-gradient oscillation Young measure on $\Omega$. Then there exists a sequence $u_j \in \WWap(\Omega)$ generating $\nu$ and such that the family $\{\ggrad u_j\}$ is $p$-equiintegrable. Furthermore, if the barycentre of $\nu$ is $0$ at all points of $\Omega$, then the functions $u_j$ may be chosen in the space $\Ccinfty(\Omega)$.
\end{corollary}
The final technical ingredient of this subsection is the following localisation result. We omit the proof and instead refer the reader to \cite{ProsinskiThesis} where a full argument may be found. It is simply a matter of following Kristensen's approach (see \cite{KristensenNotes}) and replacing the technical tools such as decomposition of generating sequences, polynomial approximations, and scalings, by their mixed smoothness counterparts stated and proven above.
\begin{proposition}\label{propLocalisation}
Let $\Omega \subset \RRN$ be an open and bounded domain. 
Fix $1 < p < \infty$ and let $\nu = (\nu_x)_{x \in \Omega}$ be a $\WWap$-gradient Young measure on $\Omega$. Then for $\mathcal{L}^n$-a.e. $x_0 \in \Omega$, $(\nu_{x_0})_{y \in Q}$ is a homogeneous $\WWap$-gradient Young measure on the unit cube $Q$. Its barycentre is $\overline{\nu_{x_0}} \indyk_{Q}$.
\end{proposition}

\subsection{$\ba$-quasiconvexity}\label{sectionaQC}
This section aims to develop the mixed smoothness equivalent of quasiconvexity introduced first by Morrey in \cite{Morrey52} (see also \cite{Dacorogna07} for a broad list of references). For now, we present only a few background results that we will need to establish the dual characterisation of $\WWap$-gradient Young measures. We will revisit $\ba$-quasiconvexity in more detail in the next section, when we study lower semicontinuity properties of functionals.

\begin{defi}
We say that a function $g \colon \RRnm \to [-\infty, \infty)$ is $\ba$-quasiconvex if for every $V \in \RRnm$ one has
$$ g(V) \leq \inf_{u \in \Ccinfty(Q;\RRn)} \dashint_Q g(V + \ggrad u(x)) \dd x \ldotp$$
\end{defi}
For functions that need not be $\ba$-quasiconvex we introduce the following:

\begin{defi}
For a measurable function $g \colon \RRnm \rightarrow [-\infty, \infty)$ we define the function $\QQ g \colon \RRnm \to [- \infty; \infty)$ by
\begin{equation}\label{eqDefQCEnvelope} 
\QQ g(V) := \inf \left\{ \dashint_Q g(V + \ggrad u(x)) \dd x \colon u \in \Ccinfty(Q)  \right\} \ldotp
\end{equation}
\end{defi}
The expression on the right-hand side of the above definition is often called Dacorogna's formula in the standard first order gradient case (see \cite{Dacorogna07}).
In what follows we will often refer to $\QQ g$ as the $\ba$-quasiconvex envelope of $g$, and the next lemma justifies this terminology, by asserting that $\QQ g$ is the largest $\ba$-quasiconvex function that is smaller than or equal to $g$. By definition $\QQ g \leq g$, simply by testing the definition with $u \equiv 0$ and, again from the definition, one sees immediately that any $\ba$-quasiconvex function that is no bigger than $g$ must be no bigger than $\QQ g$. Thus it only remains to show that $\QQ g$ is $\ba$-quasiconvex. This can be done following the approach of Fonseca and M\"{u}ller from \cite{FonsecaMuller99}. One only needs to use the anisotropic scaling we introduced earlier instead of the classical isotropic one. Thus, for the sake of brevity, we simply state the result and refer the reader to \cite{ProsinskiThesis} for the full argument.
\begin{lemma}\label{lemmaQgIsQC}
For a continuous function $g \colon \RRnm \to \RR$ we have
$$ \QQ (\QQ g) = \QQ g,$$
that is, $\QQ g$ is indeed $\ba$-quasiconvex.
\end{lemma}

\subsection{Topological structure of the space of oscillation $\WWap$-gradient Young measures}\label{sectionTopologicalYM}
Our next aim is to obtain a dual characterisation of oscillation $\WWap$-gradient Young measures in terms of $\WWap$-quasiconvex functions. The root of the principal results (Theorems \ref{thmCharacterizationOfYM} and \ref{thmCharacterisationYMnonhom}) goes back to the seminal work of Kinderlehrer and Pedregal (see \cite{KinderlehrerPedregal91} and \cite{KinderlehrerPedregal94}). Our approach, however, is based on \cite{FonsecaMuller99} and the proofs we give are adaptations of the techniques presented therein --- for the sake of brevity we focus on the main points that require adaptation to our mixed smoothness setting, and refer the reader to \cite{ProsinskiThesis} for the full arguments.

In the following we work with the space $\calE$ defined as
$$ \calE := \left\{ g \in C(\RRnm) \colon \lim_{|W| \to \infty} \frac{g(W)}{1+|W|^p} \text{ exists in } \RR \right\},$$
and equipped with the norm 
$$ \| g \|_{\calE} := \sup_{W \in \RRnm} \frac{|g(W)|}{1+|W|^p} \ldotp$$
It is not difficult to see that $\calE$ is a separable Banach space. Furthermore, the space of probability measures with finite $p$-th moment is a subset of $\calE^*$ through the pairing
$$ \langle \nu, g \rangle := \int_{\RRnm} g(W) \dd \nu(W),$$
for $\nu \in \{ \mu \in \Probability(\RRnm) \colon \intRRnm |W|^p \dd \nu(W) < \infty \}$ and $g \in \calE$. In particular, the space of homogeneous $\WWap$-gradient Young measures is a subset of $\calE^*$.
For future use, we note the following technical result.

\begin{lemma}\label{lemmaElementaryScaleAndTile}
Fix a bounded open Lipschitz set $\Omega \subset \RRN$ and a function $\phi \in \WWapn(Q)$. Suppose that for every $j \in \naturals$ we have fixed a countable family $\{Q^j_i\}_i = \{Q_{x^j_i}(r^j_i)\}_i$ of pairwise disjoint anisotropic boxes contained in $\Omega$ and with $\lim_{j \to \infty} \sup_i r^j_i = 0$. Assume furthermore that $\lim_{j \to \infty} |\Omega \setminus \bigcup_i Q^j_i| = 0$. Define
$$ \phi_j := \sum_i r^j_i \phi((r^j_i)^{-1} \scale (x - x^j_i)).$$
Then  $\phi_j \weakConv 0$ weakly in $\WWapn(\Omega)$. The sequence $\ggrad \phi_j$ is $p$-equiintegrable and generates the homogeneous Young measure $\ggrad \phi_{\#} \left( \frac{\LebesgueN \measurerestr Q}{|Q|} \right)$.
\end{lemma}
\begin{proof}
First of all note that the function $\phi_j$ is well-defined as an element of $\WWapn(\Omega)$. To see this observe that for every $\langle \beta, \bamo \rangle \leq 1$ and every $i$ we have
$$ \| \partial^\beta  [ r^j_i \phi((r^j_i)^{-1} \scale x)] \|_{\LLp(Q^j_i)} = (r^j_i)^{1 - \langle \beta, \bamo \rangle} \frac{|Q^j_i|}{|Q|}  \| \partial^\beta \phi \|_{\LLp(Q)}.$$
This also shows that the sequence $\phi_j$ converges strongly to $0$ in $\LLp(\Omega)$ and that it is bounded in $\WWapn(\Omega)$, thus converges weakly to $0$ in that space. The $p$-equiintegrability of $\{ \ggrad \phi_j\}$ follows from the fact that for any $M > 0$ we have
$$ \sup_j \int_\Omega |\ggrad \phi_j|^p \indyk_{|\ggrad \phi_j| > M} \dd x = \frac{|\Omega|}{|Q|} \int_Q  |\ggrad \phi|^p \indyk_{|\ggrad \phi| > M} \dd x,$$
which we get by a simple change of variables on each $Q^j_i$.

To show that $\ggrad \phi_j$ generates the desired Young measure let us fix arbitrary functions $f \in \Ccinfty(\Omega)$ and $g \in \calE$. Then
$$ \int_\Omega f(x) g(\ggrad \phi_j(x)) \dd x = \sum_i \int_{Q^j_i} f(x) g(\ggrad \phi_j(x)) \dd x.
$$
Since $f$ is Lipschitz we may write
\begin{equation*}
\begin{aligned} 
\left| \sum_i \int_{Q^j_i} (f(x) - f(x^j_i)) g(\ggrad \phi_j(x))  \dd x \right| \leq & \epsilon_j  \left| \sum_i \int_{Q^j_i}  g(\ggrad \phi_j(x))  \dd x \right|  \\
\leq & \epsilon_j \frac{|\Omega|}{|Q|} \int_Q \left| g(\ggrad \phi) \right| \dd x, 
\end{aligned}
\end{equation*}
where $\epsilon_j = \epsilon_j(f, \sup_i r^j_i) \to 0$ as $j \to \infty$. Changing variables on each $Q^j_i$ we may write that
$$  \sum_i \int_{Q^j_i}  f(x^j_i) g(\ggrad \phi_j(x))  \dd x  = \sum_i \frac{|Q^j_i|}{|Q|} f(x^j_i) \int_Q g(\ggrad \phi(y)) \dd y.$$
Note that the term $\sum_i |Q^j_i| f(x^j_i)$ is a Riemann sum of $f$ on the anisotropic boxes $Q^j_i$, from which we deduce that it converges, as $j \to \infty$, to the integral $\int_\Omega f(x) \dd x$. Thus, finally
$$ \lim_{j \to \infty} \int_\Omega f(x) g(\ggrad \phi_j(x)) \dd x = \left( \int_\Omega f(x) \dd x \right) \left( \dashint_Q g(\ggrad \phi(y)) \dd y \right),$$
and a standard density argument on $f$ and $g$ ends the proof.
\end{proof}
Since every bounded open set may be covered, up to a subset of measure $0$, with arbitrarily small anisotropic boxes we get the following:
\begin{corollary}\label{corElementaryHomogenous}
For any $\phi \in \WWapn(Q)$ and any bounded open Lipschitz set $\Omega \in \RRN$ the measure given by $\ggrad \phi_{\#} \left( \frac{\LebesgueN \measurerestr Q}{|Q|} \right)$ is a homogeneous $\WWap$-gradient Young measure on $\Omega$. 
\end{corollary}

We denote the space of homogeneous $\WWap$-gradient Young measures on $Q$ with barycentre $0$ by $\hpn(Q)$.

\begin{lemma}\label{lemmahpnIsConvex}
The set $\hpn$ is convex.
\end{lemma}
\begin{proof}
Fix $\nu, \mu \in \hpn$ and $\theta \in (0,1)$. Let $v_j, u_j \subset \Ccinfty(Q)$ satisfy $v_j \weakConv 0$, $u_j \weakConv 0$ in $\WWapn(Q)$ with $\ggrad v_j$, $\ggrad u_j$ generating $\nu$ and $\mu$ respectively.
Pick a sequence of smooth cut-off functions $\phi_i \in \Ccinfty(Q;[0,1])$ with $\phi_i \nearrow \indyk_{(-1,-1+2\theta) \times Q^{N-1}}$ and such that $| \{ \phi_i \not= \indyk_{(-1,-1+2\theta) \times Q^{N-1})} \} | \to 0$ with $i \to \infty$. Here $Q^{N-1}$ is the $(N-1)$ dimensional cube. Define $w_j^i := \phi_i v_j + (1-\phi_i) u_j$. Then
$$ \ggrad w_j^i = \phi_i \ggrad v_j + (1-\phi_i) \ggrad u_j + R(i,j),$$
where $R(i,j)$ is the remainder, consisting of the terms where we put some of the derivatives on $\phi_i$ or $(1-\phi_i)$. Due to Lemma \ref{lemmaCompactEmbedding} the lower order derivatives of $v_j$ and $u_j$ converge stronly to $0$ in $\LL^p$, thus
$$ \| R(i,j) \|_{\LL^p} = \| \ggrad w_j^i - (\phi_i \ggrad v_j + (1-\phi_i) \ggrad u_j \|_{\LL^p} \leq C(\phi_i) (\| \lowgrad v_j \|_{\LL^p} + \|\lowgrad u_j\|_{\LL^p}) \to 0$$
with $j \to \infty$ for any fixed $i$ (here $C(\phi_i)$ is a finite constant that depends on the function $\phi_i$). Therefore, we may select a subsequence $j(i) \to \infty$ as $i \to \infty$ such that
$$ \| \ggrad w_{j(i)}^i - (\phi_i \ggrad v_{j(i)} + (1-\phi_i) \ggrad u_{j(i)}) \|_{\LL^p} \to 0$$
as $i \to \infty$. Let $w_i := w_{j(i)}^i \in \Ccinfty(Q)$. It is straightforward to see that $\ggrad w_i$ generates $\indyk_{(-1,-1+2\theta) \times Q^{N-1}} \nu + \indyk_{(-1+2\theta, 1) \times Q^{N-1}} \mu$ as its oscillation Young measure. This clearly holds for $ \indyk_{(-1,-1+2\theta) \times Q^{N-1}} \ggrad v_{j(i)} + \indyk_{(-1+2\theta, 1) \times Q^{N-1}} \ggrad u_{j(i)}$ and, by construction, the difference of this sequence and $(\phi_i \ggrad v_{j(i)} + (1-\phi_i) \ggrad u_{j(i)})$ converges to $0$ in measure, whilst the difference of $\ggrad w_i$ and the latter sequence converges to $0$ in $\LL^p$.
Finally, $w_i$ is compactly supported in $Q$, so we may extend it periodically to $\RRN$ and define
$$ w^k_i(x) := R_k^{-1} w_i(R_k \scale x),$$
where $R_k := k^{a_1 \cdot \ldots \cdot a_N}$. Then, by Lemma \ref{lemmaElementaryScaleAndTile}, for all $\phi \in C_0^{\infty}$ and $\psi \in \calE$ we have
\begin{equation*}
\lim_{i \rightarrow \infty} \lim_{k \rightarrow \infty} \dashint_Q \phi(x) \psi(\ggrad w^k_i(x)) \dd x = \lim_{j \rightarrow \infty} \dashint_Q \phi(x) \left( \dashint_Q \psi(\ggrad w_i(y)) \dd y \right) \dd x \ldotp
\end{equation*}
Finally, since we have already identified the measure generated by $\ggrad w_i$, we may write
$$  \lim_{i \rightarrow \infty} \lim_{k \rightarrow \infty} \dashint_Q \phi(x) \psi(\ggrad w_i^k(x)) \dd x = \dashint_Q \phi(x) \dd x (\theta \langle \nu, \psi \rangle + (1 - \theta) \langle \mu, \psi \rangle ) \ldotp$$
A standard density argument and a diagonal extraction in the separable spaces $\LL^1 (Q)$ and $C_{0}(\RRnm)$ let us obtain a subsequence $\widetilde{w}_l \subset \{w^k_i\} \subset C^{\infty} $ with $\ba$-gradients generating the measure $\theta \nu + (1-\theta)\mu $, which ends the proof.
\end{proof}

\begin{lemma}\label{lemmahpoclosed}
The set $\hpn$ is relatively closed in $\Probability(\RRnm) \cap \{\mu \colon \intRRN |W|^p \dd \mu < \infty \}$ with respect to the weak* topology on $\calE^*$.
\end{lemma}
\begin{proof}
Fix an arbitrary $\nu \in \overline{\hpn}^{E^*} \cap \Probability(\RRnm)$ and let $\{f_i\} \subset C^{\infty}(Q)$, $\{g_j\} \subset C_c^{\infty} (\RRnm)$ be countable dense subsets of $L^1(Q)$ and $C_0(\RRnm)$ respectively.
Take also $f_0(x) \equiv 1$ and $g_0(W) := \left|W\right|^p$. By definition of the weak* topology, for any fixed $g \in  C_0(\RRnm)$ there exists a $\nu_k \in \hpn$ with
$$ \left|\langle \nu - \nu_k, g \rangle\right| < \frac{1}{2k}\ldotp$$
Through a diagonal argument we may ensure that this is satisfied simultaneously for any finite set of $g$'s, i.e., for any $k \in \naturals$ there exists a $\nu_k \in \hpn$ such that
$$ \left|\langle \nu - \nu_k, g_j \rangle\right| < \frac{1}{2k}, \text{ for all } j \in \{ 0, 1, \ldots , k \} \ldotp$$
Since $\nu_k \in \hpn$ we may find a sequence $\{ w^k_j \} \subset \WWapn(Q)$ with $\ba$-gradients generating $\nu_k$. 

Theorem \ref{thmFToYM} implies that, for any $g \in C_c^{\infty} (\RRnm)$, we have $g(\ggrad w^k_j) \weakConv \langle g,\nu_k \rangle$ in $\LL^1(Q)$. Another diagonal extraction and the triangle inequality let us establish existence of a sequence $\{ w_k \} \subset \WWapn(Q)$ such that
\begin{equation}\label{eq42Fonseca}
\left| \langle \nu,g_j \rangle \int_Q f_i \dd x - \int_Q f_i g_j(\ggrad w_k) \dd x \right| < \frac{1}{k}, \text{ for all } 0 \leq i,j \leq k,
\end{equation}
as all $f_i$'s are smooth and therefore bounded, so that they are admissible test functions for weak convergence in $\LL^1$. 

Setting $i = j = 0$ shows that $\{\ggrad w_k\}$ is bounded in $\LL^p(Q)$, therefore we may find a subsequence generating some $\WWap$-gradient Young measure $\mu$. For notational simplicity we assume that the entire sequence generates $\mu$. 
From \eqref{eq42Fonseca} and the fact that $g_j(\ggrad w_k) \weakConv \langle \mu, g_j \rangle$ in $\LL^1$ as $k \to \infty$ we infer that
$$ \langle \nu,g_j \rangle \int_Q f_i \mathop{dx} =  \int_Q f_i \langle \mu, g_j \rangle \mathop{dx}, \text{ for all } i, j \ldotp$$
By density of $f_i$ in $\LL^p(Q)$ we deduce that $\langle \mu, g_j \rangle = \langle \nu, g_j \rangle$ in $(\LL^p)^*$. In particular, they are equal almost everywhere, so that for almost every $x \in Q$ we have $\langle \mu_x, g_j \rangle = \langle \nu, g_j \rangle $ for all $j$. By density of $\{g_j\}$ in $C_0(\RRnm)$ we deduce that $\mu_x = \nu$ for almost all $x$, which shows that $\mu$ is in fact homogeneous so $\nu = \mu \in \hpn$, which ends the proof.
\end{proof}

Because bounded continuous functions are a subset of $\calE$ we immediately get:
\begin{lemma}\label{lemmaPortmanteau}
If a sequence $\{ \nu_j \} \subset \Probability(\RRnm) \cap \calE^*$ converges to some $\nu \in \Probability(\RRnm) \cap \calE^*$ in the space $\calE^*$ then it also converges in the sense of weak convergence of probability measures. In particular, by the portmanteau theorem, we have
$$ \lim_{j \to \infty} \int_{\RRnm} g \dd \nu_j = \int_{\RRnm} g \dd \nu$$
for all bounded and continuous functions $g$, and
$$ \liminf_{j \to \infty} \int_{\RRnm} g \dd \nu_j \geqslant \int_{\RRnm} g \dd \nu$$
for all lower semicontinuous functions $g$ bounded from below.
\end{lemma}

\subsection{Dual characterisation of oscillation $\WWap$-gradient Young measures}\label{sectionDualYM}
We are finally ready to show the two main results of this section: 

\begin{theorem}\label{thmCharacterizationOfYM}
A probability measure $\mu \in \Probability(\RRnm)$ is a homogeneous oscillation $\WWap$-gradient Young measure with mean $W_0$ if and only if $\mu$ satisfies $\intRRnm W \dd \mu(W) = W_0$, $\intRRnm \left|W\right|^p \dd \mu (W) < \infty$ and
$$ \intRRnm g(W) \dd \mu (W) \geq \QQ g(W_0) $$
for all $g \in \calE$.
\end{theorem}

\begin{theorem}\label{thmCharacterisationYMnonhom}
Fix a bounded open Lipschitz domain $\Omega$ satisfying the weak $\ba$-horn condition. Let $\{ \nu_x\}_{x \in \Omega}$ be a weak* measurable family of probability measures on $\RRnm$. Then there exists a $\WWap(\Omega)$-bounded sequence $\{v_n\} \subset \WWap(\Omega)$ with $\{ \ggrad v_n \}$ generating the oscillation Young measure $\nu$ if and only if the following conditions hold:

i) there exists $v \in \WWap(\Omega)$ such that
$$ \ggrad v(x) = \langle \nu_x, \id \rangle \text{ for a.e. } x \in \Omega;$$

ii) 
$$ \int_{\Omega} \intRRnm |W|^p \dd \nu_x(W) \dd x < \infty;$$

iii) for a.e. $x \in \Omega$ and all $g \in \calE$ we have
$$ \langle \nu_x, g \rangle \geq \QQ g( \langle \nu_x, \id \rangle) \ldotp$$
\end{theorem}

The proof strategy is based on Fonseca and M\"{u}ller's approach from \cite{FonsecaMuller99}. Having established all the necessary technical ingredients above, there are no major difficulties in adapting their arguments. Below we present a proof for the case of homogeneous measures and, for the sake of brevity, refer the reader to \cite{ProsinskiThesis} for the full argument in the general case. 

\begin{proof}(of Theorem \ref{thmCharacterizationOfYM})
First observe that it is enough to consider the case $W_0 = 0$, as we may always take a translation. For the proof of this case we argue by contradiction. Suppose that $\nu \in \Probability(\RRnm)$ satisfies
\begin{equation*}
\begin{cases}
\intRRnm W \dd \nu = 0, \\
\intRRnm \left|W\right|^p \dd \nu (W) < \infty, \\
\intRRnm g(W) \dd \nu(W) \geq \QQ g(0) \text{ for all } g \in \calE,
\end{cases}
\end{equation*}
but $\nu \not\in \hpn$. The $p$-th moment assumption on $\nu$ implies that $\nu \in \calE^*$. By Lemmas \ref{lemmahpnIsConvex} and \ref{lemmahpoclosed} and the Hahn-Banach separation theorem, there exist $g \in \calE$ and $\alpha \in \RR$ such that 
$$  \QQ g(0) \leq \langle \nu,g \rangle < \alpha \leq \langle \mu, g \rangle  \text{ for all } \mu \in \hpn \ldotp$$
For any $\phi \in \Ccinfty(Q)$ the measure $(\ggrad \phi)_{\#} \frac{\LebesgueN \measurerestr Q}{|Q|} $ is an element of $\hpn$ (see Corollary \ref{corElementaryHomogenous}) so that
$$ \dashint_Q g(\ggrad \phi(x)) \dd x\geq \alpha \ldotp$$
Taking infimum over all such $\phi$'s yields
$$\QQ g(0) =  \inf_{\phi \in \Ccinfty(Q)} \dashint_Q g(\ggrad \phi(x)) \dd x \geq \alpha,$$
which contradicts $\QQ g(0) \leq \langle \nu,g \rangle < \alpha$ and shows that $\nu$ is indeed in the set $\hpn$.

For the reverse implication let $\phi_k \in \Ccinfty$ be such that $\{\ggrad \phi_k\}$ is a $p$-equiintegrable sequence generating $\nu$ (which is possible thanks to Corollary \ref{corOscillationEquiintegrable}). Then for all $k$ we have
$$ \intRRnm g(\ggrad \phi_k(x)) \dd x \geq \QQ g(0)$$
by definition of $\QQ g$. On the other hand, $p$-equiintegrability of $\{\ggrad \phi_k\}$ and the growth bound on $g$ imply that we may use point iii) of Theorem \ref{thmFToYM} to deduce that
$$ \lim_{k \to \infty} \intRRnm g( \ggrad \phi_k(x))  \dd x = \intRRnm g \dd \nu,$$
as for a continuous function the boundedness from below assumption is irrelevant, since we may simply consider the positive and negative parts of $g$ separately. 
This and the previous estimate finish the proof. 
\end{proof}

\section{Coercivity}\label{sectionCoercivity}
We are now ready to move on to the core part of the paper, i.e., the study of existence of solutions to variational problems posed in Sobolev spaces of mixed smoothness. To use the Direct Method we need two ingredients: compactness of sequences of minimisers and sequential lower semicontinuity of the given functional. In this section we focus on the first part and study coercivity of the relevant functionals.
Our first goal is Theorem \ref{thmCoercivity}, which is a generalisation to the mixed smoothness framework of a recent result due to Chen and Kristensen from \cite{ChenKristensen17}, that was further improved by Gmeineder and Kristensen in \cite{GmeinederKristensen18}.

Here $F \colon \RRnm \to \RR$ is a continuous integrand satisfying the growth condition
\begin{equation}\label{eqGrowth} 
|F(X)| \leq C(|X|^p + 1)
\end{equation}
for some $C \in (0,\infty)$ and all $X \in \RRnm$. For a non-empty bounded open set $\Omega \subset \RRN$ we denote by $\Fun(\cdot,\Omega) \colon \WWap(\Omega) \to \RR$ the mapping defined by $ \Fun(u,\Omega) := \int_{\Omega} F(\ggrad u) \dd x$. To be able to properly introduce a Dirichlet problem for this functional we define, for a fixed $g \in \WWap(\RRN)$, the Dirichlet class
$$ \WWapg(\Omega) := \{ g + \phi \colon \phi \in \WWapn(\Omega) \},$$
Observe that we require the boundary datum $g$ to be defined on the whole of $\RRN$, rather than just $\Omega \subset \RRN$, thus bypassing possible trace issues. For $q \in [1,p]$ we introduce, as in \cite{ChenKristensen17}, the following notions:
\begin{defi} 
We say that $\Fun(\cdot, \Omega)$ is $\LLq$ coercive on $\WWapg(\Omega)$ if for any sequence $u_j \in \WWapg(\Omega)$ with $\| \ggrad u_j \|_{\LLq} \to \infty$ one has $\Fun(u_j, \Omega) \to \infty$.
\end{defi}

\begin{defi} 
We say that $\Fun(\cdot, \Omega)$ is $\LLq$ mean coercive on $\WWapg(\Omega)$ if for any $u \in \WWapg(\Omega)$ we have 
$$\Fun(u, \Omega) \geq C_1 \| \ggrad u\|_{\LLq}^q - C_2$$ 
for some strictly positive constants $C_1, C_2$ independent of $u$, but not necessarily of $g$.
\end{defi}

As in \cite{ChenKristensen17} (see also Proposition 3.1 in \cite{GmeinederKristensen18}), our main goal here is the following:
\begin{theorem}\label{thmCoercivity}
Let $F \colon \RRnm \to \RR$ be a continuous integrand satisfying \eqref{eqGrowth}. Then, for any $q \in [1, p]$ the following are equivalent:
\begin{enumerate}[label=\alph*)]
\item For any bounded open Lipschitz set $\Omega \subset \RRN$ and any boundary datum $g \in \WWap(\RRN)$ the functional $\Fun(\cdot,\Omega)$ is $\LLq$ coercive on $\WWapg(\Omega)$.
\item There exist a non-empty bounded open Lipschitz set $\Omega \subset \RRN$ and a boundary datum $g \in \WWap(\RRN)$ such that the functional $\Fun(\cdot,\Omega)$ is $\LLq$ coercive on $\WWapg(\Omega)$.
\item For any bounded open Lipschitz set $\Omega \subset \RRN$ and any boundary datum $g \in \WWap(\RRN)$ the functional $\Fun(\cdot,\Omega)$ is $\LLq$ mean coercive on $\WWapg(\Omega)$.
\item There exist a non-empty bounded open Lipschitz set $\Omega \subset \RRN$ and a boundary datum $g \in \WWap(\RRN)$ such that the functional $\Fun(\cdot,\Omega)$ is $\LLq$ mean coercive on $\WWapg(\Omega)$.
\item For any bounded open Lipschitz set $\Omega \subset \RRN$ and any boundary datum $g \in \WWap(\RRN)$ all $\WWapg(\Omega)$ minimising sequences for the functional $\Fun(\cdot,\Omega)$ are bounded in $\WWaqg(\Omega)$.
\item There exist a non-empty bounded open Lipschitz set $\Omega \subset \RRN$ and a boundary datum $g \in \WWap(\RRN)$ such that all $\WWapg(\Omega)$ minimising sequences for the functional $\Fun(\cdot,\Omega)$ are bounded in $\WWaqg(\Omega)$.
\item There exist a constant $c > 0$ and a point $X_0 \in \RRnm$ such that the integrand $X \mapsto F(X) - c|X|^q$ is $\ba$-quasiconvex at $X_0$.
\end{enumerate}
\end{theorem}

Simply put, the above asserts that $\LLq$ coercivity, $\LLq$ mean coercivity, and boundedness of all minimising sequences are mutually equivalent, and it is enough to check either of these on a particular choice of domain and boundary datum. Furthermore, coercivity may be characterized in terms of $\ba$-quasiconvexity of $F(\cdot) - c|\cdot|^q$ at a single point.

\begin{lemma}[see Proposition 3.1 in \cite{ChenKristensen17}]
Under the hypotheses of Theorem \ref{thmCoercivity} its point a) is equivalent to b), and point c) is equivalent to d).
\end{lemma}
\begin{proof}
The fact that a) is equivalent to b) in the first order gradient case is the content of Proposition 3.1 in \cite{ChenKristensen17}. The proof of c) being equivalent to d) is similar --- we will do this here omitting the first part, as all that needs to be changed in the argument of \cite{ChenKristensen17} is the scaling factor.

It is obvious that c) implies d), thus we only need to prove the other implication. Let $\Omega_0$ and $g_0$ be the domain and the boundary datum for which $\Fun(\cdot,\Omega_0)$ is $\LLq$ mean coercive on $\WWap_{g_0}(\Omega_0)$. Fix an arbitrary non-empty bounded open Lipschitz set $\Omega \subset \RRN$ and a boundary datum $g \in \WWap(\RRN)$. Take boxes $Q_{2r}(x_0) \Subset \Omega_0$ (for some $x_0 \in \Omega_0$) and $Q_R(0) \Supset \Omega$, and fix cut-off functions $\phi \in \Ccinfty(Q_{2r}(x_0))$ and $\rho \in \Ccinfty(Q_R(0))$ satisfying
$$ \mathbf{1}_{Q_r(x_0)} \leq \phi \leq \mathbf{1}_{Q_{2r}(x_0)}  \quad \text{ and } \quad \mathbf{1}_{\Omega} \leq \rho \leq \mathbf{1}_{Q_R(0)}.$$
Any $u \in \WWapg(\Omega)$ may be seen as a function in $\WWap(\RRN)$ if we extend it by $u = g$ outside $\Omega$, simply from our definition of $\WWapg(\Omega)$. We may then cut-off outside $Q_R(0)$ by setting $\widetilde{w} := \rho u \in \WWapn(Q_R(0))$. Define $w(x) := R^{-1} \widetilde{w}(R \scale x) \in \WWapn(Q_1(0))$ and set
$$ v(x) := (1-\phi(x))g_0(x) + r w(r^{-1} \scale (x - x_0)) \text{ for } x \in \Omega_0.$$
Clearly $v \in \WWap_{g_0}(\Omega_0)$ and we may calculate, using a simple change of variables, that
\begin{equation*}
\begin{aligned} \| \ggrad v \|_{\LLq(\Omega_0)}^q = & \int_{\Omega_0 \setminus Q_r(x_0)} |\ggrad(1-\phi(x))g_0|^q \dd x  \\
& + (\frac{r}{R})^{|\bamo|} \int_{Q_R(0) \setminus \Omega} |\ggrad(\rho g)|^q \dd x + (\frac{r}{R})^{|\bamo|} \|\ggrad u\|_{\LLq(\Omega)}^q,
\end{aligned}
\end{equation*}
so that $\| \ggrad v \|_{\LLq(\Omega_0)}^q = D_1 + D_2 \|\ggrad u\|_{\LLq(\Omega)}^q,$ where $D_i$ do not depend on $u$ and $D_2 > 0$. In the same way we get
\begin{equation*}
\begin{aligned}
\int_{\Omega_0} F(\ggrad v) \dd x = & \int_{\Omega_0 \setminus Q_r(x_0)} F(\ggrad(1-\phi(x))g_0) \dd x  \\
& + (\frac{r}{R})^{|\bamo|} \left( \int_{Q_R(0) \setminus \Omega} F(\ggrad(\rho g)) \dd x 
+   \int_\Omega F(\ggrad u) \dd x \right).
\end{aligned}
\end{equation*}
Thus, $\Fun(u, \Omega)$ differs from $\Fun(v,\Omega_0)$ by a constant and a scaling. Since the same is true for the $\LLq$ norms of $u$ and $v$ we easily conclude that $\Fun(u, \Omega) \geq -c_1 + c_2 \|u\|_{\LLq(\Omega)}^q$ for some $c_i > 0$, which ends the proof.
\end{proof}

\begin{defi}
Let $F \colon \RRnm \to \RR$ be a continuous integrand satisfying the growth condition \eqref{eqGrowth}. For $q \in [1,p]$ and a non-empty bounded and open Lipschitz subset $\Omega \subset \RRN$ we define, for $t \geq 0$,
$$ \theta(t)  = \theta^{\Omega}_q(t) := \inf \left\{ \dashint_{\Omega} F(\ggrad \phi) \dd x \colon \phi \in \WWapn(\Omega, \RRn), \, \dashint_{\Omega} |\ggrad \phi|^q \dd x \geq t \right\}\ldotp$$
\end{defi}

\begin{lemma}[see Lemma 2.1 in \cite{ChenKristensen17}]\label{lemmaChangeDomains}
For any open, bounded, and non-empty domains $\omega, \Omega \subset \RRN$ and any function $\phi \in \Ccinfty(\Omega)$ there exists a function $\psi \in \Ccinfty(\omega)$ such that the pushforward measures $ [\ggrad \phi]_{\#} \left( \frac{ \mathcal{L}^N \restriction \Omega}{\mathcal{L}^N(\Omega)} \right)$ and $ [\ggrad \psi]_{\#} \left( \frac{ \mathcal{L}^N \restriction \omega}{\mathcal{L}^N(\omega)} \right)$ are equal.
\end{lemma}
\begin{proof}
This is proven using a standard exhaustion argument, as in Lemma 2.1 in \cite{ChenKristensen17}. We omit the proof here, as the only difference with the aforementioned paper is that one needs to change the scaling to our anisotropic variant.
\end{proof}

\begin{corollary}\label{corQuasiconvexityIndependentOfDomain}
The unit cube $Q$ in the definition of the $\ba$-quasiconvex envelope may be replaced by any other domain without changing the resulting envelope. That is, for any open bounded Lipschitz domain $\Omega \subset \RRN$ we have
$$\QQ F(\cdot) = \inf_{\phi \in \Ccinfty(\Omega)} \dashint_{\Omega} F(\cdot + \ggrad \phi(x)) \dd x.$$
\end{corollary}

The next result asserts that the quasiconvex envelope of an integrand of $p$-growth is either degenerate and equal to $-\infty$ everywhere, or it inherits the original integrand's growth bounds. Let us remark that in the context of classical quasiconvexity this is typically shown using rank-one convexity, which is lacking in our case. Therefore, the usual proofs do not generalise to the mixed smoothness setting, hence we present a new argument that we believe is, in a sense, more natural and straightforward. 
\begin{lemma}\label{lemmaQFGrowth}
Suppose that $F$ is a continuous integrand satisfying \eqref{eqGrowth}. Then $\QQ F$ is either identically equal to $- \infty$ or it is real-valued everywhere and satisfies the growth condition \eqref{eqGrowth}, albeit possibly with a larger constant.
\end{lemma}
\begin{proof}
Assume for contradiction that $\QQ F$ is finite at some point $X_0 \in \RRnm$, but it does not satisfy the $p$ growth condition. Let $D$ be the constant with which $F$ satisfies the $p$ growth assumption \eqref{eqGrowth}. Since $\QQ F \leq F$, it follows that the $p$ growth bound must fail for the negative part of $\QQ F$. Thus, there exists a sequence of points $X_j \in \RRnm$ such that $\QQ F(X_j) < -j(|X_j|^p + 1)$. From this and Corollary \ref{corQuasiconvexityIndependentOfDomain} we infer existence of a sequence of functions $\phi_j \in \Ccinfty(Q_{1/2})$ such that 
$$ \dashint_{Q_{1/2}} F(X_j + \ggrad \phi_j) \dd x < -j(|X_j|^p + 1).$$
Fix a cut-off function $\rho \in \Ccinfty(Q)$ with $\rho \equiv 1$ on $Q_{1/2}$. For each $j$ let $P_j(x) := \sum_{\langle \alpha, \bamo \rangle = 1} (X_j - X_0) x^\alpha$ and set $\psi_j := \rho P_j + \phi_j \in \Ccinfty(Q)$. Then
\begin{equation*}
\begin{aligned}
\QQ F(X_0) \leq & \liminf_{j \to \infty} |Q|^{-1} \int_Q F(X_0 + \ggrad \psi_j) \dd x \\
= & \liminf_{j \to \infty} |Q|^{-1} \int_{Q_{1/2}} F(X_j + \ggrad \phi_j) \dd x + |Q|^{-1} \int_{Q \setminus Q_{1/2}} F(X_0 + \ggrad(\rho P_j)) \dd x  \\
\leq & \liminf_{j \to \infty}  \left( \frac{-1}{2} \right) j(|X_j|^p + 1) + \dashint_Q D(|X_0 + \ggrad(\rho P_j)|^p + 1) \dd x,
\end{aligned}
\end{equation*}
where we have used the definition of $\psi_j$, the definition of $X_j$, and the $p$-growth bound on $F$ respectively. We note that $|X_0 + \ggrad (\rho P_J)|^p \leq 2^{p-1} |X_0|^p + 2^{p-1}|\ggrad (\rho P_j)|^p$, and by construction $| \partial^\alpha P_j (x) | \leq C |X_j - X_0|$ with a uniform constant $C$ for all $x \in Q$ and all $\alpha$ with $\langle \alpha, \bamo \rangle \leq 1$, thus $|\ggrad (\rho P_j)|^p \leq C|X_j - X_0|^p$, since $\rho$ is just a fixed $\Ccinfty$ function. Plugging this into our inequality we get
$$ \QQ F(X_0) \leq \liminf_{j \to \infty} \left( \frac{-1}{2} \right) j(|X_j|^p + 1) + C(|X_0|^p + |X_j - X_0|^p + 1) = - \infty,$$
which contradicts the assumption that $\QQ F(X_0)$ is finite.
\end{proof}

\begin{lemma}\label{lemmaQFusc}
If $F$ is a continuous integrand satisfying the $p$-growth condition \eqref{eqGrowth} and such that its $\ba$-quasiconvex envelope is not identically equal to $- \infty$, then the functional
$$ u \mapsto \dashint_\Omega \QQ F(\ggrad u) \dd x $$
is sequentially upper semicontinuous along sequences $u_j$ converging to a given $u$ in the $\WWap$ norm and with $\ggrad u_j \to \ggrad u$ almost everywhere.
\end{lemma}
\begin{proof}
By Lemma \ref{lemmaQFGrowth}, $\QQ F$ satisfies the $p$-growth condition \eqref{eqGrowth} as well (perhaps with a different constant $C$). Furthermore, since $X \mapsto \dashint_\Omega F(X + \ggrad \phi(x)) \dd x$ is continuous for any $\phi \in \WWapn$ we see that $\QQ F$ is a pointwise infimum of a family of continuous functions, thus $\QQ F$ is upper semicontinuous. Since the functions $C(1+ |\ggrad u_j|^p) - \QQ F(\ggrad u_j)$ are all non-negative, we get, by Fatou's lemma, that
$$ \liminf_{j \to \infty} \dashint_\Omega C(1+ |\ggrad u_j|^p) - \QQ F(\ggrad u_j) \dd x \geq \dashint_\Omega \liminf_{j \to \infty} C(1+ |\ggrad u_j|^p) - \QQ F(\ggrad u_j) \dd x.$$
Rearranging and using strong $\LLp$ convergence of $\ggrad u_j$ yields
$$  \limsup_{j \to \infty} \dashint_\Omega  \QQ F(\ggrad u_j) \dd x \leq \dashint_\Omega  \limsup_{j \to \infty} \QQ F(\ggrad u_j) \dd x.$$
Finally, $\ggrad u_j \to \ggrad u$ almost everywhere by assumption, and $\QQ F$ is upper semicontinuous, thus $ \limsup_{j \to \infty} \QQ F(\ggrad u_j) \leq \QQ F(\ggrad u)$ almost everywhere, which ends the proof.
\end{proof}

The following result justifies the omission of the underlying set $\Omega$ in our notation $\theta (t)$ by showing that the auxiliary function $\theta$ does not depend on the choice of $\Omega$. Moreover, it shows that $\theta$ only depends on $\QQ F$ rather than $F$ itself, thus showing that $F$ is $\LLq$ (mean) coercive if and only if $\QQ F$ is.
\begin{lemma}[see Lemma 3.1 in \cite{ChenKristensen17}]\label{lemmaAuxiliaryFunction}
Let $\omega, \Omega \subset \RRnm$ be non-empty bounded open Lipschitz subsets of $\RRN$, and define $\theta^{\Omega}_q(t)$ as above. Define also 
$$ \thetaqc(t)  := \inf \left\{ \dashint_{\omega} \QQ F(\ggrad \phi) \dd x \colon \phi \in \WWapn(\omega, \RRn), \, \dashint_{\omega} |\ggrad \phi|^q \dd x \geq t \right\}\ldotp$$
Then $\theta^{\Omega}_q (t) = \thetaqc (t)$ for all $t \geq 0$.
\end{lemma}
\begin{proof}
The argument follows that of \cite{ChenKristensen17} with the key differences being that, due to lack of rank-one convexity, we cannot assert continuity of $\QQ F$ (see however Lemma \ref{lemmaContinuousEnvelope}), and that our polynomial approximation requires a countable partition. 

The case $t=0$ is the content of Corollary \ref{corQuasiconvexityIndependentOfDomain}. Another easy case is when $\QQ F$ is identically equal to $- \infty$. Then one can, for example, choose two disjoint open subsets $\Omega_1, \Omega_2 \Subset \Omega$ and construct functions $\phi_i \in \Ccinfty(\Omega_i)$ and use $\phi_1$ to satisfy the restriction $\dashint_{\Omega_i} |\ggrad \phi_i|^q \dd x \geq t$, whilst using $\phi_2$ to make the integral $\int_{\Omega_2} F(\ggrad \phi_2) \dd x$ as small as one wishes, with the last point being possible thanks to the fact that $\QQ F(0) = -\infty$.
Thus, in the following we assume that $t > 0$ and that $\QQ F > - \infty$.

Let us fix $t > 0$. For any $\epsilon > 0$ we may find a function $\phi \in \Ccinfty(\Omega)$ satisfying 
$$ \dashint_\Omega |\ggrad \phi|^q \dd x \geq t \quad \text{and} \quad \thetaqc(t) + \epsilon > \dashint_\Omega \QQ F(\ggrad \phi) \dd x \ldotp$$
To see this, it is enough to observe that the existence of $\widetilde{\phi} \in \WWapn(\Omega)$ satisfying these two inequalities, the second with a smaller $\epsilon'$, is guaranteed by the definition of $\thetaqc$. To get the same with $\phi \in \Ccinfty(\Omega)$, it is enough to observe that $\widetilde{\phi}$ may be approximated in $\WWapn(\Omega)$ by a sequence $\phi_j \in \Ccinfty(\Omega)$ with $\|\ggrad \phi_j\|_{\LLq} \geq \|\ggrad \widetilde{\phi}\|_{\LLq}$ for all $j$ and $\ggrad \phi_j \to \ggrad \widetilde{\phi}$ almost everywhere. Then we simply use Lemma \ref{lemmaQFusc} and conclude.

From here it is easy to pass to a function $\psi \in \WWapn(\Omega)$ satisfying the above with $2\epsilon$ instead of $\epsilon$ in the second inequality and for which $\ggrad \psi$ is piecewise constant on a large part of $\Omega$. We simply apply Proposition \ref{propPiecewiseApprox} to $\phi \in \Ccinfty(\Omega)$ (preserving the boundary values) and obtain a sequence $\phi_j$ approximating $\phi$ in $\WWap$ and such that the measure of the complement of the set of boxes $T$, denoted $\tau_j$, on which $\ggrad \phi_j$ is constant goes to $0$ as $j$ goes to infinity. We may then rescale the sequence to satisfy the condition $\dashint_\Omega |\ggrad \phi_j|^q \dd x \geq t$, pass to a subsequence for which the gradients converge almost everywhere, and finally use Lemma \ref{lemmaQFusc} again to deduce that elements $\phi_j$ for large enough $j$'s satisfy the desired inequalities, still with an $\epsilon$. Without loss of generality assume that this holds for all $j$. 
Since $\ggrad \phi_j$ is strongly $\LLp$ convergent it is also $p$-equiintegrable, thus (as $\QQ F$ satisfies the $p$-growth bound) $\QQ F(\ggrad \phi_j)$ and $F(\ggrad \phi_j)$ are equiintegrable as well. Therefore, we may find a $j_0$ such that the measure of $\Omega \setminus \bigcup_{\tau_j} T$ is small enough so that 
$$ \left| \int_{\Omega \setminus \bigcup_{\tau_{j_0}} T} \QQ F(\ggrad \phi_j) \dd x \right| < \epsilon |\Omega| \quad \text{and} \quad \left| \int_{\Omega \setminus \bigcup_{\tau_{j_0}} T} F(\ggrad \phi_j) \dd x \right| < \epsilon |\Omega| \ldotp$$
We set $\psi := \phi_{j_0}$ and $\tau = \tau_{j_0}$.
By Corollary \ref{corQuasiconvexityIndependentOfDomain} we may, for each $T \in \tau$, find a function $\phi_T \in \Ccinfty(T)$ satisfying
\begin{equation}\label{eqDefPhiT} \dashint_T F(\ggrad \psi + \ggrad \phi_T) \dd x < \QQ F(\ggrad \psi) + \epsilon \ldotp
\end{equation}
We set $\phi := \sum_{T \in \tau} \phi_T$. Since the sum is finite this function is well defined and belongs to the class $\Ccinfty(\Omega)$. We then have
\begin{equation*} 
\begin{aligned}
\thetaqc(t) + 2 \epsilon &> \dashint_\Omega \QQ F(\ggrad \psi) \dd x = |\Omega|^{-1} \left( \int_{\bigcup_{\tau}} \QQ F (\ggrad \psi) \dd x + \int_{\Omega \setminus \bigcup_{\tau}} \QQ F (\ggrad \psi) \dd x \right) \\
&> |\Omega|^{-1} \int_{\bigcup_{\tau}} F (\ggrad \psi + \ggrad \phi) \dd x - 2\epsilon > \dashint_{\Omega} F (\ggrad \psi + \ggrad \phi) \dd x - 3\epsilon.
\end{aligned}
\end{equation*}
Here the first inequality is from the construction of $\psi$, the second one results from  estimating the remainder integral $\left| \int_{\Omega \setminus \bigcup_{\tau}} \QQ F (\ggrad \psi) \dd x \right| $ and using \eqref{eqDefPhiT} on $T \in \tau$, and the last one is just an estimate on $\left| \int_{\Omega \setminus \bigcup_{\tau} T} F(\ggrad \psi) \dd x \right|$ paired with the fact that $\phi \equiv 0$ on $\Omega \setminus \bigcup_{\tau} T$.
Finally, since $\phi \in \Ccinfty(\Omega)$, Jensen's inequality gives 
$$\dashint_\Omega |\ggrad \psi + \ggrad \phi|^q \dd x \geq \dashint_\Omega |\ggrad \psi|^q \dd x \geq t.$$
Since $\epsilon >0$ was arbitrary this already shows, that if $\omega = \Omega$ then $\thetaqc = \theta^{\Omega}_q$. To pass to an arbitrary $\omega$ we need to realise the distribution of $\ggrad(\psi + \phi)$ on $\omega$, which follows easily from Lemma \ref{lemmaChangeDomains} --- observe that this also shows that we preserve the moment restrictions.
\end{proof}

\begin{lemma}[see Proposition 3.2 in \cite{ChenKristensen17}]\label{lemmaThetaConvex}
Let $F \colon \RRnm \to \RR$ be a continuous integrand satisfying the growth condition \eqref{eqGrowth}. Then its associated auxiliary function $\theta := \theta^{\Omega}_q \colon [0,\infty) \to \RR\cup \{-\infty\}$ is convex.
\end{lemma}
\begin{proof}
We omit the proof as our case only differs from that of  \cite{ChenKristensen17} by a scaling exponent.
\end{proof}

We are now ready to finish the proof of the main result of this section.
\begin{proof}[Proof of Theorem \ref{thmCoercivity}, part 2]
We have already seen that a) and b) are equivalent, and so are c) and d). That c) implies a) is obvious. We will now show that a) implies d). If a) holds, i.e., if $F$ is $\LLq$ coercive, then its auxiliary function $\theta$ satisfies $\lim_{t \to \infty} \phi(t) = \infty$. This, paired with the fact that $\theta$ is convex (by Lemma \ref{lemmaThetaConvex}), implies that there exist $c_1 > 0$ and $c_2 \in \RR$ such that $\theta(t) \geq c_1 t + c_2$ for all $t \geq 0$. Thus, taking $\Omega$ to be any admissible domain and imposing zero boundary conditions, we may take $t = \| \ggrad \phi \|_{\LLq(\Omega)}^q$ for any $\phi \in \WWapn(\Omega)$ and conclude that $F$ is $\LLq$ mean coercive. Thus, we conclude that the conditions a) -- d) are all equivalent.

The fact that c) implies e) and f) is immediate. For the other direction let us start with a simple case of $\Omega = Q$, and boundary datum $g$ such that $\ggrad g$ is some constant $X_0 \in \RRnm$. Let us assume that in this case all minimising sequences are bounded, we now wish to prove that $F$ is $\LLq$ mean coercive at $X_0$ on $Q$. Consider the auxiliary function $\widetilde{\theta}$ of the translated integrand $\widetilde{F}(\cdot) := F(X_0 + \cdot)$, which clearly still satisfies the $\LLp$ growth bound. The assumption that all minimising sequences are bounded easily implies that the infimum of our variational problem, equal to $\widetilde{\theta}(0)$, cannot be $- \infty$. If it was, one could take an arbitrary (bounded) minimising sequence $\phi_j$, rescale it to be supported on a smaller cube, and then use the remaining room to make the $\WWaq$ norm of the resulting sequence arbitrarily large without spoiling the minimising property, thanks to the $p$ growth assumption on $F$.

Clearly $\widetilde{\theta}$ is non-decreasing, and convex as shown in Lemma \ref{lemmaThetaConvex}. If $\widetilde{\theta}$ was bounded from above we would deduce existence of minimising sequences of arbitrarily large $\WWaq$ norms, which would contradict our assumption. Thus, $\widetilde{\theta}$ is real-valued, unbounded, and convex, and so we deduce, as previously, that $F$ is $\LLq$ mean coercive at $z_0$ on $Q$, which implies point d) of our theorem, and thus all the points a) - d). 

Now let us take an arbitrary $\Omega$ and $g$ and assume, for contradiction, that $F$ is not mean coercive, but all minimising sequences are bounded. As previously, the infimum of the variational problem has to be a real number. 
Take any minimising sequence $u_j$ and note that, due to Proposition \ref{propPiecewiseApprox}, we may approximate any $u_j$ with a sequence $u^\epsilon_j \in \WWap_{u_j} = \WWap_u$, elements of which are piecewise polynomial on a finite family of disjoint open boxes, covering $\Omega$ up to a set of measure $\epsilon$, and converge to $u_j$ in the $\WWap$ norm. Since $F$ is a continuous integrand of $p$ growth, the map $u \mapsto \int_\Omega F(\ggrad u)$ is continuous in $\WWap$, thus extracting a diagonal subsequence we may assume that our (not renamed) minimising sequence $u_j$ is such that for each $j$ there exists a finite family $\{Q^j_i\}_i$ of pairwise disjoint boxes with $\left| \Omega \setminus \bigcup_i Q^j_i \right| < 1/j$ and with $\ggrad u_j$ constant on each $Q^j_i$. 

By the previous step, since we assume that $F$ is not $\LLq$ mean coercive we may, for any open box and any $\ba$-polynomial boundary datum, construct an unbounded minimising sequence. Doing so on each $Q^j_i$ with boundary datum $u_j$ (which is indeed an $\ba$-polynomial on $Q^j_i$) we may construct functions $\phi^j_i \in \WWapn(Q^j_i)$ such that
$$ j \left| Q^j_i \right| < \int_{Q^j_i} \left| \ggrad (u_j + \phi^j_i) \right|^q \dd x,$$
and
$$  \int_{Q^j_i} F( \ggrad (u_j + \phi^j_i) ) \dd x \leq  \int_{Q^j_i} F(\ggrad u_j) + 1/j \dd x.$$
Extending each $\phi^j_i$ by zero outside $Q^j_i$ and letting $v_j := u_j + \sum_i \phi^j_i$ yields a sequence $v_j \in \WWap_u(\Omega)$ that clearly still minimizes the functional, but is unbounded in $\WWaq$, and thus we have a contradiction, which proves that f) implies c), thus the points a) through f) are equivalent.

That g) implies d) is easy. Assuming that $F - \delta |\cdot|^q$ is quasiconvex at some $X_0 \in \RRnm$ we may write
$$ F(X_0) - \delta |X_0|^q \leq \dashint_\Omega F(X_0 + \ggrad \phi (x)) - \delta |X_0 + \ggrad \phi(x)| \dd x$$
for all $\phi \in \Ccinfty(\Omega)$. This is exactly $\LLq$ mean coercivity on $\WWap_g(\Omega)$ with $g$ a polynomial such that $\ggrad g = X_0$, $c_1 = \delta > 0$ and $c_2 = F(X_0) - \delta|X_0|^q$. 

As the last step we prove that c) implies g). Assume that $F$ is $\LLq$ mean coercive, take $\Omega$ to be any admissible domain, and let the boundary condition be $g \equiv 0$. If the corresponding coercivity constants are $c_1$ and $c_2$ then take any $\delta \in (0,c_1)$ and put $G(X) := F(X) - \delta |X|^q$ for $X \in \RRnm$. Clearly, the auxiliary function of $G$ is bounded from below by $(c_1 - \delta) t + c_2$ and, as always, $G \geq \QQ G$. From Lemma \ref{lemmaAuxiliaryFunction} we know that the auxiliary functions of $G$ and $\QQ G$ are the same, thus we may write
$$ \dashint_\Omega G( \ggrad \phi) \dd x \geq \dashint_\Omega \QQ G(\ggrad \phi) \dd x \geq (c_1 - \delta) \dashint_\Omega |\ggrad \phi|^q \dd x + c_2$$
for all $\phi \in \WWapn(\Omega)$. Thanks to Corollary \ref{corQuasiconvexityIndependentOfDomain} we may find a sequence $\phi_j \subset \Ccinfty(\Omega)$ for which
$$ \QQ G(0) \leq \dashint_\Omega \QQ G(\ggrad \phi_j) \dd x \leq \dashint_\Omega G(\ggrad \phi_j) \dd x \searrow \QQ G(0).$$
Since we have picked $\delta < c_1$ it is easy to conclude that $G$ is still $\LLq$ mean coercive. Since $\dashint_\Omega G(\ggrad \phi_j) \dd x $ is bounded we infer that the sequence $\ggrad \phi_j$ is bounded in $\LLq$. As in \cite{ChenKristensen17} we consider the following probability measures on $\RRnm$
$$ \nu_j := (\ggrad \phi_j)_{\#} \left(\frac{\mathcal{L^N} \restriction \Omega}{|\Omega|} \right)$$
and observe that they have uniformly bounded $q$-th moments. Thus, passing to a subsequence if necessary, we may assume that $\nu_j \overset{*}{\rightharpoonup} \nu$ in $C_0(\RRnm)^*$, where $\nu$ is some probability measure on $\RRnm$ with finite $q$th moment. Setting $H(X) := G(X) - \QQ G(X)$, which is a non-negative and lower semicontinuous function, we may use the portmanteau theorem to write
$$ 0 \leq \int_{\RRnm} H \dd \nu \leq \liminf_{j \to \infty} \int_{\RRnm} H \dd \nu_j = 0,$$
thus $H = 0$ on the support of $\nu$. However, then $G = \QQ G$ on the support of $\nu$, i.e., $F - \delta |\cdot|^q$ is $\ba$-quasiconvex on the support of $\nu$, and since $\nu$ is a probability measure its support is non-empty, which ends the proof.
\end{proof}

\section{Lower semicontinuity}\label{sectionLSC}
The central part of our existence theory is sequential weak lower semicontinuity of integral functionals acting on Sobolev spaces of mixed smoothness. We will show that this property is equivalent to $\ba$-quasiconvexity of the integrand, with the precise variant of $\ba$-quasiconvexity depending on the regularity of the integrand and on the space on which the functional is defined. We begin with the definitions of (closed) $\WWap$-quasiconvexity and a short discussion of the relationship between the various notions. 

Similarly to Ball and Murat in \cite{BallMurat84} we define the following:
\begin{defi}
We say that a function $F \colon \RRnm \to \RR$ is $\WWap$-quasiconvex if for every $X \in \RRnm$ one has
$$ F(X) \leq \inf_{u \in \WWapn(Q;\RRn)} \dashint_Q F(X + \ggrad u(x)) \dd x \ldotp$$
\end{defi}

Observe that this definition differs from that of $\ba$-quasiconvexity only by replacing the test space $\Ccinfty$ by $\WWapn$. In fact, we have the following:

\begin{lemma}[see \cite{BallMurat84}]\label{lemmaBallMurat84}
Suppose that $F \colon \RRnm \to \RR$ is continuous and satisfies $|F(X)| \leq C(1 + |X|^p)$ for every $X \in \RRnm$. Then $F$ is $\WWap$-quasiconvex if and only if it is $\ba$-quasiconvex. 
\end{lemma}
In the classical case of first order gradients this has been shown by Ball and Murat in \cite{BallMurat84} through a simple application of Fatou's lemma. The same proof works in our case, thus we skip it and note that the above and Lemma \ref{lemmaQgIsQC} immediately imply the following.
\begin{corollary}
Suppose that $F \colon \RRnm \to \RR$ is continuous and satisfies $|F(X)| \leq C(1 + |X|^p)$ for every $X \in \RRnm$. Then $\QQ F$ is $\WWap$-quasiconvex.
\end{corollary}

Let us note that the growth assumption on $F$ cannot, in general, be removed from the statement of Lemma \ref{lemmaBallMurat84}. In other words, $\WWap$ and $\WWaq$-quasiconvexity are different notions in general. The previously mentioned paper \cite{BallMurat84} contains an example of a function that is $\WW^{1,p}$ quasiconvex only for sufficiently large exponents $p$, thus showing the importance of the relevant test space.

Finally, for future use, we introduce a third notion (analogous to the one introduced by Pedregal in \cite{Pedregal94} and later studied by Kristensen in \cite{Kristensen15}) that will be particularly useful for dealing with extended-real valued integrands.
\begin{defi}
We say that a function $F \colon \RRnm \rightarrow (-\infty, \infty]$ is closed $\WWap$-quasiconvex if $F$ is lower semicontinuous and Jensen's inequality holds for $F$ and every homogeneous oscillation $\WWap$-gradient Young measure, i.e.,
$$F(X) \leqslant \inf_{\nu \in \hp_{X}} \intRRnm F(W) \dd \nu(W),$$
where $\hp_{X}$ is the set of all homogenous $\WWap$-gradient Young measures with mean $X$. 
\end{defi}

As in Theorem \ref{thmCharacterizationOfYM}, we note that replacing $\nu \in \hp_X$ with $\delta_{X} \ast \mu$ for $\mu \in \hpn$ allows us to rewrite the inequality in the definition of closed $\WWap$-quasiconvexity as 
$$F(X) \leqslant \inf_{\mu \in \hpn} \intRRnm F(X + W) \dd \mu(W),$$
for all $X \in \RRnm$.

Before we proceed, we note the following lemma which, in the case of classical gradients, was first proven by Ball and Zhang in \cite{BallZhang90}:
\begin{lemma}\label{lemmaAllQCEquivalent}
Suppose that $F \colon \RRnm \rightarrow \RR$ is a continuous integrand satisfying $|F(X)| \leq C(|X|^p + 1)$ for some constant $C$ and for all $X \in \RRnm$. Then $F$ is $\ba$-quasiconvex if, and only if, it is $\WWap$-quasiconvex and if, and only if, it is closed $\WWap$-quasiconvex.
\end{lemma}
\begin{proof}
Under these assumptions the equivalence between $\ba$-quasiconvexity and $\WWap$-quasiconvexity is the content of Lemma \ref{lemmaBallMurat84}. To prove the equivalence of closed $\WWap$-quasiconvexity with these notions let us first observe that clearly closed $\WWap$-quasiconvexity implies $\ba$-quasiconvexity, as for every $\phi \in \Ccinfty(Q)$ the measure given by $(\ggrad \phi)_{\#} \frac{\LebesgueN \measurerestr Q}{|Q|} $ is a homogeneous $\WWap$-gradient Young measure (see Corollary \ref{corElementaryHomogenous}). Note that this part does not use any assumptions on $F$ --- indeed, closed $\WWap$-quasiconvexity is the strongest of the three notions.

For the other direction fix a point $X \in \RRnm$ and an arbitrary $\mu \in \hpn$ and let $\phi_j \in \WWapn(Q)$ be such that the sequence $\{ \ggrad \phi_j \}$ is $p$-equintegrable and generates $\mu$, which is possible thanks to Corollary \ref{corOscillationEquiintegrable}. Due to the growth conditions on $F$ we know that the family $\{ F(X + \ggrad \phi_j) \}$ is equiintegrable as well, so that Theorem \ref{thmFToYM} shows that 
$$ \intRRnm F(X + W) \dd \mu(W) = \lim_{j \to \infty} \dashint_Q F(X + \ggrad \phi_j(x)) \dd x \geq F(X),$$
where the last inequality is due to $\WWap$-quasiconvexity of $F$. Since $X$ and $\nu$ were arbitrary, this ends the proof. Observe that we do not need to assume that $F$ is bounded from below to apply Theorem \ref{thmFToYM}, as we may simply consider its positive and negative parts separately.
\end{proof}

In the following we establish an initial result that identifies closed $\WWap$-quasiconvexity as a sufficient condition for lower semicontinuity of integral functionals in the mixed smoothness setting.

\begin{lemma}\label{lemmaClosedQCImpliesLSC}
Let $\Omega$ be a bounded open Lipschitz domain. Suppose that $F \colon \RRnm \to [0, \infty]$ is a closed $\WWap$-quasiconvex integrand. Then the functional
$$ \II (u) := \int_\Omega F(\ggrad u) \dd x$$
is sequentially weakly lower semicontinuous on $\WWap(\Omega)$.
\end{lemma}
\begin{proof}
Fix an arbitrary $u \in \WWap(\Omega)$ and a sequence $u_j \weakConv u$ in $\WWap(\Omega)$. We need to show that $\liminf \II(u_j) \geq \II(u)$. Since $u_j$ is weakly convergent in $\WWap(\Omega)$ it is also bounded in that space. Passing to a subsequence if necessary, we may assume that the $\liminf$ is a true limit and, passing to a further subsequence, that $\ggrad u_j$ generates some $\WWap$-gradient Young measure $\nu = \{\nu_x\}_{x \in \Omega}$. Observe that, since $p \in (1, \infty)$, the barycentre of $\nu_x$ is $\ggrad u(x)$ for almost every $x$ in $\Omega$. By Theorem \ref{thmFToYM} we know that 
$$ \lim_{j \to \infty} \II(u_j) = \lim_{j \to \infty} \int_\Omega F(\ggrad u_j) \dd x \geq \int_\Omega \int_{\RRnm} F \dd \nu_x \dd x.$$
Proposition \ref{propLocalisation} tells us that, for Lebesgue almost every $x$, the measure $\nu_x$ is a homogeneous $\WWap$-gradient Young measure. Hence $F$, as a closed $\WWap$-quasiconvex function, satisfies Jensen's inequality when tested with $\nu_x$ for almost every $x$, which ends the proof, as we now have
$$ \lim_{j \to \infty} \II(u_j) \geq \int_\Omega \int_{\RRnm} F \dd \nu_x \dd x \geq \int_\Omega F(\overline{\nu_x}) \dd x = \int_\Omega F(\ggrad u(x)) \dd x  = \II(u).$$
\end{proof}
\begin{corollary}
Since the proof above is based on localisation, we immediately get that, under the assumptions of the above lemma, the functional
$$ \WWap(\Omega) \ni u \mapsto \int_A F(\ggrad u) \dd x$$
is sequentially weakly lower semicontinuous on $\WWap(\Omega)$ for any measurable subset $A \subset \Omega$.
\end{corollary}

\begin{lemma}\label{lemmaLSCImpliesQC}
Let $\Omega$ be a bounded open Lipschitz domain satisfying the weak $\ba$-horn condition. Suppose that $F \colon \RRnm \to (-\infty, \infty]$ is a measurable integrand and that its associated functional $\II(u) := \int_\Omega F(\ggrad u) \dd x$ is sequentially weakly lower semicontinuous on $\WWap(\Omega)$. Then $F$ is $\WWap$-quasiconvex.
\end{lemma}
\begin{proof}
Fix an arbitrary $\phi \in \WWapn(Q)$ and a $X \in \RRnm$. Let $u$ be an $\ba$-polynomial with $\ggrad u \equiv X$. For an arbitrary $k \in \naturals$ we may cover $\Omega$, up to a set of measure zero, with a countable family of anisotropic boxes of the form $\{ x^j_i + r^j_i \scale Q \}_i$, with $r^j_i < 1/j$ for all $j,i$. Let $\phi^j_i (x) := (r^j_i)\phi((r^j_i)^{-1} \scale (x - x^j))$ and set $\phi_j := \sum_{i = 1}^{\infty} \phi^j_i $.
Then $\phi_j$ converges weakly to $0$ in $\WWap(\Omega)$ (see Lemma \ref{lemmaElementaryScaleAndTile}), thus $u + \phi_j \weakConv u$ in $\WWap(\Omega)$, so that
$$ \liminf_{j \to \infty} \int_\Omega F( \ggrad (u + \phi_j) \dd x = \liminf_{j \to \infty} \int_\Omega F(X + \ggrad \phi_j) \dd x \geq \int_\Omega F(\ggrad u) \dd x = |\Omega| F(X).$$
However, for every $j$ we have, by a change variables,
$$ \int_\Omega F(X + \ggrad \phi_j) \dd x = |\Omega| \dashint_Q F(X + \ggrad \phi) \dd x.$$
which ends the proof.
\end{proof}

Putting together the content of Lemmas \ref{lemmaAllQCEquivalent}, \ref{lemmaClosedQCImpliesLSC}, and \ref{lemmaLSCImpliesQC} we obtain the following:

\begin{theorem}\label{thmLSCWithpGrowth}
Let $\Omega$ be a bounded open Lipschitz domain satisfying the weak $\ba$-horn condition. Suppose that $F \colon \RRnm \to [0, \infty)$ is a continuous integrand satisfying the $p$-growth condition $|F(X)| \leq C(|X|^p + 1)$ for some constant $C$ and all $X \in \RRnm$. Then the functional
$$ \II (u) := \int_\Omega F(\ggrad u) \dd x$$
is sequentially weakly lower semicontinuous on $\WWap$ if and only if $F$ is $\WWap$-quasiconvex.
\end{theorem}

\section{Relaxation}\label{sectionRelaxation}
We have already identified $\WWap$-quasiconvexity as an equivalent condition for lower semicontinuity of integral functionals in the mixed smoothness framework. When a functional lacks lower semicontinuity one typically studies its relaxation, defined as the sequentially weakly lower semicontinuous envelope of the original functional, i.e., through
$$ \overline{\II} (u) := \inf_{u_j \weakConv u} \left\{ \liminf_{j \to \infty} \II(u_j) \right\},$$
where the infimum is taken over all sequences $u_j$ converging to $u$ in the appropriate sense. In our case this means $u_j \weakConv u$ in $\WWap(\Omega)$. With this definition the relaxed functional $\overline{\II}$ is sequentially lower semicontinuous with respect to weak convergence in $\WWap(\Omega)$. The aim of this section is to show that, under appropriate assumptions, the relaxation is again an integral functional with integrand given by $\ba$-quasiconvexification of the original one. In the classical setting of first order gradients and under $(p,p)$ growth conditions on the integrand, this was first done by Dacorogna in \cite{Dacorogna82}. The main results of this section are contained in Theorems \ref{thmRelaxationPGrowth} and \ref{thmContinuousImpliesRelaxation}, which deal with the $p$-growth case and the extended-real valued case, respectively.

\subsection{The $p$-growth case}
\begin{lemma}\label{lemmaContinuousEnvelope}
Suppose that $F \colon \RRnm \to [0, \infty)$ is a continuous and $\LLp$ coercive integrand with $F(X) \leq C(|X|^p + 1)$. Assume further that $F$ is locally Lipschitz with
\begin{equation}\label{eqContinuousEnvelope}
|F(X) - F(W)| \leq D(1 + |X|^{p-1} + |W|^{p-1})|X-W|,
\end{equation}
for all $X,W \in \RRnm$ and some $D \in \RR$. Then $\QQ F$ is continuous.
\end{lemma}
\begin{proof}
For a continuous integrand of $p$-growth the mapping $X \mapsto \dashint_Q F(X + \ggrad \phi) \dd x$ is continuous for any fixed $\phi \in \Ccinfty(Q)$. Thus, by the Dacorogna formula \eqref{eqDefQCEnvelope}, we immediately see that the quasiconvex envelope is upper semicontinuous, as it is a pointwise infimum of a family of continuous functions. 

By Theorem \ref{thmCoercivity} we know that the functional induced by $F$ is $\LLp$ mean coercive with any boundary datum. However, the constants describing the coercivity may depend on the datum --- if this was not the case then $F$ would have to satisfy pointwise coercivity bounds of the form $F(X) \geq C|X|^p - C^{-1}$, which we do not assume. Nevertheless, under the assumption of local Lipschitz continuity of $F$, it may be shown that the constants may be chosen uniformly on compact sets. To prove that, let us take $X_0$ to be a point at which $F - c|\cdot|^p$, with some positive constant $c$, is $\ba$-quasiconvex --- such a point exists by Theorem \ref{thmCoercivity}. For an arbitrary $X \in \RRnm$ and any $\phi \in \Ccinfty(Q)$ we may, using \eqref{eqContinuousEnvelope}, write
$$ \dashint_Q \left| F(X + \ggrad \phi) - F(X_0 + \ggrad \phi) \right| \dd x \leq $$
$$ D \dashint_Q \left( 1 + \left| X + \ggrad \phi \right|^{p-1} + \left| X_0 + \ggrad \phi \right|^{p-1} \right) |X - X_0| \dd x.$$
Rearranging we get
$$ \dashint_Q \left| F(X + \ggrad \phi) - F(X_0 + \ggrad \phi) \right| \dd x \leq c_1 + c_2 \dashint_Q |\ggrad \phi|^{p-1} \dd x,$$
where the constants $c_i$ depend on $X$ but may be chosen uniformly on compact sets. These constants also depend on $D$ and $X_0$, but, since the integrand is fixed, this dependence is not important here. Rearranging and using the strict $\ba$-quasiconvexity at $X_0$ given by Theorem \ref{thmCoercivity} we may now write
$$  \dashint_Q F(X + \ggrad \phi) \dd x \geq c_0 \dashint_Q |\ggrad \phi |^p \dd x - c_1 - c_2 \dashint_Q |\ggrad \phi|^{p-1} \dd x,$$
with a different constant $c_1$. Using weighted Young's inequality we get
$$  \dashint_Q F(X + \ggrad \phi) \dd x \geq \frac{c_0}{2} \dashint_Q |\ggrad \phi |^p \dd x - c_1,$$
where, again, the constant $c_1$ has changed, but it is still independent of $\phi$ and may be chosen locally uniformly in $X$. Thus, we have shown that for any compact set $K$ there exists a constant $c > 0$ such that 
\begin{equation}\label{eqUniformCoercivity}  
\dashint_Q F(X + \ggrad \phi) \dd x \geq c \dashint_Q |\ggrad \phi |^p \dd x - c^{-1},
\end{equation}
for all $X \in K$ and all $\phi \in \Ccinfty(Q)$.

To show that $\QQ F$ is lower semicontinuous take an arbitrary point $X \in \RRnm$ and a sequence $X_j \to X$. Assume, without loss of generality, that $\liminf_{j \to \infty} \QQ F(X_j) = \lim_{j \to \infty} \QQ F(X_j) < \infty$. For every $j$ we may, by the Dacorogna formula, find a function $\phi_j \in \Ccinfty(Q)$ such that 
$$ \QQ F(X_j) > \dashint_Q F(X_j + \ggrad \phi_j) \dd x.$$
Since $X_j$ is a convergent sequence, it is contained in a compact set. Similarly, the sequence $\QQ F(X_j)$ is bounded, thus, based on what we have just proven in \eqref{eqUniformCoercivity}, we deduce that the family $\{ \ggrad \phi_j \}$ is bounded in $\LLp$. Using the Lipschitz bound \eqref{eqContinuousEnvelope}, we may write
\begin{equation}\label{eqContEnvelope2}
\begin{aligned}
 \dashint_Q F(X + \ggrad \phi_j) \dd x \leq & \dashint_Q F(X_j + \ggrad \phi_j) \dd x \\
+ & D \dashint_Q \left( 1 + |X|^{p-1} + |X_j|^{p-1} + |\ggrad \phi_j|^{p-1} \right) |X - X_j| \dd x.
 \end{aligned}
 \end{equation}
Now, $|X - X_j|$ converges to $0$ as $j \to \infty$, and the integral 
$$\dashint_Q \left( 1 + |X|^{p-1} + |X_j|^{p-1} + |\ggrad \phi_j|^{p-1} \right) \dd x$$ 
is bounded, hence the last term in \eqref{eqContEnvelope2} goes to $0$ with $j$. Thus
\begin{equation*}
\begin{aligned}
\QQ F(X) &\leq \liminf_{j \to \infty} \dashint_Q F(X + \ggrad \phi_j) \dd x \\
&\leq \liminf_{j \to \infty} \dashint_Q F(X_j + \ggrad \phi_j) \dd x = \liminf_{j \to \infty} \QQ F(X_j),
\end{aligned}
\end{equation*}
which ends the proof.
\end{proof}

Let us remark that the above could be easily generalised to the $\mathcal{A}$-free setting of compensated compactness due to Murat and Tartar (see \cite{Murat78}, \cite{Murat81} by Murat and \cite{Tartar79}, \cite{Tartar83}, \cite{Tartar92} by Tartar). It is known that in general, when the characteristic cone of $\mathcal{A}$ does not span the entire space, $\mathcal{A}$-quasiconvex envelopes of smooth functions need not be continuous (see, for example, Remark 3.5 in \cite{FonsecaMuller99}). Our proof shows that this is not a problem if one restricts to coercive integrands satisfying the bound \eqref{eqContinuousEnvelope}. That being said, in the context of $\mathcal{A}$-quasiconvexity, this may also be resolved in a different way (see the author's recent collaboration with Rai\c{t}\u{a} \cite{ProsinskiRaita19})), however that approach cannot be applied in the mixed smoothness case, and thus we do not elaborate on it further.

Our main relaxation result is the following:
\begin{theorem}\label{thmRelaxationPGrowth}
Let $\Omega$ be a bounded open Lipschitz domain satisfying the weak $\ba$-horn condition. Supppose that $F \colon \RRnm \to [0, \infty)$ is a continuous and $\LLp$ coercive integrand with $F(X) \leq C(|X|^p + 1)$. Assume furthermore that $F$ satisfies 
\begin{equation}\label{eqRelaxationLowerGrowthAssumption}
F(X) \geqslant D|X|^p - D^{-1}
\end{equation} 
or 
\begin{equation}\label{eqRelaxationLocallyLipschitzAssumption}
|F(X) - F(W)| \leq D(1 + |X|^{p-1} + |W|^{p-1})|X-W|,
\end{equation} 
for all $X, W \in \RRnm$.
Then the sequentially weakly lower semicontinuous envelope of the functional $\II_F$ is given by
$$\overline{\II}_F (u) := \inf_{u_j \weakConv u} \left\{ \liminf_{j \to \infty} \II_F (u_j) \right\}= \int_{\Omega} \QQ F(\ggrad u(x)) \dd x = \II_{\QQ F}(u),$$
where the infimum is taken over all sequences $u_j$ converging to $u$ weakly in $\WWap(\Omega)$.
\end{theorem}
\begin{proof}
When $F$ satisfies the first of the two alternative conditions we have provided, i.e., when $F$ is of $p$-growth from below as well as from above, this result is a corollary of a more general relaxation result (for extended real-valued integrands) that we will demonstrate next, thus we postpone this part of the proof. 

Under the assumption that $F$ is locally Lipschitz, we have shown in Lemma \ref{lemmaContinuousEnvelope}, that $\QQ F$ is continuous. We also know, from Lemma \ref{lemmaQFGrowth}, that $\QQ F$ satisfies the same $p$-growth assumption as $F$. Lemma \ref{lemmaQgIsQC} tells us that $\QQ F$ is $\ba$-quasiconvex, which in view of the continuity and growth bounds implies, by Lemma \ref{lemmaAllQCEquivalent}, that it is $\WWap$-quasiconvex. Finally, Lemma \ref{lemmaClosedQCImpliesLSC} shows that the functional induced by $\QQ F$ is indeed sequentially weakly lower semicontinuous on $\WWap(\Omega)$. This translates to
$$ \overline{\II}_F (u) \geq  \II_{\QQ F}(u),$$
and so it remains to prove the reverse inequality.

Proceeding similarly to a proof in Dacorogna's book \cite{Dacorogna07} (see Theorem 9.1 therein) let us start with the simple case of $\Omega = Q$ and $u$ with $\ggrad u = X$ for some constant $X \in \RRnm$. By the Dacorogna formula \eqref{eqDefQCEnvelope} there exists a sequence $\phi_j \in \Ccinfty(Q)$ with 
$$ \int_Q F( \ggrad u + \ggrad \phi_j) \dd x \to \int_Q \QQ F(\ggrad u) \dd x.$$
It only remains to show that the sequence $\phi_j$ may be chosen in such a way as to satisfy $\phi_j \weakConv 0$ in $\WWap(Q)$. The argument here is essentially a simpler version of the one in the proof of Lemma \ref{lemmaElementaryScaleAndTile}. For a fixed $j$ extend $\phi_j$ periodically and consider the sequence
$$ \phi_j^k (x) := r_k^{-1} \phi_j(r_k \scale x),$$
with $r_k := k^{a_1 \cdot \ldots \cdot a_N}$. Then $\ggrad \phi_j^k$ preserves the integral, i.e., 
$$ \int_Q F( \ggrad u + \ggrad \phi_j^k) \dd x = \int_Q F( \ggrad u + \ggrad \phi_j) \dd x,$$
for every $k$, and $\phi_j^k \weakConv 0$ in $\WWap(Q)$ with $k \to \infty$. Thus, a diagonal extraction argument ends the proof in this basic case. 

For the general case let us fix an arbitrary function $u \in \WWap(\Omega)$. Using Proposition \ref{propPiecewiseApprox} we may find a sequence $v_j \in \WWap_u(\Omega)$ with $v_j \to u$ in $\WWap(\Omega)$ and such that for each $j$ there exists a finite family of anisotropic boxes $\{Q^j_i\}_i$ such that $\ggrad v_j$ is constant on each $Q^j_i$ and $\left| \Omega \setminus \bigcup_i Q^j_i \right| \to 0$ as $j \to \infty$. 
We know, from Lemma \ref{lemmaContinuousEnvelope}, that $\QQ F$ is continuous, and from Lemma \ref{lemmaQFGrowth} that it satisfies the $p$-growth bound. Thus
$$ \int_\Omega \QQ F(\ggrad v_j) \dd x \to \int_\Omega \QQ F(\ggrad u) \dd x.$$
Since $v_j$ converges to $u$ strongly in $\WWap$, the sequence $\ggrad v_j$ is $p$-equiintegrable. By assumption and Lemma \ref{lemmaQFGrowth} we know that both $F$ and $\QQ F$ satisfy the $p$-growth bound from above, so that 
$$ \int_{\Omega \setminus \bigcup_i Q^j_i} F(\ggrad v_j) \dd x \to 0 \text{ and } \int_{\Omega \setminus \bigcup_i Q^j_i} \QQ F(\ggrad v_j) \dd x \to 0.$$
In particular,
$$ \int_{\bigcup_i Q^j_i} \QQ F(\ggrad v_j) \dd x  + \int_{\Omega \setminus \bigcup_i Q^j_i} F(\ggrad v_j) \dd x \to \int_\Omega \QQ F(\ggrad u) \dd x.$$
Using the previous step we may, for any fixed $Q^j_i$, find a sequence $\phi^j_{i,k} \in \Ccinfty(Q^j_i)$ with $\phi^j_{i,k} \weakConv 0$ in $\WWap(\Omega)$ as $k \to \infty$ and such that
$$ \int_{Q^j_i} F(\ggrad v_j + \ggrad \phi^j_{i,k}) \dd x \to \int_{Q^j_i} \QQ F(\ggrad v_j) \dd x \text{ as } k \to \infty.$$
Letting $\phi^j_k := \sum_i \phi^j_{i,k} \in \Ccinfty(\Omega)$, we have $\phi^j_k \weakConv 0$ in $\WWap(\Omega)$ as $k \to \infty$ and
$$ \int_{\bigcup_i Q^j_i} F(\ggrad v_j + \ggrad \phi^j_{k}) \dd x \to\int_{\bigcup_i Q^j_i} \QQ F(\ggrad v_j) \dd x.$$
A standard diagonal extraction argument allows us to construct a, non-relabelled, diagonal sequence $\phi_j \in \Ccinfty(\Omega)$ with $\phi_j \weakConv 0$ in $\WWap(\Omega)$ and such that
$$ \int_\Omega F(\ggrad v_j + \ggrad \phi_j) \dd x \to \int_\Omega \QQ F(\ggrad u) \dd x.$$
Therefore, setting $u_j := v_j + \phi_j$, shows that
$$ \liminf_{j \to \infty} \II_F(u_j) \leq \II_{\QQ F}(u),$$
and thus ends the proof.
\end{proof}

Observe that we only use the locally Lipschitz assumption on $F$ in order to deduce continuity of $\QQ F$ from Lemma \ref{lemmaContinuousEnvelope}. If continuity can be ensured in a different way then we can dispense with the assumption \eqref{eqRelaxationLocallyLipschitzAssumption}. In the classical setting of first order gradients it is known that quasiconvex functions are convex along rank-one directions, and these span the entire space, so that one may deduce continuity from directional convexity. We have already mentioned that, in general, there is no good analogue of rank-one convexity in the mixed smoothness setting. However, if $\ba$ is such that all multi-indices $\alpha$ with $\langle \alpha, \bamo \rangle = 1$ are of the same parity, this has been resolved by Kazaniecki, Stolyarov, and Wojciechowski in \cite{KazanieckiStolyarovWojciechowski17}. Under this assumption they have shown that $\ba$-quasiconvex functions are convex along directions of the form
\begin{equation}\label{eqRankOneVectors}
\sum_{\langle \alpha, \bamo \rangle = 1} i^{|\alpha| + |\alpha_0|} x^\alpha b^i e_{\alpha, i}.
\end{equation}
Here $x \in \RRN$ and $b \in \RRn$ are arbitrary vectors, $i^2 = -1$, $e_{\alpha,i}$ is the canonical basis of $\RRnm$, and $\alpha_0$ is an arbitrary multi-index on the hyperplane of homogeneity, i.e., with $\langle \alpha_0, \bamo \rangle = 1$. When all the multi-indices have the same parity, the coefficients in \eqref{eqRankOneVectors} are real and vectors of this form span $\RRnm$, thus continuity follows from directional convexity. Thus, we obtain the following:

\begin{corollary}
If $\bamo$ is such that all the multi-indices $\alpha$ with $\langle \alpha, \bamo \rangle = 1$ have orders of the same parity then the conclusion of Theorem \ref{thmContinuousImpliesRelaxation} holds without the additional assumptions \eqref{eqRelaxationLowerGrowthAssumption} or \eqref{eqRelaxationLocallyLipschitzAssumption}.
\end{corollary}

If, on the other hand, the parities of $|\alpha|$ do not match, then the coefficients in \eqref{eqRankOneVectors} complexify, and we do not get any directional convexity. It is still possible to use the calculation leading to determining the form of the vectors in \eqref{eqRankOneVectors} to show that $\ba$-quasiconvex functions are (pluri)subharmonic in a certain sense (see the discussion in \cite{KazanieckiStolyarovWojciechowski17}), but we have not yet been able to use that to strengthen our relaxation results. We believe, however, that this should be studied further and intend to do so in future work.

\subsection{Closed $\WWap$-quasiconvex envelope}
The $p$-growth assumption on the integrand was crucial in the results of the previous subsections. Indeed, when $F$ is not of $p$-growth the notions of $\WWap$-quasiconvexity and closed $\WWap$-quasiconvexity need not coincide, and this is precisely the reason for introducing this stricter notion of closed $\WWap$-quasiconvexity. An explicit example, in the isotropic setting of first order gradients, of an integrand that is quasiconvex but not closed quasiconvex may be found in Example 1.3 in \cite{Kristensen15}.

The first issue we run into is that, for an extended real-valued integrand, the formula \eqref{eqDefQCEnvelope} need not yield a closed $\WWap$-quasiconvex function. The purpose of this subsection is to establish a formula that does. We start with the natural definition of the closed $\WWap$-quasiconvex envelope.

\begin{defi}\label{defiQCenvelope}
For a measurable function $F \colon \RRnm \to (-\infty, \infty]$ we define its closed $\WWap$-quasiconvex envelope by
$$ \overline{F}(X) := \sup\{G(X) \colon G \leqslant F, \, G \text{ is closed } \ba \text{-} p  \text{ quasiconvex} \} \ldotp$$
\end{defi}

Our goal is the following:
\begin{proposition}\label{lemmaQCenvelope}
For any $p \in (1,\infty)$, the closed $\WWap$-quasiconvex envelope of a lower semicontinuous function $F \colon \RRnm \to [0, \infty]$ satisfying the growth condition $F(X) \geqslant c|X|^p$ for some constant $c > 0$ is given by
$$ \overlineF (X) = \inf_{\nu \in \hpn} \langle F(\cdot + X), \nu \rangle = \inf_{\nu \in \mathbb{H}^p_{X}} \langle F, \nu \rangle \ldotp
$$
Moreover, the function $\overline{F}$ is indeed closed $\WWap$-quasiconvex. 
\end{proposition}
The reason we put the lower growth assumption in this result is that in a number of places in the proof we will use an argument that selects, at each $X \in \RRnm$ a measure $\nu_{X} \in \hp_{X}$ which (nearly) achieves $\langle F, \nu_X \rangle = \overline{F}(X)$. The main idea of the proof is that then, with any fixed $\nu \in \hpn$ the measure $\mu$ defined by $\dd \mu := \dd \nu_X \dd \nu$ is again a homogeneous $\WWap$-gradient Young measure. To prove that we need to ensure that it has a finite $p$-th moment, which is where coercivity comes into the picture. For now we do not know whether it is possible to relax this assumption, nevertheless it is in line with the relaxation result we prove next. As in the case of integrands of $p$-growth, to prove a relaxation formula we need some sort of a coercivity assumption on the integrand. If we were to relax the pointwise coercivity to simply $\LLp$ or $\LLp$-mean coercivity we would need a Lipschitz assumption on our integrand of the form appearing in Theorem \ref{thmContinuousImpliesRelaxation}, which we cannot have if we wish to allow $F$ to take the value $+ \infty$. Thus, for now at least, we content ourselves with including the lower-growth bound also in our formula for the closed quasiconvex envelope. We note that the lower growth assumption is also present in the relaxation result in Kristensen's paper \cite{Kristensen15}, on which we base our relaxation proof.

Before we proceed to the proof let us recall the following classical result due to Kuratowski and Ryll-Nardzewski:
\begin{theorem}[see \cite{KuratowskiRyllNardzewski65}]\label{thmKuratowski}
Let $X$ be a metric space and $Y$ be a separable and complete metric space. Fix a multi-valued function $\mathcal{G} \colon X \rightarrow 2^Y$. If for any closed set $K \subset Y$ the set $\{ x \in X \colon \mathcal{G}(x) \cap K \not= \emptyset \}$ is Borel measurable then $\mathcal{G}$ admits a measurable selector, i.e., there exists a Borel measurable function $g \colon X \rightarrow Y$ such that for all $x \in X$ we have $g(x) \in \mathcal{G}(x)$. 
\end{theorem}

\begin{proof}[Proof of Proposition \ref{lemmaQCenvelope}]
The argument that we present closely follows the one in a recent paper by the author (see \cite{Prosinski18}), where it was employed in the context of $\mathcal{A}$-quasiconvexity. Denote 
$$ R(X) := \inf_{\nu \in \hpn} \langle F(\cdot + X), \nu \rangle \ldotp$$
Clearly for any $\nu \in \hpn$ and $X \in \RRnm$ we have $ \overlineF(X) \leqslant \langle F(\cdot + X), \nu \rangle,$ therefore taking the infimum over $\nu \in \hpn$ yields
$$ \overlineF(X) \leqslant R(X),$$
hence showing that $R$ is closed $\WWap$-quasiconvex will give the reverse inequality and end the proof, as one immediately gets $R \leq F$ by testing with $\nu := \delta_0 \in \hpn$. 

To show that $R$ is lower semicontinuous fix $X_0 \in \RRnm$, a sequence $X_j \rightarrow X_0$, and an $\varepsilon > 0$. We will show that 
$$ \varepsilon + \liminf_{j \to \infty} R(X_j) \geqslant R(X_0) \ldotp$$
Without loss of generality assume that $ \lim_{j \to \infty} R(X_j) = \liminf_{j \to \infty} R(X_j) < \infty,$
and let $M$ be such that $R(X_j) + \varepsilon \leqslant M$ for all $j$. By definition of $R$, for each $X_j$ there exists $\nu_j \in \hpn$ with
$$ M \geqslant R(X_j) + \varepsilon \geqslant \langle F(\cdot + X_j), \nu_j \rangle \ldotp$$
Our growth assumption on $F$ and boundedness of $\left|X_j\right|$ (as a convergent sequence) give
$$ M \geqslant \int_{\RRnm} c \left| X + X_j \right|^p \dd \nu_j \geqslant C \left( \int_{\RRnm} \left|X\right|^p \dd \nu_j -1 \right),$$
which yields $ \sup_j \int_{\RRnm} \left|X\right|^p \dd \nu_j < \infty \ldotp$
We see that the family $\{ \nu_j \}$ is bounded in $\calE^*$, therefore we may extract a weakly*-convergent subsequence from it --- without loss of generality assume that the whole sequence converges, i.e., $\nu_j \weaksConv \nu_0$ in $\calE^*$. By Lemma \ref{lemmahpoclosed} we have $\nu_0 \in \hpn$. Moreover
$ \translation_{X_j} \ast \nu_j \weaksConv \translation_{X_0} \ast \nu_0 \ldotp$
Since $F$ is lower semicontinuous and bounded from below we have
\begin{equation*}
\begin{aligned}
\varepsilon + \liminf_{j \to \infty} R(X_j)\geqslant  & \liminf_{j \to \infty} \langle F, \translation_{X_j} \ast \nu_j \rangle \geqslant \langle F, \translation_{X_0} \ast \nu_0 \rangle  \\
= & \int_{\RRnm} F(\cdot + X_0) \dd \nu_0 \geqslant R(X_0),
\end{aligned}
\end{equation*}
where the last inequality comes from the definition of $R$ and the fact that $\nu \in \hpn$. Since $\varepsilon > 0$ was arbitrary we conclude that $R$ is in fact lower semicontinuous. 

It now remains to show that $R$ satisfies Jensen's inequality with respect to homogeneous oscillation $\WWap$-gradient Young measures. To that end fix $X_0 \in \RRnm$ and $\nu \in \mathbb{H}^p_{X_0}$. We wish to show that
$R(X_0) \leqslant \int_{\RRnm} R \dd \nu \ldotp$
Without loss of generality we may assume that $\int_{\RRnm} R \dd \nu < \infty$. Fix an $\varepsilon > 0$ and observe that, by definition of $R$, for all $X \in \RRnm$ there exists $\nu_{X} \in \hpn$ satisfying
$$ \langle F(\cdot + X), \nu_{X} \rangle \leqslant \varepsilon + R(X),$$
so that, for now only formally,
$$\int_{\RRnm} \left( \int_{\RRnm} F(\cdot + X) \dd \nu_{X} \right) \dd \nu(X) \leqslant \varepsilon + \int_{\RRnm} R \dd \nu \ldotp$$
Now --- if we manage to show that $\nu_{X}$ may be chosen in such a way that $X \mapsto \nu_{X}$ is weak* measurable and that the measure $\mu$ defined by duality as

\begin{equation}\label{eqMuDefDuality}
\langle g, \mu \rangle := \int_{\RRnm} \left( \int_{\RRnm} g(\cdot + X) \dd \nu_{X} \right) \dd \nu(X)
\end{equation}
is a homogeneous $\WWap$-gradient Young measure with mean $X_0$ then the claim will follow, as by definition $\langle F, \mu \rangle \geqslant R(X_0)$. 

Note that weak* measurability of $X \mapsto \nu_{X}$ only means Lebesgue measurability of $X \mapsto  \int_{\RRnm} g(\cdot + X) \dd \nu_{X}$, which need not be enough to integrate this function with respect to $\nu$. However, if we manage to get Borel measurability of the function in question then the construction is justified, as $\nu$ is a Radon (hence Borel) measure --- we will call such a map Borel weak* measurable. It is clear that if one makes sense of the integration on the right-hand side of \eqref{eqMuDefDuality} then it defines a linear functional on $C_0(\RRnm)$. Its boundedness follows from the fact that all $\nu_{X}, \nu$ are probability measures, thus showing that the functional is given by some finite Radon measure $\mu$. 

For the measurable selection part we define a multifunction $\mathcal{G}$ given by
$$ \mathcal{G}(X) := \left\{ \mu \in \hpn \colon \int_{\RRnm} F(\cdot + X) \dd \mu \leqslant \varepsilon + R(X) \right\} \ldotp$$
For the measurable selection result we intend to use (see Theorem \ref{thmKuratowski}) we need $\mathcal{G}$ to take values in $2^Y$ for some complete metric space $Y$. For that we define, for a given $M > 0$,
$$ \Omega_M := \{ X \in \RRnm \colon \left|X\right| < M, R(X) \leqslant M \} \ldotp $$
Observe that since we assumed $R$ to be integrable with respect to $\nu$, we have that $X \in \bigcup_{M = 1}^{\infty} \Omega_M$ for $\nu$-a.e. $X \in \RRnm$. Let us fix $M \in \naturals$. Then, for any $X \in \Omega_M$ and any $\mu \in \mathcal{G}(X)$, we have
$$ \int_{\RRnm} F(W + X) \dd \mu(W) \leqslant \varepsilon + R(X) \leqslant 2 \varepsilon + R(X) \leqslant M+2 \varepsilon \ldotp$$
The factor $2$ in front of $\varepsilon$ is not important here, we only put it there to allow for some room in later parts of the argument. Due to the growth assumption on $F$ we have
\begin{equation*}
\begin{aligned} 
\int_{\RRnm} F(W + X) \dd \mu (W) \geqslant & C \int_{\RRnm} \left|W + X\right|^p \dd \mu (W) \\
\geqslant  & C \int_{\RRnm} \left|W\right|^p \dd \mu (W) - C^{-1} \left|X\right|^p \ldotp
\end{aligned}
\end{equation*}
Finally $  \int_{\RRnm} \left|W\right|^p \dd \mu (W) \leqslant C_M,$ holds for all $\mu \in \mathcal{G}(X)$, with the constant $C_M$ depending only on $M$ (and $\varepsilon$). Therefore, we may consider our operator $\mathcal{G}$ as a map $\Omega_M \rightarrow 2^{Y_M}$, where
$$ Y_M := \left\{ \mu \in \hpn \colon \int_{\RRnm} \left|W\right|^p \dd \mu \leqslant C_M \right\} \ldotp$$
The set $Y_M$ may be equipped with the weak* topology inherited from $\calE^*$. Since we put a uniform bound on the $p$-th moments (so also on the norm in $\calE^*$), this topology is metrisable in a complete and separable manner when restricted to $Y_M$. To prove that, first recall that due to Lemma \ref{lemmahpoclosed} $\hpn$ is weak* closed in $\calE^*$. Since $| \cdot |^p \in \calE$ we know that the map $\mu \mapsto \intRRnm |W|^p \dd \mu$ is weak* continuous, thus $Y_M$ is weak* closed and bounded. The Banach-Alaoglu Theorem (see for example Theorem 3.16 in \cite{Brezis10}) then implies that $Y_M$ is weak* compact. Since $\calE$ is clearly separable we deduce that the weak* topology on $Y_M$ is metrisable (see Theorem 3.28 in \cite{Brezis10}). Finally, compact metric spaces are complete and separable, thus proving our claim. 
\begin{lemma}\label{lemmaMeasurability1}
For any $X \in \Omega_M$ the set $\mathcal{G}(X)$ is non-empty and closed.
\end{lemma}
\begin{proof}
The fact that $\mathcal{G}(X) \not= \emptyset$ comes straight from the definition of $R$. To show that it is closed it is enough to show that it is sequentially closed. Let us then fix a sequence $\{\mu_j\} \subset \mathcal{G}(X)$ and assume that it converges weak* in $\calE^*$ to some $\mu \in Y_M$. Since the function $F$ is lower semicontinuous and bounded from below we get by Lemma \ref{lemmaPortmanteau} that
$$ R(X) + \varepsilon \geqslant \liminf_{j \to \infty} \int_{\RRnm} F(\cdot + X) \dd \mu_j \geqslant  \int_{\RRnm} F(\cdot + X) \dd \mu,$$
so $\mu \in \mathcal{G}(X)$, which ends the proof.
\end{proof}

\begin{lemma}\label{lemmaMeasurability2}
For any non-empty closed set $O \subset Y_M$ the set 
$$\left\{ X \in \Omega_M \colon \mathcal{G}(X) \cap O \not= \emptyset \right\}$$ 
is Borel measurable. 
\end{lemma}
\begin{proof}
First note that we may rewrite the set in question as
$$\bigcap_{k=1}^{\infty} \left\{ X \in \Omega_M \colon \inf_{\mu \in O} \int_{\RRnm} F(\cdot + X) \dd \mu \leqslant R(X) + \varepsilon(1 + 2^{-k}) \right\}\ldotp$$
Hence, it is enough to show that the sets 
$$\left \{ X \in \Omega_M \colon \inf_{\mu \in O} \int_{\RRnm} F(\cdot + X) \dd \mu \leqslant R(X) + \varepsilon(1 + 2^{-k}) \right\}$$ 
are all Borel measurable. Define
$$ U(X) := \inf_{\mu \in O} \int_{\RRnm} F(\cdot + X) \dd \mu \ldotp $$
We claim that $U$ is lower semicontinuous. Let $X_j \rightarrow X$. We need to show that $ \liminf_{j \to \infty} U(X_j) \geqslant U(X) \ldotp$
Without loss of generality assume that the $\liminf$ is a true limit and that it is finite, i.e., 
$$\lim_{j \to \infty} U(X_j) = \liminf_{j \to \infty} U(X_j) < \infty.$$ 
By definition of $U$, for each $k$ there exists a measure $\mu_j \in O$ with
$$ \int_{\RRnm} F(\cdot + X_j) \dd \mu_j \leqslant U(X_j) + 1/k \ldotp$$
Therefore
$$ \lim_{j \to \infty} \int_{\RRnm} F(\cdot + X_j) \dd \mu_j = \lim_{j \to \infty} U(X_j) \ldotp$$
Since the set $O$ is a closed subset of a compact space $Y_M$ we may extract an $\calE^*$ weak* convergent subsequence from $\mu_j$. Without loss of generality assume that the entire sequence $\mu_j$ converges weak* to some $\mu \in O$. This, combined with $X_j \rightarrow X$, implies that we have $ \delta_{X_j} \ast \mu_j \weaksConv \delta_{X} \ast \mu$ in the sense of probability measures. Therefore, since $F$ is lower semicontinuous, the portmanteau theorem yields
\begin{equation*}
\begin{aligned}
\liminf_{j \to \infty} \int_{\RRnm} F(\cdot + X_j) \dd \mu_j   = & \liminf_{j \to \infty} \int_{\RRnm} F \dd \left( \delta_{X_j} \ast \mu_j \right)  \\
\geqslant  & \int_{\RRnm} F \dd \left( \delta_{X} \ast \mu \right) =  \int_{\RRnm} F(\cdot + X) \dd \mu \geqslant U(X),
\end{aligned}
\end{equation*}
which shows that $U$ is indeed lower semicontinuous. Since the set
$$ \left\{ X \in \Omega_M \colon \inf_{\mu \in O} \int_{\RRnm} F(\cdot + X) \dd \mu \leqslant R(X) + \varepsilon(1 + 2^{-k}) \right\} $$
is the same as
$$ \{ X \in \Omega_M \colon U(X) \leqslant R(X) + \varepsilon(1 + 2^{-k}) \},$$
and both $U$ and $R$ are lower semicontinuous (hence Borel measurable) the set in question is Borel measurable as well, which ends the proof.
\end{proof}

Now, thanks to Lemmas \ref{lemmaMeasurability1} and \ref{lemmaMeasurability2} we may use Theorem \ref{thmKuratowski} to deduce the existence of a weak* measurable map $\nu^M \colon \Omega_M \rightarrow \hpn$ such that for any $X \in \Omega_M$ the measure $\nu^M_{X}$ satisfies
$$ \int_{\RRnm} F(\cdot + X) d\nu^M_{X} \leqslant \varepsilon + R(X) \ldotp$$
Finally let us define the map $\widetilde{\nu} \colon \RRnm \rightarrow \hpn$ by
\begin{equation*} 
\widetilde{\nu}_X := \begin{cases}
\nu^M_{X} \text{ for } X \in \Omega_M \setminus \Omega_{M-1} \\
\widetilde{\mu} \text{ for } X \not\in \bigcup_{M=1}^{\infty} \Omega_M,
\end{cases}
\end{equation*}
where $\widetilde{\mu}$ is some arbitrary element of the (non-empty) set $\hpn$. Observe that the choice of $\widetilde{\mu}$ does not matter, as we have already observed that the set $\RRnm \setminus \bigcup_{M=1}^{\infty} \Omega_M$ is of $\nu$ measure $0$. This set is also Borel since we already know that $R$ is Borel measurable, hence each $\Omega_M$ is Borel. Clearly the map $\widetilde{\nu}$ is Borel weak* measurable, i.e., it is a measurable map from $\RRnm$ equipped with the Borel $\sigma$-algebra into $\hpn$ equipped with the weak* topology inherited from $\calE^*$. Therefore, we may define $\mu \in (C_0(\RRnm))^*$ as in \eqref{eqMuDefDuality}. It only remains to show that $\mu \in \hpn$. 

Positivity of $\mu$ results immediately from positivity of all $\nu_{X}$ and $\nu$. In the same way we show that $\mu$ is a probability measure, as
$$\langle 1, \mu \rangle =  \int_{\RRnm} \left( \int_{\RRnm} 1 \dd \nu_{X} \right) \dd \nu(X) =   \int_{\RRnm} 1 \dd \nu(X) = 1,$$
since all measures considered are probability measures. To prove that $\mu$ has a finite $p$-th moment we write
$$\langle \left|\cdot\right|^p, \mu \rangle = \int_{\RRnm} \left( \int_{\RRnm} \left|\cdot + X\right|^p \dd \nu_{X} \right) \dd \nu(X) \ldotp $$
Using the growth assumption on $F$ we get
$$ \int_{\RRnm} \left|\cdot + X\right|^p \dd \nu_{X} \leqslant C  \int_{\RRnm} F(\cdot + X) \dd \nu_{X} \leqslant C (R(X) + \varepsilon),$$
where the last inequality is satisfied for $\nu$-a.e. $X$. Integrating with respect to $\nu$ gives
$$\langle \left|\cdot\right|^p, \mu \rangle \leqslant C \left(\varepsilon + \int_{\RRnm} R(X) \dd \nu(X)\right) < \infty,$$
since, by assumption, $R$ is integrable with respect to $\nu$. Lastly, it remains to show that $\mu$ satisfies the inequality in Theorem \ref{thmCharacterizationOfYM}. Fix any continuous functions $g \colon \RRnm \to \RR$ with $|g(v)| \leqslant C(1 + |v|^p)$ for some constant $C$. We have
\begin{eqnarray*}
\langle \mu, g \rangle &=&  \int_{\RRnm} \left( \int_{\RRnm} g(\cdot + X) \dd \nu_{X} \right) \dd \nu(X) \\
&\geqslant&  \int_{\RRnm} \QQ g(X) \dd \nu(X) \geqslant \QQ (\QQ g) (X_0) = \QQ g(X_0),
\end{eqnarray*} 
where the first inequality comes from the fact that all $\nu_{X}$'s are Young measures with mean $0$, the second one from the respective property of $\nu$, and the last equality from Lemma \ref{lemmaQgIsQC}.     
This shows that we indeed have $\mu \in \mathbb{H}^p_{X_0}$ and ends the proof, as discussed earlier (see eq. \eqref{eqMuDefDuality}).
\end{proof}

\subsection{Relaxation in the extended real-valued setting}\label{sectionRelaxationExtended}
We begin by defining a relaxed notion of convergence for vector fields that are nearly (up to an $\LLp$-small error) $\ba$-gradients of functions in $\WWap$. The notion is reminiscent of the one often used in the $\mathcal{A}$-free setting, where instead of working with sequences that satisfy the constraint exactly, i.e., with $\mathcal{A} V_j = 0$, one only requires $\mathcal{A} V_j \to 0$ strongly in $\WW^{-1,p}$, see for example \cite{FonsecaMuller99}. In the case of standard first order gradients this corresponds to the condition $\mathrm{curl} \, V_j \to 0$ strongly in $\WW^{-1,p}(\Omega)$ investigated in \cite{Kristensen15}. 
\begin{defi}
We say that a sequence of vector fields $V_j \in \LLp(\Omega; \RRnm)$ is a sequence of approximate $\WWap$ gradients if there exist sequences $u_j \in \WWap(\Omega; \RRn)$ and $v_j \in \LLp(\Omega; \RRnm)$ such that 
$$ V_j = \ggrad u_j + v_j$$
and $v_j \to 0$ strongly in $\LLp$. 
\end{defi}

Following \cite{KristensenNotes} we introduce the following notion of convergence:
\begin{defi}
We say that a sequence of vector fields $V_j \in \LLp(\Omega; \RRnm)$ converges to $V$ in the sense of approximate $\WWap$ gradients if $V_j $ converges to $V$ weakly in $\LLp$ and $(V_j - V)$ is a sequence of approximate $\WWap$ gradients. In such a case we write $V_j \ApgradConv V$.
\end{defi}

Proposition \ref{propTranslatedYM} immediately implies the following:
\begin{lemma}\label{lemmaApproximateMeasures}
Assume that $\Omega$ satisfies the weak $\ba$-horn condition. Suppose that a sequence $V_j = \ggrad u_j + v_j \in \LLp(\Omega; \RRnm)$ of approximate $\WWap$ gradients converges weakly to $0$ in $\LLp$ and generates an oscillation Young measure $\nu$. Then $\{\ggrad u_j \}$ generates the same Young measure $\nu$. In particular, any oscillation Young measure generated by a sequence of approximate $\WWap$ gradients is an oscillation $\WWap$-gradient Young measure.
\end{lemma}

\begin{corollary}\label{corClosedQCSufficient}
Let $\Omega$ be a bounded open Lipschitz domain satisfying the weak $\ba$-horn condition. Suppose that $F \colon \RRnm \to (-\infty, \infty]$ is bounded from below and closed $\WWap$-quasiconvex. Then the functional
$$ \II (V) := \int_\Omega F(V) \dd x$$
is sequentially lower semicontinuous with respect to approximate $\WWap$ gradient convergence.
\end{corollary}
\begin{proof}
Since, by Lemma \ref{lemmaApproximateMeasures} the Young measures generated by sequences of approximate $\WWap$ gradients are $\WWap$-gradient Young measures, the argument of Lemma \ref{lemmaClosedQCImpliesLSC} carries through unchanged.
\end{proof}

The main result here is the following:
\begin{theorem}\label{thmContinuousImpliesRelaxation}
If $F \colon \RRnm\to (-\infty, \infty]$ is a continuous integrand satisfying $F(X) \geqslant C|X|^p - C^{-1}$ for some $C > 0$ then the sequentially (with respect to approximate $\WWap$ gradient convergence) weakly lower semicontinuous envelope of the functional $\II_F$ is given by
$$\overline{\II}_F [V] := \inf_{V_j \ApgradConv V} \left\{ \liminf_{j \to \infty} \II_F [V_j] \right\}= \int_{\Omega} \overline{F}(V(x)) \dd x,$$
where the infimum is taken over all sequences $V_j$ converging to $V$ in the sense of approximate $\WWap$ gradient convergence. As before, $\overline{F}$ denotes the closed $\WWap$-quasiconvex envelope of $F$. 
\end{theorem}
\begin{proof}
Corollary \ref{corClosedQCSufficient} guarantees that $\overline{\II}_F [V] \geqslant \int_{\Omega} \overline{F}(V(x)) \dd x$, thus we only need to prove the opposite inequality. If $F$ is identically equal $+ \infty$ then there is nothing to show, so we may restrict to proper integrands.
Using a translation we may, without loss of generality, assume $F(X) \geqslant C|X|^p$. In any case, the fact that $F$ is bounded from below immediately implies the same for $\overline{F}$. Fix any $V \in \LL^p$. Without loss of generality we may assume  $\int_{\Omega} \overline{F}(V(x)) \dd x < \infty$, as otherwise there is nothing to prove. Fix an $\varepsilon > 0$ and observe that we must have $ \overline{F}(V(x)) < \infty \text{ a.e. in } \Omega \ldotp$
Therefore, using Proposition \ref{lemmaQCenvelope}, we may find a family of homogeneous oscillation $\WWap$-gradient Young measures $\{\nu_x\}_{x \in \Omega}$ with mean $0$ and such that, for almost every $x \in \Omega$, we have
\begin{equation}\label{eqNuxFbar} 
\overline{F}(V(x)) + \varepsilon \geqslant \intRRnm F(\cdot + V(x)) \dd \nu_x \ldotp
\end{equation}
Using exactly the same argument as in the proof of Proposition \ref{lemmaQCenvelope} we may ensure weak* measurability of $x \rightarrow \nu_x$. We intend to show that $\nu$ is a suitable Young measure using Theorem \ref{thmCharacterisationYMnonhom}. Recall that we need to prove the following:

i) there exists $v \in \WWap(\Omega)$ such that
$$ \ggrad v(x) = \langle \nu_x, \id \rangle \text{ for a.e. } x \in \Omega;$$

ii) 
$$ \int_{\Omega} \intRRnm |W|^p \dd \nu_x(W) \dd x < \infty;$$

iii) for a.e. $x \in \Omega$ and all continuous functions $g \colon \RRnm \to \RR$ satisfying $|g(W)| \leqslant C(1 + |W|^p)$ for some positive constant $C$ one has
$$ \langle \nu_x, g \rangle \geqslant \QQ g( \langle \nu_x, \id \rangle) \ldotp$$

The first point is clearly satisfied, as all our measures are of mean $0$. The second one may be checked in the same way as in the already mentioned proof of Proposition \ref{lemmaQCenvelope}, using the growth assumption on $F$. Finally, the third point results immediately from the fact that all $\nu_x$'s are, by definition, elements of $\hpn$, so we may use Theorem \ref{thmCharacterizationOfYM}. 

This shows that $\nu$ is indeed generated by some $p$-equiintegrable family $\{\ggrad w_j\}$ with $w_j \in \WWap(\Omega)$ and $w_j \weakConv 0$ in $\WWap$. For a given $M \in \naturals$ consider 
$$F^M(z) := \min(F(z), M(|z|^p + 1)).$$
Clearly, for each $M$, the function $F^M$ is continuous and the family $\{F^M(V + \ggrad w_j)\}_j$ is $p$-equiintegrable, due to the same property of $\{V + \ggrad w_j\}$. Theorem \ref{thmFToYM} then yields
$$ \int_{\Omega} F^M(V + \ggrad w_j) \dd x \rightarrow \int_{\Omega} \left( \intRRnm F^M(V(x) + \cdot) \dd \nu_x \right) \dd x \ldotp$$
On the other hand, since $F^M \leqslant F$ and $\nu_x$ are non-negative and satisfy \eqref{eqNuxFbar}, we have
\begin{eqnarray*}
\int_{\Omega} \left( \intRRnm F^M(V(x) + \cdot) \dd \nu_x \right) \dd x &\leqslant& \int_{\Omega} \left( \intRRnm F(V(x) + \cdot) \dd \nu_x \right) \dd x\\ 
&\leqslant& \int_{\Omega} \overline{F}(V(x)) \dd x + \varepsilon \ldotp
\end{eqnarray*}
From this we deduce, through a diagonal extraction, that there exists a sequence $j(M) \in \naturals$ with $\lim_{M \rightarrow \infty} j(M) = \infty$ such that for all $M$ one has
\begin{equation}\label{eqEstimateForFM} 
\int_{\Omega} F^M(V + \ggrad w_{j(M)}) \dd x \leqslant \int_{\Omega} \overline{F}(V(x)) \dd x + 2 \varepsilon \ldotp
\end{equation}
Define the set
$$ \gm := \left\{ x \in \Omega \colon F(V(x) + \ggrad w_{j(M)}(x)) \leqslant M(|V(x) + \ggrad w_{j(M)}(x)|^p + 1) \right\},$$
and fix some $X_0 \in \RRnm$ for which $F(X_0) < \infty$ --- such a point exists, as $F$ is proper. Next define a vector field $W_M$ in such a way that
\begin{equation}\label{eqDefWnTilde} 
V(x) + W_M(x) = (V(x) + \ggrad w_{j(M)}(x)) \indyk_{\gm} + X_0 \indyk_{\gm^c} \ldotp
\end{equation}
We claim that $\{V + W_M\}_M$ is an admissible vector field in the $\overline{\II}_F[V]$ problem.
For that it is enough to show that $\| V + W_M - (V + \ggrad w_{j(M)}) \|_{\LL^p(\Omega)} \rightarrow 0$. 
By definition we have 
\begin{equation*}
\begin{aligned}
\| V + W_M - (V + \ggrad w_{j(M)})  \|_{\LL^p(\Omega)} & =  \| V + W_M - (V + \ggrad w_{j(M)}) \|_{\LL^p(\gm^c)}  \\
& \leqslant   \|X_0\|_{\LL^p(\gm^c)} + M^{-1} \left(\intOmega F^M(V + \ggrad w_{j(M)}) \dd x \right)^{1/p},
\end{aligned}
\end{equation*}
where the last inequality comes from the definition of the set $\gm^c$ and extending the integral to all of $\Omega$. Now, \eqref{eqEstimateForFM} yields 
$$  M^{-1} \left(\intOmega F^M(V + \ggrad w_{j(M)}) \dd x \right)^{1/p} \leq M^{-1} \left( \int_{\Omega} \overline{F}(V(x)) \dd x + 2 \varepsilon \right)^{1/p},$$
thus showing the desired convergence to $0$ in $\LL^p$, as $\|X_0\|_{\LL^p(\gm^c)} \to 0$ results simply from the fact that clearly the Lebesgue measure of $\gm^c$ tends to $0$, because 
$$F(V(x) + \ggrad w_{j(M)}(x)) > M \text{ on } \gm^c, $$ 
and we have a uniform (with respect to $M$) bound on the integral of the function in question.
This implies in particular that $V + W_M$ converges to $V$ in the sense of approximate $\WWap$ gradients convergence. Therefore, we have
\begin{eqnarray*}
\overline{\II}_F[V] &\leqslant& \liminf_{M \rightarrow \infty} \int_{\Omega} F(V + W_M) \dd x\\ 
&=& \liminf_{M \rightarrow \infty} \int_{\gm} F^M(V + w_{j(M)}) \dd x + \int_{\gm^c} F(X_0) \dd x\\
&\leqslant& \liminf_{M \rightarrow \infty} \int_{\Omega} \overline{F}(V(x)) \dd x + 2 \varepsilon = \int_{\Omega} \overline{F}(V(x)) \dd x + 2 \varepsilon,
\end{eqnarray*} 
where the last inequality results from \eqref{eqEstimateForFM} and the measure of $\gm^c$ tending to $0$. Since $\varepsilon > 0$ was arbitrary the proof is complete.
\end{proof}


\begin{thebibliography}{99}



\bibitem{AcerbiFusco84} {\sc E. Acerbi, N. Fusco}, Semicontinuity problems in the calculus of variations, {\it Archive for Rational Mechanics and Analysis}, 86.2 (1984), 125-145.


\bibitem{AlibertDacorogna92} {\sc J.-J. Alibert, B. Dacorogna}, An example of a quasiconvex function that is not polyconvex in two dimensions, {\it Archive for rational mechanics and analysis}, 117.2 (1992), 155-166.


\bibitem{ArroyoPhilippisRindler17} {\sc A. Arroyo-Rabasa, G. De Philippis, F. Rindler}, Lower semicontinuity and relaxation of linear-growth integral functionals under PDE constraints, {\it Advances in calculus of variations}, 13.3 (2020), 219-255.



\bibitem{Balder84} {\sc E. J. Balder}, A general approach to lower semicontinuity and lower closure in optimal control theory, {\it SIAM journal on control and optimization}, 22.4 (1984), 570-598.

\bibitem{Ball78} {\sc J. M. Ball}, Convexity conditions and existence theorems in nonlinear elasticity, {\it  Archive for rational mechanics and Analysis}, 63.4 (1976), 337-403.


\bibitem{Ball89} {\sc J. M. Ball}, A version of the fundamental theorem for Young measures in {\it PDEs and continuum models of phase transitions,  Springer, Berlin, Heidelberg}, (1989), 207-215.


\bibitem{BallCurrieOlver81} {\sc J. M. Ball, J. C. Currie, P. J. Olver}, Null Lagrangians, weak continuity, and variational problems of arbitrary order {\it Journal of Functional Analysis}, 41.2 (1981), 135-174.

\bibitem{BallKircheimKristensen00} {\sc J. M. Ball, B. Kirchheim, J. Kristensen}, Regularity of quasiconvex envelopes {\it Calculus of Variations and Partial Differential Equations}, 11.4 (2000), 333-359.

\bibitem{BallMurat84} {\sc J. M. Ball, F. Murat}, $\W1p$-quasiconvexity and variational problems for multiple integrals, {\it Journal of Functional Analysis}, 58.3 (1984), 225-253.

\bibitem{BallZhang90} {\sc J. M. Ball, K. Zhang}, Lower semicontinuity of multiple integrals and the biting lemma, {\it Proceedings of the Royal Society of Edinburgh Section A: Mathematics}, 114.3-4 (1990), 367-379.


\bibitem{BenedettoCzaja} {\sc J. J. Benedetto, W. Czaja}, Integration and modern analysis {\it Springer Science \& Business Media}, 2010.


\bibitem{BerliocchiLasry73} {\sc H. Berliocchi, J.-M. Lasry}, Int\'{e}grandes normales et mesures param\'{e}tr\'{e}es en calcul des variations {\it Bulletin de la Soci\'{e}t\'{e} Math\'{e}matique de France}, 101 (1973), 129-184.



\bibitem{Besov67} {\sc O. V. Besov}, Concerning the theory of embedding and continuing classes of differentiable functions, {\it Mathematical Notes}, 1.2 (1967), 156-161.

\bibitem{Besov74} {\sc O. V. Besov}, Growth of a mixed derivative of a function of $ C^{(l_1, l_2)}$, {\it Mathematical notes of the Academy of Sciences of the USSR}, 15.3 (1974), 201-206.


\bibitem{BesovIlin68} {\sc O. V. Besov, V. P. Il'in}, Natural extension of the class of regions in embedding theorems {\it Sbornik: Mathematics}, 4.4 (1968), 445-456.

\bibitem{BesovIlinNikolskii78} {\sc O. V. Besov, V. P. Il'in, S. M. Nikolskii}, Integral representations of functions and imbedding theorems Vol. 1 {\it Winston \& sons}, 1978.

\bibitem{BesovIlinNikolskii78p2} {\sc O. V. Besov, V. P. Il'in, S. M. Nikolskii}, Integral representations of functions and imbedding theorems Vol. 2 {\it Winston \& sons}, 1978.



\bibitem{Boman72} {\sc J. Boman}, Supremum norm estimates for partial derivatives of functions of several real variables, {\it Illinois Journal of Mathematics}, 16.2 (1972), 203-216.

\bibitem{BraidesFonsecaLeoni00} {\sc A. Braides, I. Fonseca, G. Leoni}, A-quasiconvexity: relaxation and homogenization, {\it ESAIM: Control, Optimisation and Calculus of Variations}, 5 (2000), 539-577.


\bibitem{Brezis10} {\sc H. Brezis}, Functional analysis, Sobolev spaces and partial differential equations, {\it Springer Science \& Business Media}, 2010.




\bibitem{Burenkov66} {\sc V. I. Burenkov}, Imbedding and extension theorems for classes of differentiable functions of several variables defined on the entire spaces, {\it Itogi Nauki i Tekhniki. Seriya Matematicheskii Analiz}, 3 (1966), 71-155.


\bibitem{BurenkovFain76} {\sc V. I. Burenkov, B. L. Fain}, On the extension of functions from anisotropic spaces with preservation of class in Doklady Akademii Nauk 228.3 {\it Russian Academy of Sciences}, 1976.

\bibitem{Buttazzo89} {\sc G. Buttazzo}, Semicontinuity, relaxation and integral representation in the calculus of variations, {\it Longman}, 1989.

\bibitem{Cagnetti11} {\sc F. Cagnetti}, k-quasi-convexity reduces to quasi-convexity {\it Proceedings of the Royal Society of Edinburgh Section A: Mathematics}, 141.4 (2011), 673-708.

\bibitem{CalderonTorchinsky} {\sc A. P. Calder\'{o}n, A. Torchinsky}, Parabolic maximal functions associated with a distribution, {\it Advances in Mathematics}, 16.1 (1975), 1-64.



\bibitem{ChenKristensen17} {\sc C. Y. Chen, J. Kristensen}, On coercive variational integrals {\it Nonlinear Analysis: Theory, Methods \& Applications}, 153 (2017), 213-229.



\bibitem{Dacorogna82} {\sc B. Dacorogna}, Quasiconvexity and relaxation of nonconvex problems in the calculus of variations, {\it Journal of Functional Analysis}, 46.1 (1982), 102-118.

\bibitem{Dacorogna07} {\sc B. Dacorogna}, Direct methods in the calculus of variations, {\it Springer Science \& Business Media}, 2007.


\bibitem{DacorognaMarcellini88} {\sc B. Dacorogna, P. Marcellini}, A counterexample in the vectorial calculus of variations, {\it Material instabilities in continuum mechanics}, (1988), 77-83.

\bibitem{DacorognaMarcellini99} {\sc B. Dacorogna, P. Marcellini}, Implicit Partial Differential Equations, {\it Birkh\"{a}user Basel}, 1999.



\bibitem{DalMasoFonsecaLeoniMorini04} {\sc G. Dal Maso, I. Fonseca, G. Leoni, M. Morini}, Higher-order quasiconvexity reduces to quasiconvexity, {\it Archive for rational mechanics and analysis}, 171.1 (2004), 55-81.





\bibitem{DellIsolaLekszyckiPawlikowskiGrygorukGreco15} {\sc F. dell'Isola, T. Lekszycki, M. Pawlikowski, R. Grygoruk, L. Greco}, Designing a light fabric metamaterial being highly macroscopically tough under directional extension: first experimental evidence, {\it Zeitschrift f\"{u}r angewandte Mathematik und Physik}, 66.6 (2015), 3473-3498.


\bibitem{DemidenkoUpsenskii03} {\sc G. V. Demidenko, S. V. Upsenskii}, Partial differential equations and systems not solvable with respect to the highest-order derivative {\it CRC Press}, 2003.




\bibitem{DiPernaMajda87} {\sc R. J. DiPerna, A. J. Majda}, Oscillations and concentrations in weak solutions of the incompressible fluid equations, {\it Communications in Mathematical Physics}, 108.4 (1987), 667-689.


\bibitem{DupontScott80} {\sc T. Dupont, R. Scott}, Polynomial approximation of functions in Sobolev spaces, {\it Mathematics of Computation}, 34.150 (1980), 441-463.




\bibitem{EremeyevDellIsolaBoutinSteigmann18} {\sc V. A. Eremeyev, F. dell'Isola, C. Boutin, D. Steigmann}, Linear pantographic sheets: existence and uniqueness of weak solutions, {\it Journal of Elasticity}, 132.2 (2018), 175-196.


\bibitem{Evans86} {\sc L. C. Evans}, Quasiconvexity and partial regularity in the calculus of variations, {\it Archive for rational mechanics and analysis}, 95.3 (1986), 227-252.

\bibitem{FonsecaLeoni07} {\sc I. Fonseca, G. Leoni}, Modern Methods in the Calculus of Variations: $\LLp$ Spaces {\it Springer Science \& Business Media}, 2007.

\bibitem{FonsecaLeoniMuller04} {\sc I. Fonseca, G. Leoni, S. M\"{u}ller}, A-quasiconvexity: weak-star convergence and the gap, {\it Annales de l'IHP, Analyse non lin\'{e}aire}, 21.2 (2004), 209-236.
 

\bibitem{FonsecaMuller99} {\sc I. Fonseca, S. M\"{u}ller}, A-quasiconvexity, lower semicontinuity, and Young measures, {\it SIAM Journal on Mathematical Analysis}, 30.6 (1999), 1355-1390.





\bibitem{Giaquinta83} {\sc M. Giaquinta}, Multiple integrals in the calculus of variations and nonlinear elliptic systems {\it Princeton University Press}, 1983.


\bibitem{Giusti03} {\sc E. Giusti}, Direct methods in the calculus of variations {\it World Scientific}, 2003.

\bibitem{GmeinederKristensen18} {\sc F. Gmeineder, J. Kristensen}, Partial Regularity for BV Minimizers, {\it Archive for Rational Mechanics and Analysis}, 232.3 (2019), 1429-1473.




\bibitem{Ilin68} {\sc V. P. Il'in}, Conditions of validity of inequalities between $\LLp$-norms of partial derivatives of functions of several variables, {\it Trudy Matematicheskogo Instituta imeni VA Steklova}, 96 (1968), 205-242.




\bibitem{KazanieckiStolyarovWojciechowski17} {\sc K. Kazaniecki, D. M. Stolyarov, M. Wojciechowski}, Anisotropic Ornstein noninequalities, {\it Analysis \& PDE}, 10.2 (2017), 351-366.

\bibitem{KinderlehrerPedregal91} {\sc D. Kinderlehrer, P. Pedregal}, Characterizations of Young measures generated by gradients, {\it Archive for rational mechanics and analysis}, 115.4 (1991), 329-365.

\bibitem{KinderlehrerPedregal94} {\sc D. Kinderlehrer, P. Pedregal}, Gradient Young measures generated by sequences in Sobolev spaces, {\it Journal of Geometric Analysis}, 4.1 (1994), 59-90.

\bibitem{KirchheimKristensen16} {\sc B. Kirchheim, J. Kristensen}, On rank one convex functions that are homogeneous of degree one, {\it Archive for rational mechanics and analysis}, 221.1 (2016), 527-558.

\bibitem{Kolyada07} {\sc V. I. Kolyada}, On embedding theorems, {\it Nonlinear Analysis, Function Spaces and Applications}, (2007), 35-94.


\bibitem{KolyadaPerez04} {\sc V. I. Kolyada, F. J. P\'{e}rez}, Estimates of difference norms for functions in anisotropic Sobolev spaces, {\it Mathematische Nachrichten}, 267.1 (2004), 46-64.



\bibitem{Kristensen99one} {\sc J. Kristensen}, Lower semicontinuity in spaces of weakly differentiable functions, {\it Mathematische Annalen}, 313.4 (1999), 653-710.

\bibitem{Kristensen99two} {\sc J. Kristensen}, On the non-locality of quasiconvexity, {\it Annales de l'Institut Henri Poincare (C) Non Linear Analysis, Elsevier Masson}, 16.1 (1999), 1-13.


\bibitem{Kristensen15} {\sc J. Kristensen}, A necessary and sufficient condition for lower semicontinuity, {\it Nonlinear Analysis: Theory, Methods \& Applications}, 120 (2015), 43-56.

\bibitem{KristensenNotes} {\sc J. Kristensen}, Nonlinear analysis \& applications, Lecture notes for a course given at the University of Oxford, 2015.

\bibitem{KristensenProsinski21} {\sc  J. Kristensen, A. Prosinski}, Regularity of minimisers of variational problems in the mixed smoothness setting, ({\it in preparation}).
\bibitem{KuratowskiRyllNardzewski65} {\sc K. Kuratowski, C. Ryll-Nardzewski} A general theorem on selectors, {\it Bulletin de l'Acad\'{e}mie polonaise des sciences. S\'{e}rie des sciences math\'{e}matiques, astronomiques, et physiques}, 13.1 (1965), 397-403.




\bibitem{Marcellini85} {\sc P. Marcellini}, Approximation of quasiconvex functions, and lower semicontinuity of multiple integrals, {\it Manuscripta Mathematica}, 51.1 (1985), 1-28. 


\bibitem{McShane40} {\sc E. J. McShane}, Generalized curves, {\it Duke Mathematical Journal}, 6.3 (1940), 513-536.


\bibitem{MejlbroTopsoe77} {\sc L. Mejlbro, F. Tops{\o}e}, A precise Vitali theorem for Lebesgue measure, {\it  Mathematische Annalen}, 230.2 (1977), 183-193.

\bibitem{Meyers65} {\sc N. G. Meyers}, Quasi-convexity and lower semi-continuity of multiple variational integrals of any order, {\it Transactions of the American Mathematical Society}, 119.1 (1965), 125-149.


\bibitem{Morrey52} {\sc C. B. Morrey}, Quasiconvexity and the lower semicontinuity of multiple integrals, {\it Pacific journal of mathematics}, 2.1 (1952), 25-53.



\bibitem{Muller99} {\sc S. M\"{u}ller}, Rank-one convexity implies quasiconvexity on diagonal matrices, {\it International Mathematics Research Notices}, 20 (1999), 1087-1095.



\bibitem{Murat78} {\sc F. Murat}, Compacit\'{e} par compensation, {\it Annali della Scuola Normale Superiore di Pisa-Classe di Scienze}, 5.3 (1978), 489-507.

\bibitem{Murat81} {\sc F. Murat}, Compacit\'{e} par compensation: condition n\'{e}cessaire et suffisante de continuit\'{e} faible sous une hypothese de rang constant, {\it Annali della Scuola Normale Superiore di Pisa-Classe di Scienze}, 8.1 (1981), 69-102.



\bibitem{Nikolskii51} {\sc S. M. Nikol'skii}, Inequalities for entire functions of finite degree and their application in the theory of differentiable functions of several variables, {\it Trudy Matematicheskogo Instituta imeni VA Steklova}, 38 (1951), 244-278.


\bibitem{Pedregal94} {\sc P. Pedregal}, Jensen's inequality in the calculus of variations, {\it Differential Integral Equations}, 7.1 (1994) 57-72.

\bibitem{Pedregal97} {\sc P. Pedregal}, Parametrized measures and variational principles, {\it Birkh\"{a}user}, (1997).

\bibitem{Pelczynski89} {\sc A. Pe\l{}czy\'{n}ski}, Boundedness of the canonical projection for Sobolev spaces generated by finite families of linear differential operators, Analysis at Urbana, vol. I, London Mathematical Society Lecture Note Series 137 {\it Cambridge University Press}, 1989, 395-415.

\bibitem{PelczynskiSenator86} {\sc A. Pe\l{}czy\'{n}ski, K. Senator}, On isomorphisms of anisotropic Sobolev spaces with classical Banach spaces and a Sobolev type embedding theorem, {\it Studia Mathematica}, 84.2 (1986), 169-215.

\bibitem{PelczynskiSenator862} {\sc A. Pe\l{}czy\'{n}ski, K. Senator}, Addendum to the paper 'On isomorphisms of anisotropic Sobolev spaces with classical Banach spaces and a Sobolev type embedding theorem', {\it Studia Mathematica}, 84.2 (1986), 217-218.

\bibitem{Prosinski18} {\sc A. Prosinski}, Closed $\mathcal{A}$-$p$ Quasiconvexity and Variational Problems with Extended Real-Valued Integrands, {\it ESAIM: Control, Optimisation and Calculus of Variations}, 24.4 (2018), 1605-1624.

\bibitem{ProsinskiThesis} {\sc A. Prosinski}, Calculus of variations in the mixed smoothness setting, {\it Doctoral dissertation, University of Oxford}, (2019).

\bibitem{ProsinskiRaita19} {\sc A. Prosinski, B. Rai\c{t}\u{a}}, On the well-posedness of some variational problems, ({\it in preparation}).

\bibitem{Raita19} {\sc B. Rai\c{t}\u{a}}, Potentials for A-quasiconvexity, ({\it Calculus of Variations and Partial Differential Equations}), 58.3 (2019), 105.



\bibitem{RindlerBook} {\sc F. Rindler}, Calculus of Variations {\it Springer International Publishing}, 2018.


\bibitem{Saks37} {\sc S. Saks}, Theory of the Integral {\it Hafner Publishing Company}, 1937.



\bibitem{Slobodeckii58one} {\sc L. N. Slobodeckii}, Generalized Sobolev spaces and their application to boundary problems for partial differential equations, {\it Leningradskii Gosudarstvennyi Pedagogiceskii Institut imeni A. I. Gercena. Ucenye Zapiski}, 197 (1958), 54-112.


\bibitem{Slobodeckii58two} {\sc L. N. Slobodeckii}, S. L. Sobolev's spaces of fractional order and their application to boundary problems for partial differential equations, {\it Doklady Akademii Nauk SSSR}, 118 (1958), 243-246.

\bibitem{Solonnikov75} {\sc V. A. Solonnikov}, Inequalities for functions of the classes $\overrightarrow{\WW}_p(\RR^n)$, {\it Journal of Mathematical Sciences}, 3.4 (1975), 549-564.




\bibitem{Sverak92} {\sc V. \v{S}ver\'{a}k}, Rank-one convexity does not imply quasiconvexity, {\it Proceedings of the Royal Society of Edinburgh Section A: Mathematics}, 120.1 (1992), 185-189.




\bibitem{Tartar79} {\sc L. Tartar}, Compensated compactness and applications to partial differential equations, {\it Nonlinear analysis and mechanics, Heriot-Watt symposium, Pitman}, 4 (1979), 136-211.



\bibitem{Tartar83} {\sc L. Tartar}, The compensated compactness method applied to systems of conservation laws in Systems of nonlinear partial differential equations, {\it Springer, Dordrecht}, 1983, 263-285.



\bibitem{Tartar92} {\sc L. Tartar}, On mathematical tools for studying partial differential equations of continuum physics: H-measures and Young measures in {\it Developments in partial differential equations and applications to mathematical physics, Springer, Boston, MA}, (1992), 201-217.




\bibitem{TurcoGiorgioMisraDellIsola17} {\sc E. Turco, I. Giorgio, A. Misra, F. dell'Isola}, King post truss as a motif for internal structure of (meta) material with controlled elastic properties, {\it Royal Society open science}, 4.10 (2017), 171153.




\bibitem{Young37} {\sc L. C. Young}, Generalized curves and the existence of an attained absolute minimum in the calculus of variations, {\it Comptes Rendus de la Soci\'{e}t\'{e} des Sciences et des Lettres de Varsovie}, 30 (1937), 212-234.


\bibitem{Young42one} {\sc L. C. Young}, Generalized surfaces in the calculus of variations, {\it Annals of mathematics}, (1942), 84-103.

\bibitem{Young42two} {\sc L. C. Young}, Generalized surfaces in the calculus of variations II, {\it Annals of mathematics}, (1942), 530-544.

\bibitem{Young00} {\sc L. C. Young}, Lectures on the calculus of variations and optimal control theory {\it American Mathematical Society}, 2000.






















\end{thebibliography}
\end{document}